\documentclass[12pt]{amsart}

\usepackage{pstricks}

\usepackage{caption}

\newif\ifdebug
\debugfalse

\newif \iffig
\figtrue
\usepackage{subfigure}

\usepackage{tikz}
\usepackage{tikz-cd}

\pgfdeclarelayer{edgelayer}
\pgfdeclarelayer{nodelayer}
\pgfsetlayers{edgelayer,nodelayer,main}

\usepackage{sfilip_package_settings}
\usepackage{sfilip_abbreviations}
\usepackage{sfilip_thm_style_long}

\DeclareMathOperator{\Hess}{{Hess}}

\begin{document}
%
\title[Tropical Dynamics of Area-Preserving Maps]{Tropical Dynamics of Area-Preserving Maps}
%
\thanks{Revised \textsc{\today} }

\date{March 2018\\
\textit{Dedicated to the memory of Bill Veech}}


\author{
	Simion Filip
}
\address{
	\parbox{0.6\textwidth}{
		School of Mathematics\\
		Institute for Advanced Study\\
		1 Einstein Drive,
		Princeton, NJ 08540}
}
\email{{sfilip@math.ias.edu}}
%
%
\begin{abstract}
We consider a class of area-preserving, piecewise affine maps on the $2$-sphere.
These maps encode degenerating families of K3 surface automorphisms and are profitably studied using techniques from tropical and Berkovich geometries.
\end{abstract}

%
\maketitle
%
\noindent\hrulefill
\tableofcontents
\noindent \hrulefill
\ifdebug
   \listoffixmes
\fi

\section{Introduction}
	\label{sec:introduction}

To study a family of objects that depend on a small parameter, it is convenient to first understand the behavior of the smallest-order terms.
Tropical geometry is a technique for doing this systematically for families of algebraic varieties and our goal in this paper is to apply some of these techniques in the context of dynamical systems.

A more sophisticated way to approach this question is via Berkovich geometry, as it encodes more of the information available in an algebraic family.
The application of Berkovich geometry in dynamics is a well-established area, with many contributions, a partial list of authors including Baker, DeMarco, Favre, Jonsson, Rivera--Letelier, Rumely \cite{BakerDeMarco,Baker_IntroductionBerkovich,Favre_RiveraLetelier,FavreJonsson,BakerRumely_PotentialTheory} and many more.
Perhaps one of the first direct applications to dynamics was by Einsiedler, Kapranov, and Lind \cite{EinsiedlerKapranovLind} though it is of a different flavor than the topic of this paper.

We are concerned with area-preserving maps on compact surfaces.
In the algebraic setting, this immediately restricts to abelian or K3 surfaces, and we further restrict here to K3 surfaces as dynamics on abelian surfaces is better-understood in the context of linear actions on tori.
In the complex-analytic setting this study was initiated by Cantat \cite{Cantat_DynamiqueK3}, who constructed a measure of maximal entropy and certain stable and unstable currents, corresponding to the equidistribution of stable and unstable leaves of the dynamics.
Examples with non-trivial Fatou components, specifically Siegel disks, were first constructed by McMullen \cite{McMullen_K3}.

We consider the same situation in the tropical setting.
That such a situation should be considered was suggested first by Kontsevich--Soibelman \cite[\S6.8]{KontsevichSoibelman_AffineStructures}.
Concretely, a tropical K3 surface is a sphere, realized as the boundary of a convex polyhedron cut out by rather special linear planes.
The tropical automorphisms preserve not only the area but also a piecewise integral-affine structure.
See \autoref{fig:pictures_computer_simulations_random_orbits1} for some illustrations of the induced dynamics.
One sees familiar pictures of elliptic islands and stochastic seas.
Note that because the derivatives of the transformations are in $\GL_2(\bZ)$, the structure near an elliptic fixed point is considerably simpler: the map is always finite order\footnote{In fact, the possible orders are $1,2,3,4$ or $6$.} in a whole neighborhood with polygonal boundary.

\begin{figure}[h]
	\centering
	\includegraphics[width=0.4\linewidth]{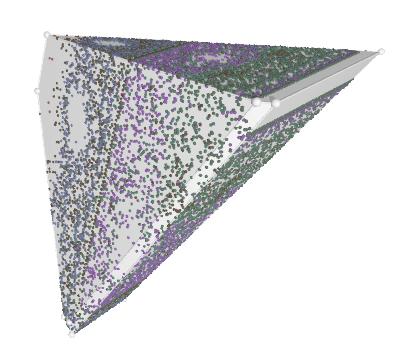}
	\includegraphics[width=0.2\linewidth]{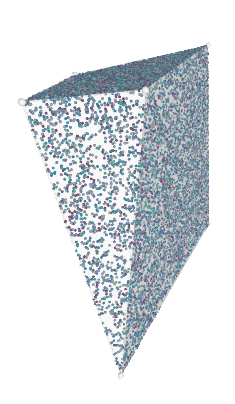}
	\includegraphics[width=0.38\linewidth]{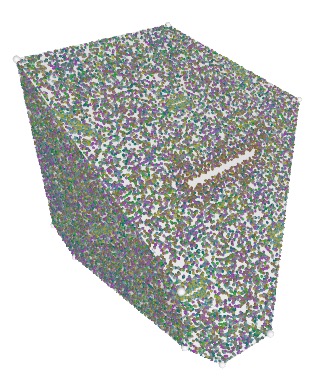}
	\caption{Orbits of a hyperbolic map on three randomly constructed tropical K3 surfaces.
	}
	\label{fig:pictures_computer_simulations_random_orbits1}
\end{figure}

A setup as above arises when considering a $1$-parameter family of K3 surfaces, such that the central fiber is maximally degenerate in a precise sense.
The real $2$-dimensional sphere corresponding to the tropical K3 approximates much of the data on the degenerating family of complex surfaces.
For example, the Gromov--Hausdorff limit of the complex surfaces with normalized Ricci-flat metric is predicted \cite{KontsevichSoibelman_HMStorusFibrations} to be the real $2$-sphere with a (singular) metric of a special form.
Boucksom--Jonsson \cite{BoucksomJonsson_LimitsOfVolumes} proved that this holds at the level of volume forms.


\subsection{Context and results}
	\label{ssec:context_and_results}

In this text we consider the Berkovich and tropical pictures in parallel.
A direct relationship between the two is given by a theorem of Payne \cite{Payne_AnalytificationLimitTropicalizations} which says that the Berkovich space can be recovered as the projective limit of tropicalizations.
The interesting dynamical phenomena can be observed already at the level of the tropicalization.

\begin{figure}[h]
	\centering
	\includegraphics[width=0.49\linewidth]{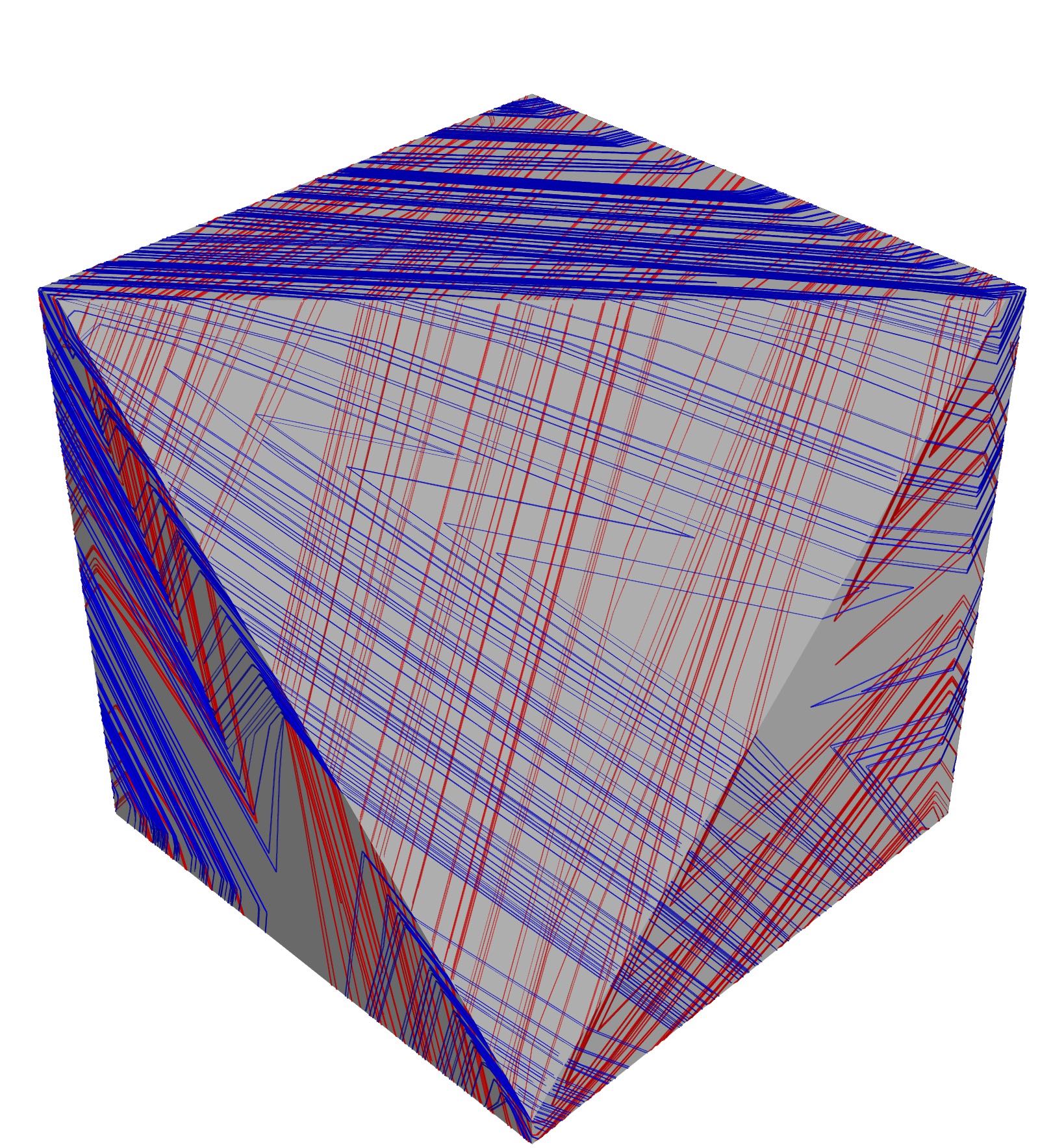}
	\includegraphics[width=0.49\linewidth]{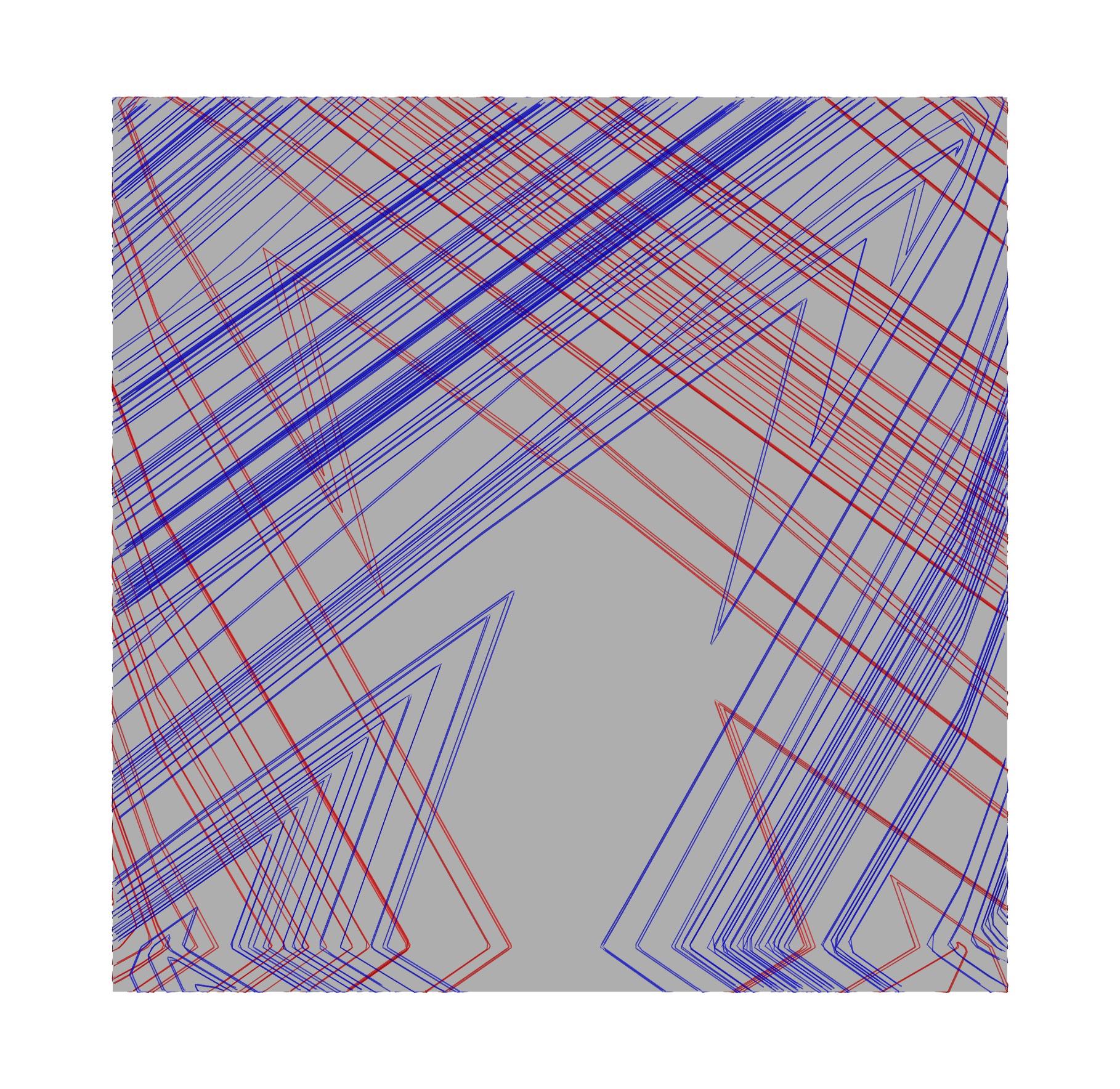}
	\caption{Pictures of currents in Rubik's cube example.}
	\label{fig:pictures_of_stable_unstable}
\end{figure}

\paragraph{Main results}
\autoref{thm:positivity_of_the_current} and \autoref{cor:invariant_measures} construct currents which are scaled by the dynamics, and an invariant measure.
Specifically, we have:
\begin{theorem*}
	Let $f \colon X\to X$ be a projective automorphism of a K3 surface over a complete non-archimedean valued field $K$ of residue characteristic zero.
	Suppose that the action of $f^*$ on $\Pic(X)$ is hyperbolic, i.e. there exists a unique up to scale $v\in \Pic(X)\otimes_\bZ\bR$ which is an eigenvector with eigenvalue $\lambda>1$.
	Let $f^{an}\colon X^{an}\to X^{an}$ denote the Berkovich analytification of the automorphism and K3 surface.

	\begin{enumerate}
		\item There exist closed positive $(1,1)$-currents $\eta_\pm$ on $X^{an}$ such that $(f^{an})^{*}\eta_\pm = \lambda^{\pm 1} \eta_\pm$.
		\item The measure $\mu^{an} = \eta_+\wedge \eta_-$ is $f^{an}$-invariant and non-zero.
		\item The currents satisfy $\eta_+\wedge \eta_+ = 0 = \eta_-\wedge \eta_-$.
	\end{enumerate}
\end{theorem*}

\paragraph{Context and motivations}
The currents on Berkovich spaces that we construct in \autoref{ssec:the_currents_in_berkovich_dynamics} can be viewed as ``non-uniformly hyperbolic measured foliations''.
Indeed, locally in $\bR^2$ a closed positive current $\eta$ is determined by a convex function $\phi$ via $\eta = d'd''\phi$ for appropriate operators $d',d''$ (see \autoref{ssec:differential_forms_on_berkovich_spaces}).
When $\phi$ is $C^2$ the condition that $\eta\wedge\eta=0$ becomes $\det(\operatorname{Hess}\phi)=0$, where $\Hess \phi = \left\lbrace \partial_{i,j}\phi \right\rbrace_{i,j}$ is the Hessian matrix of second derivatives.
Assuming for simplicity that $\rk \Hess (\phi) =1$, a result of Hartman \& Nirenberg \cite{HartmanNirenberg_JacobiansSign} (see Foote \cite{Foote_MongeAmpereFoliations} for a more extended discussion) implies that there exists locally a foliation of $\bR^2$ by lines (defined by $\ker \Hess (\phi)$), and $\phi$ restricted to each line is affine.
A convex function restricted to a line segment determines a measure (take second derivatives) and for the foliation described above, two segments which have endpoints on the same leaves, and are parallel, will have the same induced measure.
A good example to keep in mind is $\sqrt{x^2+y^2}$ on $\bR^2\setminus 0$.
For further analogies between the dynamics on K3 and Riemann surfaces, see the introduction in \cite{sfilip_CountingSlags}.

\begin{figure}[h]
	\centering
	\includegraphics[width=0.49\linewidth]{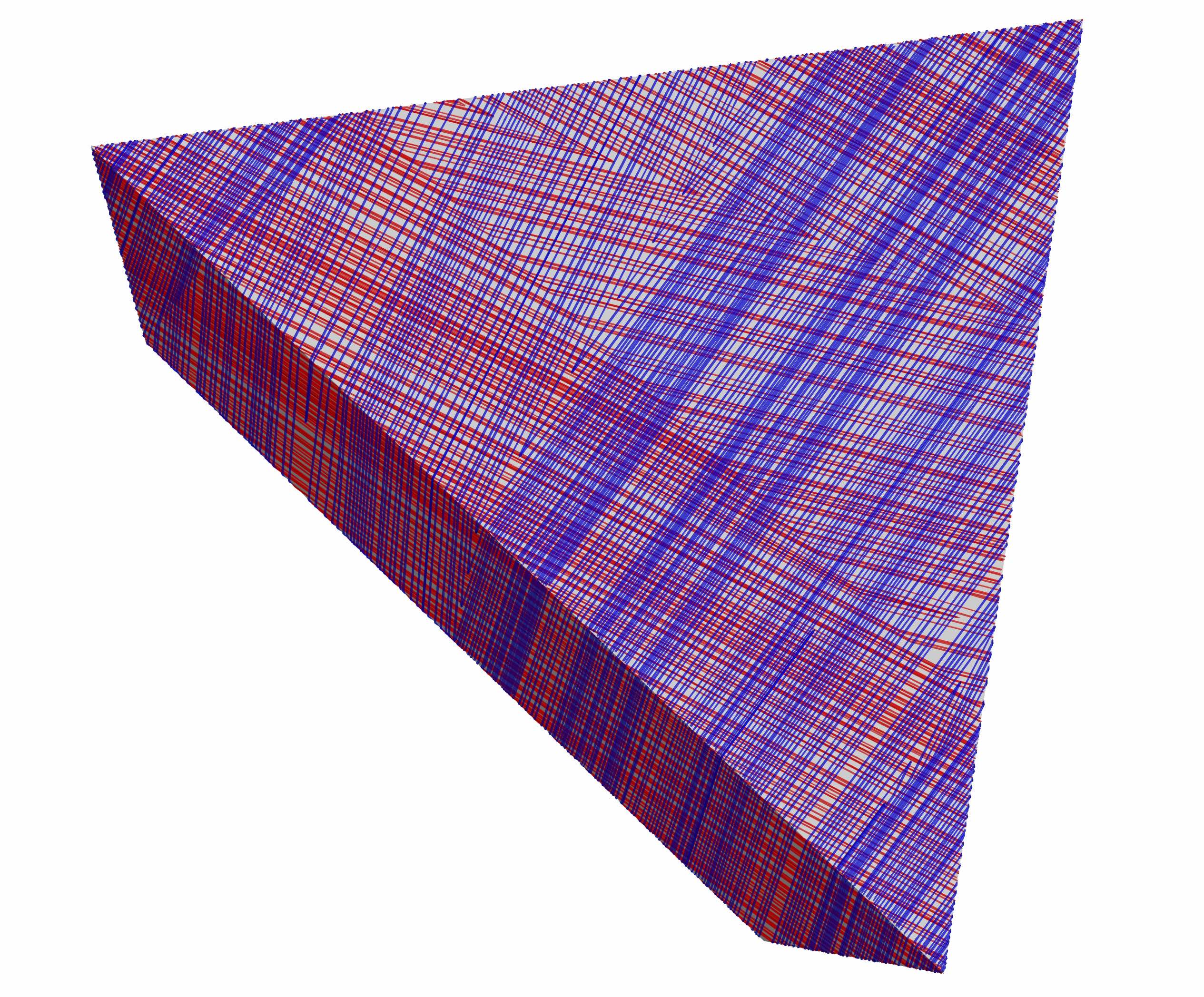}
	\includegraphics[width=0.49\linewidth]{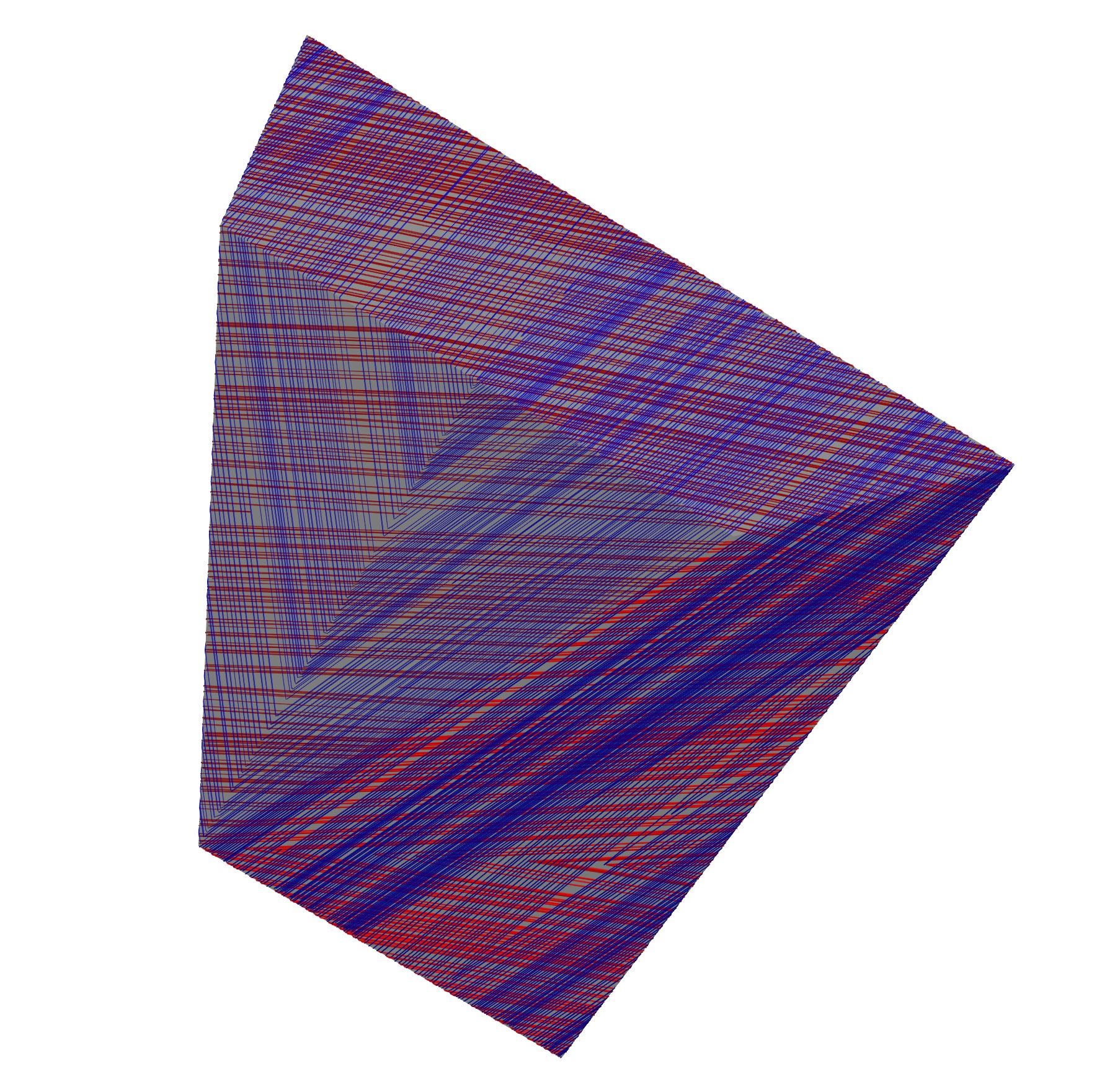}
	\caption{Stable and unstable currents of a perturbed Kummer example, viewed from different angles.
	The perturbed Kummers exhibit tangency of the stable and unstable manifolds.}
	\label{fig:kummer_chopped_currents}
\end{figure}

When $\phi$ is not $C^2$ the condition $\eta\wedge \eta=0$ has meaning using the theory developed by Lagerberg \cite{Lagerberg_Supercurrents}.
The examples in \autoref{fig:pictures_of_stable_unstable}, \autoref{fig:kummer_chopped_currents} and \autoref{fig:kummer_currents} suggest that $\phi$ is never $C^2$ except in the case of Kummer examples (see \cite{CantatDupont} and \cite{FilipTosatti_SmoothRough,FilipTosatti_KummerRigidity} for the complex version of this statement).

\begin{wrapfigure}{L}{0.5\textwidth}
	\centering
	\includegraphics[width=\linewidth]{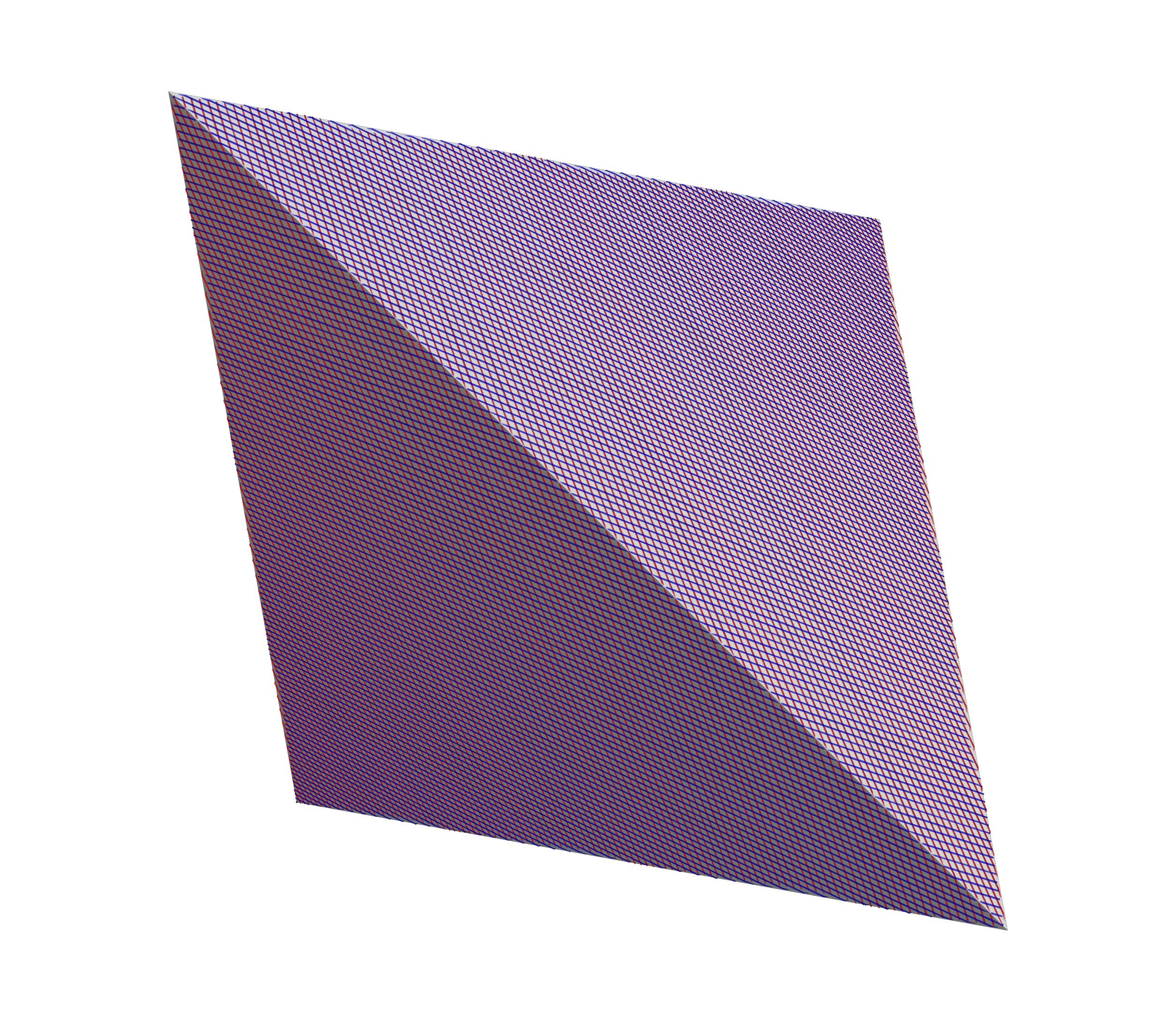}
	\caption{Stable and unstable currents in the Kummer examples have smooth potentials and are uniformly hyperbolic.}
	\label{fig:kummer_currents}
\end{wrapfigure}

\paragraph{Relation between non-archimedean and archimedean objects}
In future work we hope to explore the relation between the non-archi\-me\-dean objects introduced in this work and degenerating families of K3 surface automorphisms.
Specifically, suppose that $\cX$ is a K3 surface with an automorphism, each defined over $K$, where $K$ is the field of holomorphic functions on the punctured unit disc, with poles of bounded order at the origin allowed.
Then for every $t\neq 0$ and $|t|<1$ we have a complex K3 with associated automorphism.
The measure of maximal entropy $\mu_t$ has a Lyapunov exponent $\lambda_t$.
There are also stable/unstable currents $\eta_{\pm,t}$.

The results of Favre \cite{Favre_Degeneration} and Boucksom--Jonsson \cite{BoucksomJonsson_LimitsOfVolumes} suggest that on the associated hybrid space $\cX^{hyp}$, which contains both complex and Berkovich points, the measures $\mu_t$ converge (weakly) to the measure $\mu^{an}$, constructed in \autoref{cor:invariant_measures}.
Furthermore, the Lyapunov exponent of $\mu_t$ should blow up for $t\to 0$ as $\log \frac{1}{|t|}$ times the Lyapunov exponent of $\mu^{an}$.
In the case of endomorphisms of $\bP^k$, this last fact is the main result of \cite{Favre_Degeneration}.
We hope to investigate these questions on K3 surfaces in future work.



\subsection{Paper Outline}
	\label{ssec:paper_outline}

In \autoref{sec:general_constructions} we introduce the basic definitions and constructions, in the tropical and algebraic setting.

In \autoref{sec:the_elliptic_curve_case} we study the simplest case, namely dynamics on tropical elliptic curves.
Since these are just circles equipped with a volume form, it is not surprising that one recovers dynamical systems conjugate to circle rotations.

\autoref{sec:general_properties_of_tropical_k3_automorphisms} considers some general constructions in the setting of tropical K3 surfaces.
The main result is the construction of an unstable potential whose Laplacian is expected to give the unstable current.
It is possible that there is an elementary way to establish the positivity of the unstable current, the way this is done in \autoref{sec:expanding_pl_maps_on_the_line} in a $1$-dimensional context.
The approach taken below is via Berkovich spaces, which is no longer elementary and requires a much bigger machinery.

\autoref{sec:berkovich_spaces_and_k3_dynamics} provides a brief summary of some aspects of Berkovich spaces, which were introduced by Berkovich in \cite{Berkovich_SpectralTheory}.
A key tool in our discussion is the formalism of differential forms on Berkovich spaces, developed by Chambert--Loir \& Ducros \cite{ChamberLoirDucros} based on earlier work of Lagerberg \cite{Lagerberg_Supercurrents}.
In \autoref{ssec:the_currents_in_berkovich_dynamics} we construct the currents in the Berkovich setting, following a strategy similar to the tropical case.
In both situations, the ideas go back to Cantat \cite{Cantat_DynamiqueK3}.

\autoref{sec:examples_of_pl_maps_on_tropical_k3_surfaces} includes a number of examples and illustrations.
It contains, in particular, uniformly hyperbolic examples that come from the Kummer construction and are thus semiconjugate to Anosov automorphisms of the torus.
It also describes a simple $1$-parameter family to which it would be interesting to apply the methods in earlier sections.

\autoref{sec:expanding_pl_maps_on_the_line} concludes the paper with an application of these ideas to the tropicalization of rational maps of $1$ variable.
These have been studied extensively in the Berkovich setting (see e.g. \cite{Favre_RiveraLetelier}) but the treatment here is elementary and in a slightly different setting.
A natural ``positive'' cone of functions is invariant under the dynamics in this case and it allows us to construct directly a convex potential.
It would be interesting to see if some of the constructions in \autoref{sec:expanding_pl_maps_on_the_line} extend to K3 surfaces, in order to obtain positivity of the stable and unstable currents.



\subsection{Some further questions and remarks}
	\label{ssec:some_further_questions}

A number of questions for further investigation are scattered throughout the text.
Here we collect some of them.
\begin{enumerate}
	\item Simple elliptic integrable dynamics arises from the above tropical constructions in dimension $1$.
	A question regarding the variation of twist angle is formulated in \autoref{rmk:variation_of_twist_angle}.
	Next, one sees a coupling of the two integrable regimes at the ``infinity'' of the ambient tropical space.
	This is outlined in \autoref{sssec:dynamics_at_infinity} and it would be interesting to investigate it further.
	\item For the PL dynamical systems on the line, some further questions are formulated in \autoref{sssec:some_further_questions}.
	\item Following DeMarco \cite{DeMarco_Bifurcation}, one can also study the variation of the currents associated to the dynamical systems.
	This is done by constructing the potential on the entire parameter space and checking its subharmonicity.
	In the present case, a similar approach should establish its convexity.
\end{enumerate}

\paragraph{Lozi mappings}
Introduced in \cite{Lozi_AttracteurEtrange}, these are PL analogues of H\'enon mappings of the form
\[
	H_{a,b}(x,y) = (y+1-a|x|,bx) \quad H_{a,b}\colon \bR^2\to \bR^2
\]
In analogy with the work of Bedford--Lyubich--Smillie \cite{BedfordLyubichSmillie}, it seems natural to construct the stable and unstable currents for Lozi mappings using the techniques from \autoref{sec:general_properties_of_tropical_k3_automorphisms}.
Note that the same observation as in \autoref{sec:expanding_pl_maps_on_the_line} (\autoref{prop:preserving_cones_1d}) applied to $\bR^4$ instead of $\bR^2$ implies that the constructed currents are positive.

\paragraph{Compatibility between tropical and Berkovich pictures}
In future work, we will discuss the relation between the constructions in \autoref{sec:general_properties_of_tropical_k3_automorphisms} and \autoref{sec:berkovich_spaces_and_k3_dynamics}.
The diagram summarizing the relevant maps is the following:
	\[
	\begin{tikzcd}[row sep=scriptsize, column sep=scriptsize]
	 & E^{trop} \arrow[hookleftarrow]{rr} \arrow{dd} & & E^{trop}\vert_{\Sk(X^{trop})}  \arrow{dd} \\
	E^{an} \arrow[hookleftarrow, crossing over]{rr}  \arrow{dd} \arrow{ur}{\Trop} & & E^{an}\vert_{\Sk(X^{an})} \arrow{ur}{\Trop}  \\
	& X^{trop} \arrow[hookleftarrow]{rr}& & \Sk(X^{trop}) \arrow[leftrightarrow]{dl} \\
	X^{an} \arrow[hookleftarrow]{rr} \arrow{ur}{\Trop} & & \Sk(X^{an}) \arrow[crossing over, leftarrow]{uu} \\
	\end{tikzcd}
	\]
Given a projective K3 surface $X$ equipped with an automorphism $f\colon X\to X$, we can construct a vector bundle $E\to X$ and a lift $F\colon E\to E$ of $f$ to $E$.
The bundle $E$ can be used to establish the positivity properties of the appropriate currents.

The entire picture is compatible with passing to Berkovich analytification, giving spaces $E^{an}\to X^{an}$ and maps $f^{an},F^{an}$.
There is a canonical Berkovich skeleton $\Sk(X^{an})\subset X^{an}$ which in the case of a maximally degenerating K3 is homeomorphic to a $2$-sphere.
The map $f^{an}$ preserves the Berkovich skeleton.

A choice of embedding of $X$ into a toric variety gives tropicalization maps $\Trop$.
There is an associated tropical skeleton $\Sk(X^{trop})$ which under appropriate conditions is identified with $\Sk(X^{an})$ (see \cite{BPR_NonArchimedeanGeometryTropicalization} for a related discussion for curves).
While the map $f^{an}$ does not descend to a map of $\Trop(X^{an})=:X^{trop}$, it is possible to describe tropically the map on $\Sk(X^{trop})$.

\paragraph{Some conventions}
Throughout, we follow the $\min$ convention of tropical geometry as we are viewing the objects as depending on a \emph{small} parameter, so $\ve^{\alpha}+ \ve^{\beta}\approx \ve^{\min(\alpha,\beta)}$.
This has the notational disadvantage that it leads to concave functions instead of convex ones, but this is only a matter of convention.
For signs of the potential defining a metric on a line bundle, see \autoref{rmk:on_signs}.

\paragraph{Acknowledgments}
I am grateful to the referee for a thorough reading of the text and numerous remarks and suggestions that improved the presentation.
I am grateful to Laura DeMarco, Phil Engel, Mattias Jonsson, Sam Payne for discussions on the topic of this paper.
I am grateful to Curt McMullen for discussions related to computer simulations.
Figures were created using Sage (\cite{sagemath}) and Mayavi (\cite{ramachandran2011mayavi}).
This research was partially conducted during the period the author served as a Clay Research Fellow.



\section{General Constructions}
	\label{sec:general_constructions}

After introducing some basic terminology from tropical geometry we describe a basic class of Calabi--Yau manifolds to which our discussion will apply.
These are given as degree $(2,2,\cdots,2,2)$ hypersurfaces in $(\bP^1)^n$ and are classical examples with a large automorphism group.
The focus will be on $n=2, 3$, which leads to elliptic curves and K3 surfaces respectively.
We describe their basic algebraic and tropical properties.
An analysis of the dynamics for elliptic curves is in \autoref{sec:the_elliptic_curve_case}, and for K3 surfaces in \autoref{sec:general_properties_of_tropical_k3_automorphisms}.

An introduction to tropical geometry and its techniques is in the monograph \cite{MaclaganSturmefls} as well as the collection of notes \cite{ItenbergMikhalkinShustin_OberwolfachSeminars}.

\subsection{Fields and Tropicalization}
	\label{ssec:fields_and_tropicalization}

\subsubsection{Fields with a valuation}
	\label{sssec:fields_with_a_valuation}
Recall that a \emph{valuation} on a field $K$ is a map
\[
	v\colon  K^\times \to \Gamma
\]
from the non-zero elements of $K$ to an ordered abelian group $\Gamma$ (written additively), satisfying for all $a,b\in K^\times$
\begin{align*}
	v(a\cdot b) & = v(a) + v(b)\\
	v(a+b) & \geq \min(v(a),v(b))
\end{align*}
One can define also $v(0)=\infty$ with $\infty>\gamma, \forall \gamma \in \Gamma$.

The basic example used later on will be of fields containing power series in one complex or real variable.
For example, if $K=k(t)$ is the field of rational functions, then the order of vanishing at the origin
\[
	v\colon k(t)\to \bR \quad v(f):= \ord_{t=0}(f)
\]
provides a valuation.

In all situations below the value group $\Gamma$ is contained in $\bR$, so a valuation also gives a non-Archimedean norm on $K$ via:
\[
	\norm{f}_v := e^{-v(f)}
\]
satisfying $\norm{fg}_v = \norm{f}_v \norm{g}_v$ and $\norm{f+g}_v \leq \max(\norm{f}_v,\norm{g}_v)$.

\subsubsection{Tropicalization map}
	\label{sssec:tropicalization_map}

Given a field $K$ with a valuation $v\colon K \to \bR$, define the map
\begin{align*}
	\Trop\colon  (K^\times)^n &\longrightarrow \bR^n\\
	(f_1,\ldots,f_n) &\mapsto (v(f_1), \ldots, v(f_n))
\end{align*}

\subsubsection{PL-functions}
	\label{sssec:pl_functions}
A monomial $m:=ax_1^{i_1}\cdots x_n^{i_n}\in K[x_1^{\pm 1},\ldots, x_n^{\pm 1}]$ gives an affine linear function
\begin{align*}
	\Trop(m)\colon \bR^n &\longrightarrow \bR\\
	(e_1,\ldots, e_n)& \mapsto v(a) + i_1\cdot e_1  + \cdots + i_n\cdot  e_n
\end{align*}
A Laurent polynomial in several variables $f = \sum c_\alpha x^\alpha$ (with $\alpha$ a multi-index) gives a PL-function
\begin{align*}
	\Trop(f) \colon & \bR^n \longrightarrow \bR\\
	(e_1,\ldots,e_n) & \mapsto \min_{m_\alpha}(\Trop(m_\alpha))
\end{align*}
where $m_\alpha = c_\alpha x^\alpha$ are the individual monomials.

\subsubsection{Tropical Varieties}
	\label{sssec:tropical_varieties}
A PL-function is affine linear except along a collection of positive-codimension polyhedra in $\bR^n$ (possibly infinite in some directions).
The tropical variety $V(\Trop(f))$ associated to a PL-function $\Trop(f)$ is defined to be the break locus of $\Trop(f)$, i.e. the polyhedra along which two or more of the defining affine linear functions of $\Trop(f)$ agree, and where they take minimal values among the other defining functions.

Each component in the complement of a tropical variety is naturally labeled by a monomial $m_\alpha$, corresponding to the affine linear function $\Trop(m_\alpha)$ which agrees with $\Trop(f)$ on that component.
If there is a point $x\in (K^\times)^n$ such that $f(x)=0$ then necessarily $\Trop(x)\in V(\Trop(f))$, since there must be at least two monomials of lowest valuation in order to have $f(x)=0$.
Furthermore, an approximate converse to this holds when $K$ is algebraically closed (see \cite[Thm.~3.2.3]{MaclaganSturmefls}) and is called the ``Fundamental Theorem of Tropical Geometry''.

\subsubsection{Dual Subdivision}
	\label{sssec:dual_subdivision}
The convex hull of the multi-indices $\alpha\in \bZ^n$ which appear in the definition of $f$ is called the \emph{Newton polytope} of $f$.
The valuations of the coefficients of $m_\alpha$ determine a subdivision of the Newton polytope, such that the tropical variety associated to $\Trop(f)$ is dual to this subdivision (see \cite[Prop.~3.1.10]{MaclaganSturmefls}).
Vertices of $V(\Trop(f))$ correspond to cells in the subdivision, and vice-versa a vertex in the Newton polytope corresponds to a component of $\bR^n\setminus V(\Trop(f))$ on which the corresponding affine linear function is minimized.
See \autoref{fig:pictures_elliptic_curve_monomials} for an illustration.

\subsubsection{Balancing and Smoothness}
	\label{sssec:balancing_and_smoothness}
(see \cite[\S3.3]{MaclaganSturmefls})
First, consider the case $n=2$, when $V(\Trop(f))$ is a $1$-dimensional polyhedral complex in $\bR^2$.
Consider the star of a vertex $x\in V(\Trop(f))$, i.e. the rays coming out of $x$.
Each component of the complement (in a small neighborhood of $x$) is labeled by an affine linear function with integral slope; the functions all agree at $x$ and pairwise agree along the rays separating their components.
On each of these rays there is a choice of integral vector $v_i$, and if $v_i$ is not primitive, its \emph{multiplicity} is the positive integer $k_i$ such that $\frac{1}{k_i}v_i$ is primitive.
The vectors $v_i$ satisfy the balancing condition:
\[
	\sum_i v_i = 0.
\]
For a higher-dimensional tropical hypersurface $V(\Trop(f))$, the balancing condition is expressed in terms of codimension $1$ faces, by taking the quotient by a linear subspace and reducing to dimension $2$.
The multiplicities are now assigned to top-dimensional faces of $V(\Trop(f))$.

A tropical variety is \emph{smooth} if all its top-dimensional faces have multiplicity $1$.

\subsubsection{Integral-affine structure on the tropical manifold}
	\label{sssec:integral_structure_on_the_tropical_manifold}
Each top-dimensional face of a tropical manifold $V(\Trop(f))$ is naturally a subset of an affine space such that the associated vector space contains a $\bZ$-lattice.
Indeed, the affine space is the level set of an integral linear function.

Consider now a codimension $1$ subset where three top-dimensional faces $\delta_1,\delta_2,\delta_3$ intersect along a common face $\tau$; each $\delta_i$ has dimension $n-1$ and $\tau$ has dimension $n-2$.
Let $v_i, i=1,2,3$ be integer vectors starting at the intersection locus and pointing along each of the faces, such that when taking the quotient by the linear space determined by $\tau$ the $v_i$ satisfy the balancing condition $v_1 + v_2 + v_3 = 0$.
Then taking the quotient $\bR^n/ \bR\cdot v_3$ gives an identification of the linear spaces determined by $\delta_1$ and $\delta_2$ and moreover it is compatible with the $\bZ$-lattices contained in the corresponding linear spaces.

Note that the identification of $\bZ$-lattices obtained this way depends on the choice of $v_i$, and the $v_i$ are well-defined up to the addition of integer elements in the space determined by $\tau$.
Applying the identifications along a loop in the tropical manifold (which avoids a codimension $2$ set) there can (and typically will) be monodromy.

The case of interest below is when the tropical manifold will have a skeleton homeomorphic to the $n$-sphere and $n=1,2$.
For $n=1$ the above construction will endow the skeleton (a circle) with charts to $\bR$ with gluing along translations.
For $n=2$, it will give an integral-affine structure on the $2$-sphere, with singular points along certain edges; the singular points are free to move along those edges (depending on the choice of charts and vectors $v_i$).

\subsection{The basic example of an automorphism}
	\label{ssec:the_basic_example}

Denote by $\bG_m$ the invertible elements of $K$, i.e. $K^\times$.
When the background space is $\bG_m^n$, also called an $n$-dimensional torus, varieties will be considered of the form
\begin{align}
	\label{eqn:algebraic_equation_on_torus}
	0 = h(x_1,\cdots, x_n) = \sum_{i_\bullet = -1,0,1} a_{i_1,\ldots, i_n} x_1^{i_1} \cdots x_n^{i_n}
\end{align}
When compactifying, varieties will be considered inside $(\bP^1)^n \supset \bG_m^n$, and the homogenized equation will be
\begin{multline}
	\label{eqn:homogenized_algebraic_equation}
	0 = H(X_1:Y_1, \ldots, X_n:Y_n) = \\
	= \sum_{i_\bullet = -1,0,1} a_{i_1,\ldots, i_n} (X_1^{1+i_1}Y_1^{1-i_1}) \cdots (X_n^{1+i_n}Y_n^{1-i_n})	
\end{multline}
so that $x_j=\frac{X_j}{Y_j}$ and \autoref{eqn:homogenized_algebraic_equation} is obtained from \autoref{eqn:algebraic_equation_on_torus} by multiplying by $X_jY_j$ for $j=1\ldots n$.
For considerations involving say the first pair of variables $(X_1:Y_1)$ it is convenient to write the equation as
\[
	0 = Y_1^2\cdot H_{-1} + X_1 Y_1\cdot  H_0 + X_1^2 \cdot H_{1}
\]
where $H_{\bullet}$ are homogeneous degree $2$ functions of the remaining variables.

\subsubsection{Holomorphic volume form}
	\label{sssec:holomorphic_volume_form}
\autoref{eqn:homogenized_algebraic_equation} determines a Calabi--Yau manifold in $(\bP^1)^n$, i.e. there exists a holomorphic nowhere vanishing $n$-form $\Omega$ on the zero locus of $H$.
An explicit way to write $\Omega$ is by the residue construction; in coordinates on $\bG^n_m$ it is given by
\[
	\Omega = \Res_{h=0} \left( \frac{1}{h}\frac{dx_1}{x_1}\wedge\cdots \wedge \frac{dx_n}{x_n} \right)
\]
where $h$ is defined in \autoref{eqn:algebraic_equation_on_torus}.

\subsubsection{Vieta involutions}
	\label{sssec:vieta_involutions}
Holding the variables $x_2,\ldots,x_n$ fixed, rewrite the equation as
\[
	0 = \frac 1 {x_1} \cdot  h_{-1}(x_2,\ldots,x_n) + h_{0}(x_2,\ldots,x_n) + x_1\cdot h_{1}(x_2,\ldots,x_n)
\]
Viewing this as a quadratic equation in $x_1$ (by multiplying out by $x_1$), the involution $\sigma_1$ exchanges the two roots:
\begin{align*}
	pr_{x_i}\circ \sigma_1(x_1,\ldots,x_n) = \begin{cases}
		x_i &\text{ if }i\geq 2\\
		\frac{h_{-1}(x_2,\ldots,x_n)}{h_{1}(x_2,\ldots,x_n)} \cdot \frac{1}{x_1} & \text{ if } i =1
	\end{cases}
\end{align*}
where $pr_{x_i}$ denotes projection onto the $i$-th coordinate.
The other involutions $\sigma_i, i=2\ldots n$ are defined analogously.
Note that one can define equivalently $pr_{x_1}\circ\sigma_1(x_1,\ldots,x_n)=-\frac{h_{0}(x_2,\ldots,x_n)}{h_{1}(x_2,\ldots,x_n)} - x_1 $, but in the multiplicative form specified above it will be easier to compute valuations.

\subsubsection{Indeterminacies}
	\label{sssec:indeterminacies}
The definition of $\sigma_1$ above is ambiguous since it is given as a rational function.
However, using the principle that $\sigma_1$ exchanges the two roots of the quadratic equation with $x_2,\ldots, x_n$ fixed, it is possible to define $\sigma_1$ in any chart.

Note that the involution will be well-defined when $h_{-1},h_0, h_1$ have no common zeros.
This is the generic situation when $n\leq 3$ so we restrict to this case.

\subsubsection{Tropicalizing the Vieta involutions}
	\label{sssec:tropicalizing_the_vieta_involutions}

Using the formula from \autoref{sssec:vieta_involutions} for $\sigma_1$, its tropicalization is naturally
\begin{align*}
	pr_{e_i}\circ\Trop(\sigma_1)(e_1,\ldots,e_n):=
	\begin{cases}
		e_i &\text{ if } i\geq 2\\
		\Trop(h_{-1}) - \Trop(h_1) - e_1 & \text{ if } i =1
	\end{cases}
\end{align*}
with $pr_{e_i}$ denoting projection to the $i$-th coordinate.
Note that each of $\Trop(h_{-1}),\Trop(h_1)$ is itself a PL-function (of fewer variables).
Similar formulas define involutions $\Trop(\sigma_i)$ for $i=2\ldots n$.

\subsubsection{Homogenizing the Vieta involutions}
	\label{sssec:homogenizing_the_vieta_involutions}

Assume now that $n=3$ and for convenience denote coordinates by $(x,y,z)$ instead of $(x_1,x_2,x_3)$ in earlier sections.
Homogenizing gives a bundle
\[
	(\bA^2\setminus 0)^3 \to (\bP^1)^3
\]
with homogeneous coordinates $(X_0:X_1,Y_0:Y_1,Z_0:Z_1)$ on the total space, and equation
\[
	X_0^2 \cdot H_{-1} + X_0 X_1 \cdot H_{0} + X_1^2 \cdot H_1=0
\]
following \autoref{eqn:homogenized_algebraic_equation}.
The involution $\sigma_x$ can then be lifted as
\begin{align}
	\label{eqn:Sigma_x_formula}
	pr_{X}\circ \Sigma_x(X_0:X_1) = \left( \frac{H_{1}}{ X_0} \colon \frac{H_{-1}}{X_1} \right)
\end{align}
where $pr_X$ denotes projection to $(X_0:X_1)$ and the other coordinates do not change; the lifts $\Sigma_y,\Sigma_z$ are defined analogously and are themselves involutions (note that not all possible lifts are involutions).
The reason for this particular choice of homogenization is the validity of \autoref{prop:lift_of_the_automorphisms} below.

\subsubsection{The bundles and the space}
	\label{sssec:the_bundles_and_the_space}
Denote by $X$ the vanishing locus in $(\bP^1)^3$ of the homogeneous equation $H=0$.
Each of the $\bP^1$-factors carries a natural $\bG_m$-bundle coming from $\bA^2\setminus 0 \to \bP^1$.
This $\bG_m$-bundle is the total space of the line bundle $\cO(-1)$, with the zero section removed.

Denote by $E\subset (\bA^2\setminus 0)^3$ the locus where $H=0$, which is the total space of the $\bG_m^3$-bundle over $X$.

\begin{proposition}[Lift of the automorphisms]
	\label{prop:lift_of_the_automorphisms}
	Assume that $H_{-1},H_0,H_1$ have no common zero except at the origin, and assume the analogous condition for the $Y_\bullet,Z_\bullet$ variables.
	
	Then the automorphisms $\Sigma_x,\Sigma_y,\Sigma_z$ defined in \autoref{sssec:homogenizing_the_vieta_involutions} provide biregular lifts of $\sigma_x,\sigma_y,\sigma_z$ from $X$ to $E$.
\end{proposition}
\begin{proof}
	It is enough to check the claim for $\Sigma_x$, the other involutions being analogous.
	Since the $Y_\bullet,Z_\bullet$-variables are not moved by the involution, only the behavior in the $X_\bullet$-variables needs to be considered.
	From \autoref{eqn:Sigma_x_formula}, the value of $\Sigma_x(X_0:X_1)$ is not defined when either $X_0=0$ or $X_1=0$.
	It suffices to deal with one case, the other being analogous.

	Using the equation of the K3 surface, rewrite
	\[
		\frac{H_1}{X_0} = \frac{-1}{X_1^2}\cdot \left( X_0 \cdot H_{-1} + X_1\cdot H_0 \right)
	\]
	which can be used on the locus where $X_0=0$.
	Since the point $(0:0)$ is excluded by construction, it follows that the lift $\Sigma_x$ is well-defined everywhere.

	It remains to check that there is no value of $(X_0:X_1)$ in the domain for which $\Sigma_x(X_0:X_1)=(0:0)$, or equivalently consider the case $H_1=0=H_{-1}$.
	Since by assumption in that case $H_0\neq 0$, the locus where $H_1=0=H_{-1}$ intersected with $H=0$ is where $X_0=0$ or $X_1=0$ but not both.
	On this locus, using the alternative expression for $\frac{H_1}{X_0}$ from above, it follows that no such $(X_0:X_1)$ exists.
\end{proof}

\subsubsection{A general construction}
	\label{sssec:a_general_construction}
The above discussion is an instance of a more general construction, due to Cantat \cite[\S3.1]{Cantat_DynamiqueK3}.
Specifically, if $f:X\to X$ is a projective automorphism of a variety with $H^1(X)=0$ (or: $\Pic^0(X)=0$) then selecting $L_1,\ldots, L_n$ a basis of the Picard group of line bundles, $f^*$ induces a linear map that can be expressed as an $n\times n$ integral matrix.
Take $E$ to be the total space of $L_1\oplus\cdots \oplus L_n$, or in the present case $L_1^\times\oplus \cdots \oplus L_n^\times$ where $L^\times$ denotes the bundle without the zero section.
Then there is a lift of $f:X\to X$ to the bundle as $F:E\to E$, commuting with the projection $E\to X$.


\subsection{The Picard group in the basic example}
	\label{ssec:the_picard_group_in_the_basic_example}

\subsubsection{The tautological line bundles}
	\label{sssec:the_obvious_line_bundles}
Recall that over $\bP^1$ there is the tautological line bundle $\cO(-1)\to \bP^1$.
Over the product $(\bP^1)^3$ there are therefore three natural line bundles $L_1,L_2,L_3$ associated to each factor.
Let $X:=\left\lbrace H=0 \right\rbrace$ be a degree $(2,2,2)$ surface as before, and let $C_x,C_y,C_z$ be the intersections of $X$ with the planes $\left\lbrace \alpha=const \right\rbrace$ (with $\alpha=x,y,z$) in $(\bP^1)^3$.
Let $[C_x],[C_y],[C_z]\in H_2(X;\bZ)$ denote their homology classes.

\subsubsection{Intersection theory calculation}
	\label{sssec:intersection_theory_calculation}
It is clear that $[C_x]\cap [C_x]=0$ since $[C_x]$ is a fiber in a fibration of $X$ over $\bP^1$, as the $x$-coordinate varies.
Next, $[C_x]\cap [C_y]=2$, since intersecting the surface $X$ to the plane $\left\lbrace x=const \right\rbrace$ gives a degree $(2,2)$-surface, and intersecting with the line $\left\lbrace y=const \right\rbrace$ will have two intersection points.
So for generic values of $x,y$, the curves $[C_x],[C_y]$ intersect at two points.
Using symmetry considerations, it follows that the intersection matrix of the three divisors $C_x,C_y,C_z$ is:
\[
	\begin{bmatrix}
		0 & 2 & 2\\
		2 & 0 & 2\\
		2 & 2 & 0
	\end{bmatrix}
\]

\subsubsection{The action of involutions}
	\label{sssec:the_action_of_involutions}
Suppose now that $\sigma_x$ is the involution defined above.
Then it is clear that $\sigma_x(C_y)=C_y$ and $\sigma_x(C_z)=C_z$ (note that the involution does not fix the curves pointwise, but rather restricts to an involution on each).
To determine $\sigma_x([C_x])$, we compute its intersections with the other generators.

First, observe that $\sigma_x([C_x])\cap [C_y] = [C_x]\cap [C_y]=2$ by applying the involution to both terms in the cap product and recalling that $[C_y]$ is fixed.
For computing $\sigma_x([C_x])\cap [C_x]$, fix say $(X_0':X_1')\in \bP^1$ and note that the condition
\[
	\sigma_x(X_0':X_1') = (X_0':X_1') \quad \text{ in }\bP^1
\]
becomes the equation $H_1\cdot (X_1')^2 - H_{-1}(X_0')^2=0$, which for fixed $X_\bullet$-coordinates, describes a degree $(2,2)$ curve in $\bP^1\times \bP^1$ with the $(Y_\bullet,Z_\bullet)$-coordinates.
Imposing also that the points lie on the surface $\left\lbrace H=0 \right\rbrace$ and in the plane $X_\bullet=(X_0':X_1')$ gives another curve of degree $(2,2)$ and the intersection of two $(2,2)$ curves in $\bP^1\times \bP^1$ has 8 points (by Bezout's theorem).
It follows that $\sigma_x([C_x])\cap [C_x]=8$.
To combine this with the information on the intersection matrix from \autoref{sssec:intersection_theory_calculation}, set $\sigma_x([C_x])=a[C_x]+b[C_y]+c[C_z]$, take intersections with $[C_x],[C_y],[C_z]$ to determine that
\[
	\sigma_x([C_x]) = - [C_x] + 2([C_y]+[C_z])
\]
so in matrix form, using the basis $[C_x],[C_y],[C_z]$ for the subspace they generate in $H_2(X;\bZ)$:
\begin{align}
	\label{eqn:sigma_x_action_on_homology}
	\sigma_x = \begin{bmatrix}
		-1 & 0 & 0\\
		2  & 1 & 0\\
		2  & 0 & 1
	\end{bmatrix}	
\end{align}
and similarly for $\sigma_y,\sigma_z$.

\begin{remark}
	Note that the homogeneous lift \autoref{eqn:Sigma_x_formula}, because it extends to the ambient space of the direct sum of line bundles, also determines the action of $\sigma_x$ on cohomology.
	This action can be read off by considering the homogeneity matrix of $\Sigma_x$ relative to the different variables.

	That such a homogeneous lift exists in general was established by \cite{Cantat_DynamiqueK3} and some of his techniques will be useful later on.
\end{remark}

\subsubsection{Examples of automorphisms}
	\label{sssec:examples_of_automorphisms}
A product of two involutions, say $\sigma_x\sigma_y$, will give an automorphism of the K3 surface which is fixing the fibers of a fibration to $\bP^1$ with the $z$-coordinate.
The fibers are elliptic curves, and the automorphism will act by translations along the elliptic curve.

However, already the product $\sigma_x\sigma_y\sigma_z$ will give a hyperbolic action in cohomology, using \autoref{eqn:sigma_x_action_on_homology}.
This is the typical kind of automorphism with interesting dynamics that one would like to understand.


\subsection{The core pencil}
	\label{ssec:the_core_pencil}

The expression for the automorphisms $\sigma_\bullet$ never involves the ``central'' coefficient of the defining equation $\left\lbrace h=0 \right\rbrace\subset \bG_m^n$.
\begin{definition}[Core pencil]
	\label{def:core_pencil}
	The pencil whose fiber over $t\in \bA^1$ is the hypersurface $\left\lbrace h=t \right\rbrace\subset \bG_m^n$ (or $\subset (\bP^1)^n$) will be called the \emph{core pencil}.
\end{definition}
An advantage of the core pencil is that the same algebraic expressions define automorphisms for all member of the family.

\subsubsection{Tropical core pencil}
	\label{sssec:tropical_core_pencil}
It is in fact convenient to define the tropical function
\begin{align}
	\label{eqn:h_parametrizing_core}
	\Trop(h^\circ)(e_1,\ldots,e_n):=\min_{ \substack{i_\bullet\in{\left\lbrace -1,0,1 \right\rbrace}\\ (i_1,\ldots,i_n)\neq (0,\ldots,0)} } c_{i_1,\ldots,i_n} + i_1 e_1 + \cdots + i_n e_n
\end{align}
so that a tropical Calabi--Yau $X_c$ will be the break locus of $\Trop(h)=\min(\Trop(h^\circ),c)$ for a constant $c$.
The locus $\Trop(h^\circ)=c$ will be called the \emph{skeleton} of $X_c$ and denoted $\Sk(X_c)$ (or $\Sk^{tr}(X_c)$ when it has to be distinguished from the Berkovich skeleton).
When $c$ is fixed, it will be omitted from the notation.

The different skeletons in the core pencil are the different level sets of the same function, so they sweep out $\bR^n$.

\begin{theorem}[The skeleton is preserved]
	\label{thm:the_skeleton_is_preserved}
	The involutions $\Trop(\sigma_i)$ preserve the skeleton of a tropical Calabi--Yau $X_c$ as defined above.
\end{theorem}
\begin{proof}
	It suffices to prove it for $\Trop(\sigma_1)$, the other involutions being analogous.
	Write the equation for $\Trop(h^\circ)$ in the form
	\[
		\Trop(h^\circ)= \min (-e_1 + \Trop(h_{-1}),\Trop(h_0^\circ),e_1+\Trop(h_1))
	\]
	Since $\Trop(\sigma_1)(x_1)= \Trop(h_{-1})-\Trop(h_1)-e_1$, it is clear that $\Trop(\sigma_1)$ preserves the values of $\Trop(h^\circ)$, hence its level sets.
	The level sets are, by definition, the skeletons.
\end{proof}


\section{The Elliptic Curve Case}
	\label{sec:the_elliptic_curve_case}

In this section we specialize the discussion from \autoref{ssec:the_basic_example} to the case $n=2$, which gives a tropical curve in the plane.
The objects are the same as in \autoref{sec:general_constructions} but for ease of notation, instead of $\Trop(h)$ we use below $h$ and omit $\Trop$ from notation; everything in this section is ``tropicalized''.
After explaining the geometry of the tropical elliptic curve in \autoref{ssec:drawing_the_curve}, we consider the dynamics of the two involutions in \autoref{ssec:elliptic_dynamics}.

\subsection{Drawing the Curve}
	\label{ssec:drawing_the_curve}

The PL-function of interest is of the form
\[
	h(e_1,e_2) = \min_{i,j\in \left\lbrace -1,0,1 \right\rbrace} ( c_{i,j} + i\cdot e_1 + j\cdot e_2 ) 
\]
and the locus where two (or more) of the affine linear functions agree and are smaller than the rest equals the tropical elliptic curve.

\subsubsection{The Tentacles}
	\label{sssec:the_tentacles}
Fix $e_1$ and assume that it is sufficiently negative, so that only the affine linear functions of the form $c_{1,j} + e_1 + j\cdot e_2$ will matter for the definition of $\Trop(h)$.
This implies that the break locus will consist of two horizontal lines, computed as follows.
The functions to consider are
\[
	c_{1,-1} + e_1 - e_2,  \quad c_{1,0} + e_1, \quad c_{1,1} + e_1 + e_2
\]
For $e_2\gg 0$ the first function will be smallest, for $e_2\ll 0 $ the last one, and under a genericity assumption on $c_{i,j}$ the middle function will also be minimal for a bounded set of values of $e_2$.
Setting the appropriate functions to be equal gives the break points $e_2 = c_{1,-1} - c_{1,0}$ and $e_2 = c_{1,0} - c_{1,1}$.

A similar analysis applies to the situations when $e_1\gg 0$, $e_2\ll 0$, $e_2\gg 0$.
It implies that the break locus, outside of a compact set, consists of eight rays, two going out in each cardinal direction.

\begin{figure}[ht]
	\centering
	\includegraphics[width=0.95\linewidth]{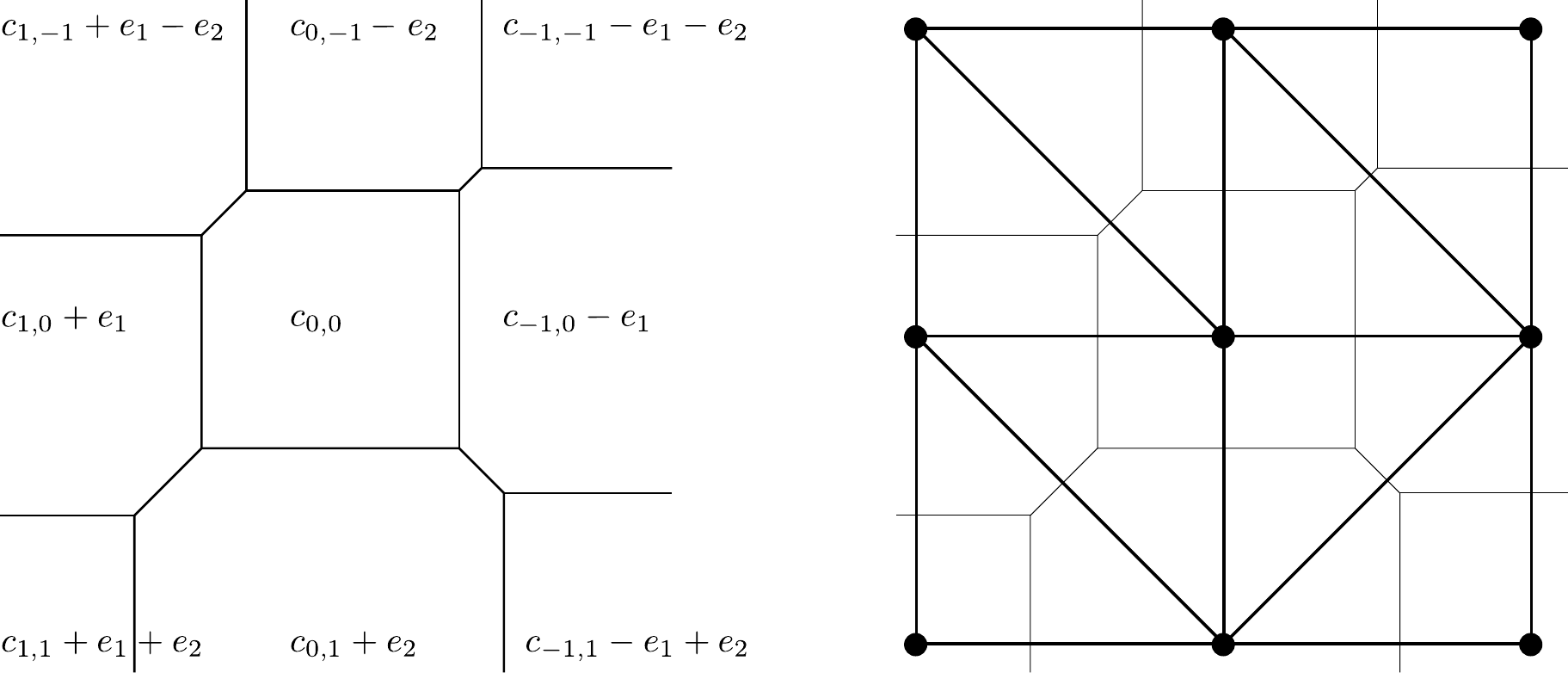}
	\captionsetup{width=0.9\linewidth}
	\caption{Left: The monomials that are minimized in each region of the plane, together with the tropical elliptic curve in the $(e_1,e_2)$-plane.
	Right: The dual subdivision of the Newton polytope.
	The Legendre transform of the function on the left determines the subdivision on the right.
	In the picture, all affine linear functions in the definition of $h$ are minimized for some value of $e_1,e_2$.}
	\label{fig:pictures_elliptic_curve_monomials}
\end{figure}

\subsubsection{The Skeleton}
	\label{sssec:the_skeleton}
Associated to a tropical elliptic curve is its tropical $j$-invariant (see \cite{KMM_TropicalJInvariant}).
When the tropical elliptic curve comes from an algebraic elliptic curve over a valued field, the tropical $j$-invariant is the negative of the valuation of the usual $j$-invariant.
When the $j$-invariant is strictly positive, the curve has an interior cycle which we call its skeleton.
The skeleton is the boundary of the region on which the constant function is minimized.
From now on, we assume that the $j$-invariant is strictly positive and hence that there is a non-trivial skeleton.

\subsubsection{The $\bZ$-structure on the edges and the skeleton}
	\label{sssec:the_Z_structure_edges}
Along each point $p$ in the interior of an edge $E$, there is a natural choice of $\bZ$-lattice $\Lambda_p\subset T_pE$ in the tangent space; this $\bZ$-lattice is obtained by moving $p\in E$ to the origin in $\bR^2$ and intersecting the line generated by $E$ with $\bZ^2$.

At a meeting point of three edges of the tropical curve, there is a balancing condition
\[
	v_1 + v_2 + v_3 = 0
\]
for vectors $v_i$ pointing outwards from the vertex along the edges.
Moreover, each $v_i=(v_{i,1},v_{i,2})\in \bZ^2$ is primitive and has to satisfy $|v_{i,1}|+|v_{i,2}|\leq 3$.
The quotient $\bZ^2/(\bZ v_3)$ is naturally isomorphic to the lattice generated by either of $v_1$ or $v_2$, so there is a natural identification of the lattices along the edges generated by $v_1$ and $v_2$.

This construction gives a well-defined $\bZ$-affine structure on the skeleton.
In the simple case of dimension $1$, this is equivalent to a volume form (or a Riemannian metric).
The length of the skeleton, for this structure, is equal to the tropical $j$-invariant.

\subsection{Elliptic Dynamics}
	\label{ssec:elliptic_dynamics}

\subsubsection{The reflections}
	\label{sssec:the_reflections}

Recall (see \autoref{sssec:tropicalizing_the_vieta_involutions}) that a tropical reflection is given by the formula
\[
	\sigma_2(e_1,e_2) = (e_1,h_1(e_1) - h_{-1}(e_1) - e_2)
\]
where
\begin{align*}
	h_1(e_1) & = \min \left( c_{1,1} + e_1,\, c_{0,1},\, c_{-1,1} - e_1 \right)\\
	h_{-1}(e_1) & = \min \left( c_{1,-1} + e_1,\, c_{0,-1},\, c_{-1,-1} - e_1 \right)
\end{align*}
Each of the functions $h_{\pm 1}$ is PL with two points where it changes slope, and the slopes are $-1, 0, 1$.

The \emph{reflection line} which is fixed by $\sigma_2$ will be of the form $(e_1, \frac{1}{2}\left( h_1(e_1)-h_{-1}(e_1) \right)$ and so will have slopes ranging $-1$ to $1$ with steps of $\frac{1}{2}$.
Note that the two infinite rays of the reflection line will be parallel and pointing in opposite directions, since the slopes of $h_1$ and $h_{-1}$ agree ``at infinity'', i.e. for $e_1\gg 0$ or $0\gg e_1$.
The values of $e_1$ where the reflection line will change slope correspond to those where one of $h_{\pm 1}$ changes slope, so there will be (typically) four of them (the atypical case is that some of the points can coalesce).

A similar analysis applies to the second reflection $\sigma_1$.

\begin{figure}[ht]
	\centering
	\includegraphics[width=0.5\linewidth]{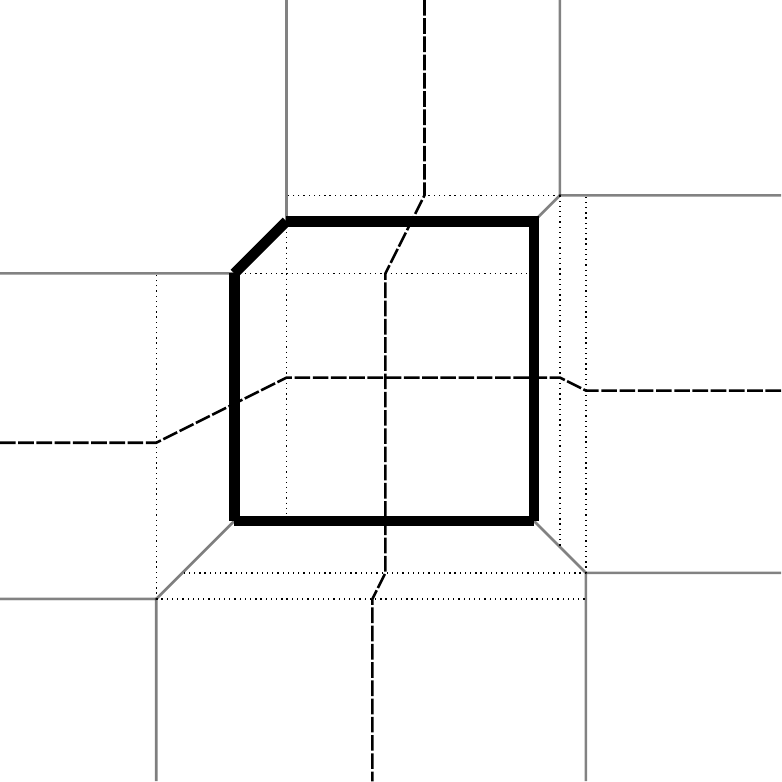}
	\caption{A tropical elliptic curve with the skeleton in bold and dashed reflection lines.
	The dotted vertical and horizontal lines denote the points where the reflection lines change slope.}
	\label{fig:tropical_elliptic}
\end{figure}

A tropical reflection will usually not map the tentacles of the elliptic curve to other tentacles, except for those which are aligned with the reflection line.
However, the skeleton will be preserved by all reflections by \autoref{thm:the_skeleton_is_preserved}.

\begin{wrapfigure}{R}{0.4\textwidth}
	\centering
	\includegraphics[width=\linewidth]{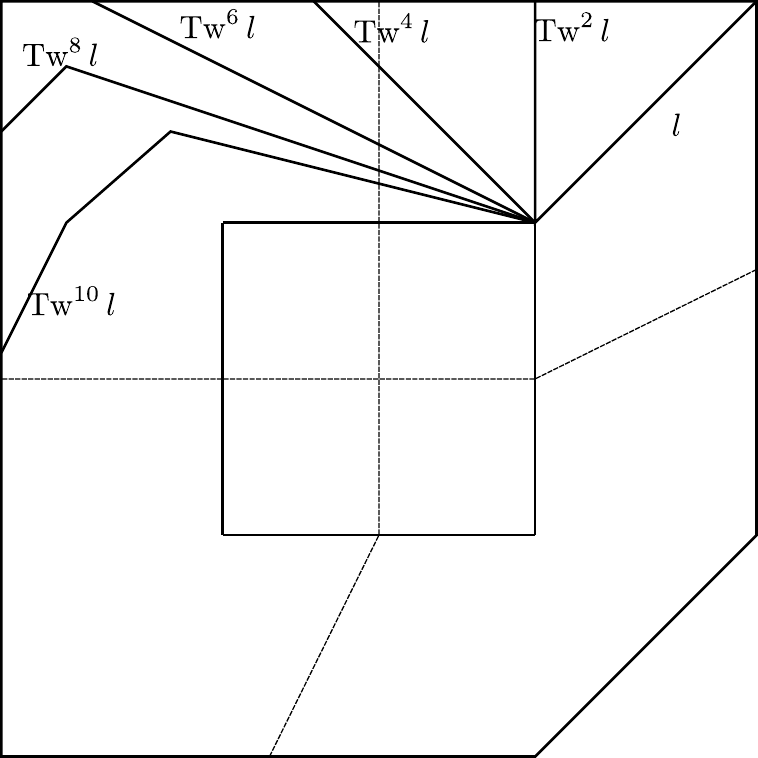}
	\captionsetup{width=0.95\linewidth}
	\caption{The iterate of a segment under the twist map, an analogue of \autoref{ssec:rubik_s_cube_example} in the present case.}
	\label{fig:pictures_elliptic_rubik}
\end{wrapfigure}

\begin{proposition}[Reflections preserve the $\bZ$-structure]
	\label{prop:reflections_preserve_the_}
	A tropical reflection of a tropical elliptic curve preserves the integral structure (in the sense of \autoref{sssec:the_Z_structure_edges}) on the skeleton.
	Equivalently, reflections are measure-preserving for the Lebesgue measure on the skeleton induced by the integral structure.
\end{proposition}
\begin{proof}
	For each reflection, there is a partition of $\bR^2$ into polyhedral domains (in fact strips) such that on each of them, the reflection acts as an affine map with linear part a matrix in $\GL_2(\bZ)$.
	This $\bZ$-structure is compatible with the ambient $\bZ^2$-structure of $\bR^2$.
	The skeleton of the tropical curve breaks up into segments, on each of which the $\bZ$-structure as defined in \autoref{sssec:the_Z_structure_edges} is also induced from the ambient $\bR^2$, since the segments have rational defining vectors.
	Since the $\GL_2(\bZ)$-action preserves the ambient $\bZ$-structure, if it takes one segment to another it will respect their induced $\bZ$-structure.
\end{proof}

\subsubsection{Rotation on the circle}
	\label{sssec:rotation_on_the_circle}
For a circle, the composition of two reflections is a rotation by twice the angle between the fixed points of each reflection.
The action of the two reflections on the skeleton of a tropical elliptic curve is conjugated to the action of two reflections on a circle.

To find the angle of reflection, one must first compute the total Lebesgue volume of the skeleton of the elliptic curve.
Then, one computes the total Lebesgue measure between two fixed points of the reflections.
The ratio of the two volumes is one half of the rotation angle.

\subsubsection{Rotation on the disc}
	\label{sssec:rotation_on_the_disc}

The two reflections of a tropical elliptic curve preserve not just the skeleton, but also the interior of the skeleton.
Moreover, it is clear that the natural Lebesgue volume form is also conserved.
The two reflections will preserve a foliation of the interior by PL-curves, each of which is isomorphic to the skeleton of a tropical elliptic curve (the foliation will have a segment in the center which is fixed by each rotation).

\subsubsection{The scaffolding in the skeleton}
	\label{sssec:the_scaffolding_in_the_skeleton}
Recall (\autoref{ssec:the_core_pencil}) that the curves which are invariant by both reflections are in fact level sets of a natural function
\[
	h^\circ(e_1,e_2) := \min_{\substack{i,j\in \left\lbrace -1,0,1 \right\rbrace \\ (i,j)\neq (0,0)}} c_{i,j} + i\cdot e_1 + j \cdot e_2
\]
which is the same as the function defining the tropical elliptic curve, except that the constant function is omitted from the minimization.
Equivalently, the invariant curves can be viewed as a $1$-parameter family of tropical elliptic curves, whose defining equations have varying $c_{0,0}$ term.

\begin{figure}[ht]
	\centering
	\includegraphics[width=0.95\linewidth]{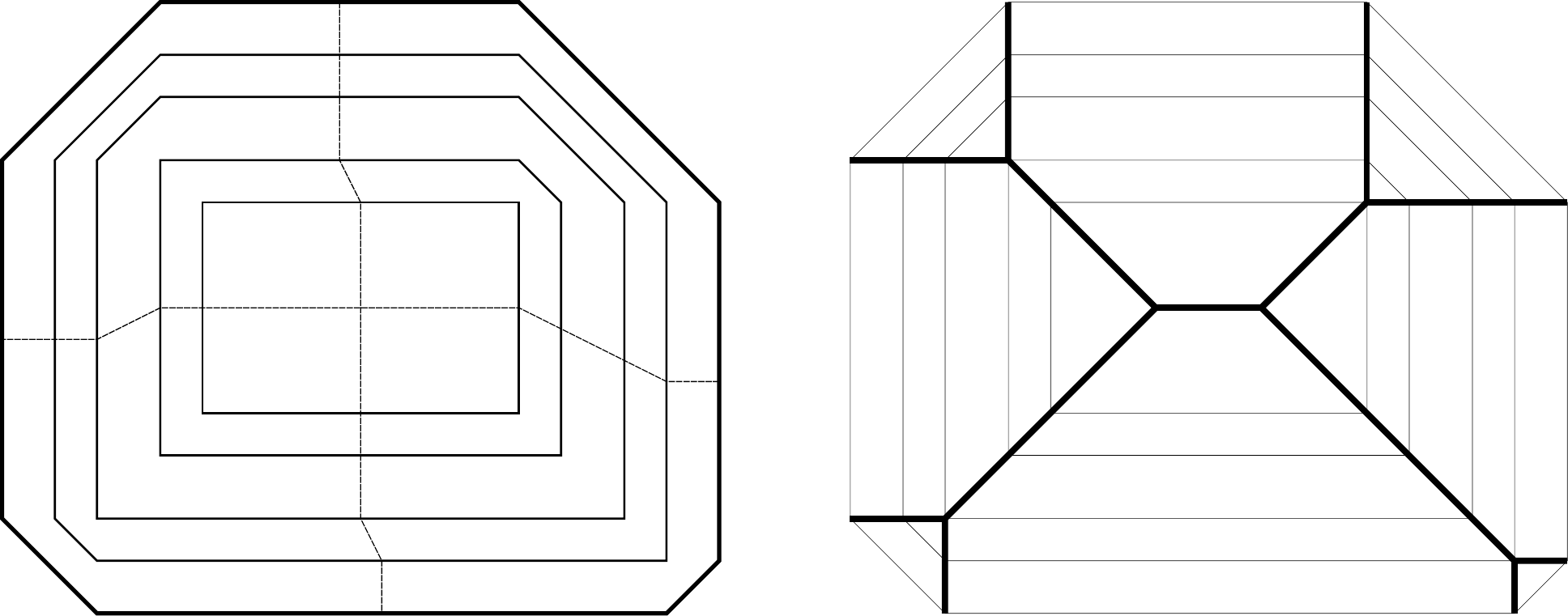}
	\captionsetup{width=0.95\linewidth}
	\caption{Left: The invariant curves of the rotation and the lines of reflection for the involutions.
	Right: The break lines of the function $h^\circ$ which determines the core pencil.}
	\label{fig:pictures_disc_twisting_region}
\end{figure}

The function $h^\circ$ is PL and concave, with derivative constant on some convex domains.
The jump set of the derivative gives a $1$-dimensional complex in $\bR^2$ (see \autoref{fig:pictures_disc_twisting_region}).

\begin{remark}
	\label{rmk:variation_of_twist_angle}
	At the level set which is a maximum of the function $h^\circ$, the composition $\sigma_1\sigma_2$ acts by a finite-order automorphism (of order at most $2$).
	On the level set $h^\circ=c$ with $c\to -\infty$, the composition $\sigma_1\sigma_2$ again approaches a finite-order automorphism.
	The reason is that most of the mass of the skeleton is concentrated on segments with slopes $\pm \frac{\pi}{4}$, which are almost exchanged by the involutions.
	It would be interesting to investigate the variation of angle of rotation of $\sigma_1\sigma_2$ from the maximum of $h^\circ$ to $-\infty$.
\end{remark}


\section{General properties of tropical K3 automorphisms}
	\label{sec:general_properties_of_tropical_k3_automorphisms}

In this section we continue the study of the basic example from \autoref{ssec:the_basic_example} in the case $n=3$, leading to K3 surfaces.
Our focus will now be on the dynamics of an automorphism which acts hyperbolically on the Picard group, where the action on the group is determined by the matrices from \autoref{eqn:sigma_x_action_on_homology}.
We will construct a potential for the dynamics, analogous to classical constructions in complex dynamics, e.g. those of Hubbard--Papadopol \cite{HubbardPapadopol}.
This standard construction uses homogeneous variables, discussed in \autoref{ssec:homogenizations}, and yields the corresponding potential as an application of the contraction mapping principle in \autoref{thm:existence_of_a_potential}.
However, to understand the positivity properties of the potential will require Berkovich spaces, studied in \autoref{sec:berkovich_spaces_and_k3_dynamics}.

For notation, we continue to write $h,h^\circ$ instead of $\Trop(h),\Trop(h^\circ)$ and generally omit $\Trop$ from notation, as all concepts are tropical.
For convenience of notation, we switch to variables $x,y,z\in \bR$ instead of $e_1,e_2,e_3$.

\subsection{Homogenizations}
	\label{ssec:homogenizations}

\subsubsection{Some conventions}
	\label{sssec:some_conventions}
For discussing below the homogenization of a tropical automorphism, we will replace the inhomogeneous variable $x\in \bR$ by the pair $(X_0:X_1)\in \bR^2$ with the quotient map $x=X_1-X_0$ (and similarly for $y,z$).
For convenience, let $\bbV:=\begin{bmatrix}
	1\\
	1
\end{bmatrix}$ denote a vector in the kernel of this quotient map, and write $\bbX $ for $\begin{bmatrix}
	X_1\\
	X_0
\end{bmatrix}$.

When passing to tropical functions, homogeneity will be additive and not multiplicative.
Specifically, a function $\phi\colon \bR^2\to \bR$ will be called \emph{homogeneous of degrees $\alpha$} (for $\alpha\in \bR$) if $\phi(\bbX + \bbV\cdot t)=\phi(\bbX)+ t\cdot \alpha$ for all $t\in \bR$.

More generally, view $(\bR^2)^3=\bR^{2\times 3}$ as $2\times 3$ matrices.
For a column vector $\alpha\in \bR^3$, a function $\phi\colon (\bR^2)^3\to \bR$ will be called \emph{homogeneous of degree $\alpha$} if:
\[
	\phi(p + \bbV \cdot t) = \phi(p) + t\cdot \alpha
\]
for every row vector $t\in \bR^3$ and $2\times 3$ matrix $p$.

Yet more generally, a transformation $\Phi\colon(\bR^2)^3\to (\bR^2)^3$ will be called \emph{homogeneous of degree $M_\Phi$}, where $M_\Phi$ is a $3\times 3$ matrix, if
\[
	\Phi(p + \bbV\cdot t) = \Phi(p) + \bbV \cdot t \cdot M_\Phi
\]

Instead of ``homogeneous of degree $\bullet$'' we will also say ``$\bullet$-homogeneous''.

A direct calculation shows that if $\phi\colon(\bR^2)^3\to \bR$ is $\alpha$-homogeneous and $\Phi$ is $M_\Phi$-homogeneous, then the pulled-back function $\Phi^*\phi$ is $(M_\Phi\cdot \alpha)$-homogeneous.

\subsubsection{The basic example}
	\label{sssec:the_basic_example}
Recall from \autoref{ssec:the_basic_example} that we have the tropical ``rational function''
\[
	h^\circ(x,y,z):= \min_{
	\substack{i,j,k\in \left\lbrace -1,0,1 \right\rbrace \\ 
	(i,j,k)\neq (0,0,0)}}
	{i}x + j y + k z + c_{i,j,k} 
\]
and that the level sets of $h^\circ$ in $\bR^3$ are the skeletons of a family of K3 surfaces (see \autoref{ssec:the_core_pencil}).
Fix a value $c$ of $h^\circ$ and let $\Sk(X):=\left\lbrace h^\circ = c \right\rbrace$ be the corresponding level set, which is the skeleton of the tropical K3 surface associated to $h=\min(h^\circ,c)$.

Recall that to describe an involution that preserves $\Sk(X)$, the $x$ terms in the minimization for $h^\circ$ are grouped as:
\[
	h^\circ(x,y,z) = \min \left( -x + h_{x,-1}(y,z),\,\, h_{x,0}(y,z),\,\, x + h_{x,1}(y,z) \right)
\]
where $h_{x,\bullet}$ are themselves tropical rational functions of $y,z$ written using minimizations in $\pm y,\pm z$.
Then the following involution preserves the level sets of $h$:
\begin{align}
	\label{eqn:sigma_x_tropical}
	\begin{split}
	\sigma_x(x,y,z) & = (x',y,z) \quad \text{ where }\\
	x' & = h_{x,-1}(y,z) - h_{x,1}(y,z) - x
	\end{split}
\end{align}
Indeed, the fixed point of the involution is $x_0 = \frac{1}{2}(h_{x,-1}-h_{x,1})$, which is the $x$-value where $-x+h_{x,-1} = x+h_{x,1}$, and the reflected point can be expressed as $x' = -(x-x_0)+x_0$.
The formulas hold for all values of $y,z$, which are implicit and have been omitted in the last sentence.

\subsubsection{Homogenizing the equations}
	\label{sssec:homogenizing_the_equations_K3}
The natural bundle over $\bR^3$ to which the tropical automorphism lifts is:
\begin{align*}
	\bR^2\times \bR^2 \times \bR^2 & \to \bR^3\\
	\left( X_0:X_1 \right)\times \left( Y_0:Y_1 \right)\times \left( Z_0:Z_1 \right) & 
	\to (X_1-X_0, Y_1-Y_0,Z_1-Z_0)
\end{align*}
The homogenization of $h^\circ$ then becomes:
\begin{align*}
	H^\circ(\bbX,\bbY,\bbZ) = \min_{
	\substack{i,j,k\in \left\lbrace -1,0,1 \right\rbrace \\ 
	(i,j,k)\neq (0,0,0)}}
	P_i(X_0:X_1) + P_j(Y_0:Y_1) + P_k(Z_0:Z_1) + c_{i,j,k}
\end{align*}
where
\[
	P_{-1}(A_0:A_1)=2A_0 \quad P_0(A_0:A_1) = A_0+A_1 \quad P_1(A_0:A_1) = 2A_1
\]
Note that each $P_\bullet$ is homogeneous of degree $2$ (see \autoref{sssec:some_conventions}), so $H^\circ$ is homogeneous of degree $(2,2,2)$.

The equation $\left\lbrace h^\circ=c \right\rbrace$ gives the skeleton $\Sk(X)$ and its lift now becomes
\[
	H^\circ(\bbX,\bbY, \bbZ) = c + (X_0 + X_1) + (Y_0 + Y_1) + (Z_0 + Z_1)
\]
Define $E\subset (\bR^2)^3$ to be the locus of $(\bbX,\bbY,\bbZ)$ where the above equation holds.
Then $E$ is the total space of an $\bR^3$-bundle over $\Sk(X)$:
\[
	E\to \Sk(X)
\]
and the lift of the automorphisms of $\Sk(X)$ to $E$ will be the basic object of study below.

\subsubsection{Homogenizing the automorphisms}
	\label{sssec:homogenizing_the_automorphisms_K3}
Group again the terms in $H^\circ$ according to the $X$-variable:
\begin{align*}
	H^\circ(\bbX,\bbY,\bbZ) = \min & \big( 2X_0 + H_{X,-1}(\bbY, \bbZ), X_0 + X_1 + H_{X,0}(\bbY,\bbZ),\\
	& 2X_1 + H_{X,1}(\bbY,\bbZ) \big)
\end{align*}
so that the lifted automorphism (see \autoref{eqn:sigma_x_tropical}) can be written as:

\begin{align}
	\label{eqn:Sigma_x}
	\Sigma_X(X_0:X_1 ) = \big( H_{X,1}(\bbY,\bbZ) - X_0 : H_{X,-1}(\bbY,\bbZ) - X_1 \big)
\end{align}

An alternative lift of $\sigma_x$ which removes the negative signs would be given by adding $(X_1+X_0)\cdot (1:1)$ to the $X$-coordinate:
\begin{align}
\label{eqn:Sigma_x_pos}
\begin{split}
	\Sigma_X^{pos}(X_0, X_1) & = \big( X_1 + H_{X,1}(\bbY,\bbZ) : X_0 + H_{X,-1}(\bbY,\bbZ) \big)
\end{split}	
\end{align}

\begin{remark}
The advantage of $\Sigma_X^{pos}$ is that it only involves positive slopes, while the advantage of $\Sigma_X$ is that it is a genuine involution: $(\Sigma_X)^2=id$.

The homogeneity matrices of $\Sigma_X$ and $\Sigma^{neg}_X$ are:
\begin{align*}
	M_{\Sigma_X} =
	\begin{bmatrix}
		-1 & 0 & 0 \\
		2 & 1 & 0 \\
		2 & 0 & 1
	\end{bmatrix}
	\quad 
	M_{\Sigma^{pos}_X} =
	\begin{bmatrix}
		1 & 0 & 0 \\
		2 & 1 & 0 \\
		2 & 0 & 1
	\end{bmatrix}.
\end{align*}
Comparing with \autoref{eqn:sigma_x_action_on_homology}, it is clear that $\Sigma_X$ is the right choice of lift.
However, it is tempting to use $\Sigma^{pos}_X$ since the methods of \autoref{sec:expanding_pl_maps_on_the_line} below would imply that the potential constructed in \autoref{thm:existence_of_a_potential} would be, in fact, convex (and with controlled derivatives).
\end{remark}


\subsection{Constructing a potential}
	\label{ssec:constructing_a_potential}

\subsubsection{Eigen-homogenizations}
	\label{sssec:eigen_homogenizations}
Given an automorphism $f\colon \Sk(X)\to \Sk(X)$ written as a product of the involutions $\sigma_\bullet$, there are many possible lifts $F\colon E\to E$ to the total space of the bundle.
Each lift has a homogeneity matrix $M_F$ and determines possible eigen-homogenizations $v_F$ (column vector), with eigenvalue $\lambda_F$, satisfying $M_F v_F= \lambda_F  v_F$.
In particular, the class of $v_F$-homogeneous functions on $E$ (or more generally $(\bR^2)^3$) will be invariant under pullback by $F$ and rescaling by $\lambda_F$.
The interesting case is when $\lambda_F>1$, or $1>\lambda_F$ in which one applies the arguments to $F^{-1}$.

\begin{definition}[Potentials]
	\label{def:potentials}
	A \emph{potential} for the action of $F$ on $E$ is a $v_F$-homogeneous function $G\colon E\to \bR$ satisfying
	\[
	G(F(p)) = \lambda_F \cdot G(p)
	\]

\end{definition}

\subsubsection{Cocycles}
	\label{sssec:cocycles}
Fix a linear section of the projection $(\bR^2)^3\to \bR^3$, for example
\[
	\sigma(x,y,z) = \begin{bmatrix}
		x/2 & y/2 & z/2 \\
		-x/2 & -y/2 & -z/2
	\end{bmatrix}
\]
In order to relate the potentials on $E$ to the dynamics on the base, one can pull the potentials back using the section.
The dynamics will determine a cocycle:
\[
	F(\sigma(p)) - \sigma(f(p)) =: \bbV\cdot c_F(p)
\]
where $c_F(p)$ is row $3$-vector and $\bbV=\begin{bmatrix}
	1\\
	1
\end{bmatrix}$.
Defining $g(p):=G(\sigma(p))$ for a potential $G\colon E\to \bR$, this gives the functional equation
\[
	g(f(p)) = \lambda_F\cdot g(p)+ (c_F(p)\cdot v_F)
\]
where $v_F$ is, as above, the homogeneity of $G$ (hence a column vector).

\begin{theorem}[Existence of a potential]
	\label{thm:existence_of_a_potential}
	Suppose that $F\colon E\to E$ is a lift of the automorphism $f\colon \Sk(X)\to \Sk(X)$ and $M_F$ is the homogeneity matrix of $F$ in the fiber direction of $E$.
	If $M_F$ has a real eigenvector $v_F$ with eigenvalue $\lambda_F>1$, then there exists a unique continuous
	\[
		G\colon E \to \bR
	\]
	which is $v_F$-homogeneous and satisfies $G(F(p))=\lambda_F \cdot G(p)$.
\end{theorem}
\begin{proof}
	Let $\cC(E,v_F)$ denote the space of continuous functions on $E$, which are $v_F$-homogeneous and equipped with the distance function
	\[
		\dist(G_1,G_2):=\sup_{e\in E} |G_1(e)-G_2(e)|.
	\]
	It is a complete metric space, since the base $\Sk(X)$ is compact, and the functions have the same homogeneity in the fiber direction.
	Therefore the difference $|G_1(e)-G_2(e)|$ is independent of the choice of point in the fiber.
	Put differently, $\cC(E,v_F)$ has a natural structure of \emph{affine} space\footnote{Its non-emptiness can be checked using a partition of unity on $\Sk(X)$.} over $\cC(\Sk(X))$, the space of continuous functions on $\Sk(X)$.

	The transformation $\frac{1}{\lambda_F}F^*\colon\cC(E,v_F)\to \cC(E,v_F)$ is contracting distances by a factor of $\lambda_F$, so it has a unique fixed point -- the desired $G$.
\end{proof}

\begin{remark}[On compatibility]
	From the proof it is clear that one can also use a space of \Holder-continuous functions (with \Holder exponent depending on $\lambda_F$ and a choice of metric on $\Sk(X)$) to find that $G$ is \Holder.
	In fact, with the appropriate convexity properties of $G$, one would expect $G$ to be Lipschitz.

	In \autoref{sec:berkovich_spaces_and_k3_dynamics} below, we sketch a relation to Berkovich spaces and explain why one expects further convexity properties of the potential $G$ constructed in \autoref{thm:existence_of_a_potential}, assuming a correct choice of lift of the automorphism to the vector bundle.
\end{remark}

\subsubsection{Dynamics at infinity}
	\label{sssec:dynamics_at_infinity}
One can view the tropical projective line as $\bP^{1,tr}=\left\lbrace -\infty \right\rbrace \cup \bR\cup \left\lbrace +\infty \right\rbrace$ with the appropriate topology.
Then the action of the involutions $\sigma_\bullet$ extends, in fact to $(\bP^{1,tr})^3$ continuously.
The action of $\sigma_x$ exchanges the strata $\left\lbrace \pm\infty \right\rbrace\times (\bP^{1,tr})^2$, and the actions of $\sigma_y,\sigma_z$ preserve these two strata and give elliptic dynamics on each of them, as in \autoref{sec:the_elliptic_curve_case}.
However, the two elliptic dynamical systems will not be isomorphic (except for special choices of the parameters) so the automorphism $\sigma_x\sigma_y\sigma_z$ will mix the two elliptic factors.

It would be interesting to investigate this area-preserving system further.
For example, is there some finite Lebesgue measure set on the two planes that will be invariant?




\section{Berkovich spaces and K3 dynamics}
	\label{sec:berkovich_spaces_and_k3_dynamics}

The theory of Berkovich spaces was started in \cite{Berkovich_SpectralTheory}.
An introduction, with a view towards dynamics in dimension $1$, is in the notes of Baker \cite{Baker_IntroductionBerkovich}.

Our goal in this section is to develop the analogues of the constructions in \autoref{sec:general_properties_of_tropical_k3_automorphisms} in the setting of Berkovich spaces, where a more flexible and developed formalism is available.
The basic definitions are contained in \autoref{ssec:the_berkovich_analytification}, followed by a discussion of the skeleton from the point of view of Berkovich geometry in \autoref{ssec:berkovich_skeletons}.
Line bundles, which will be used when constructing the potentials, are discussed in \autoref{ssec:line_bundles_on_berkovich_spaces}.
The relation of the previous constructions with tropical geometry is discussed in \autoref{ssec:berkovich_spaces_and_tropicalization}.
After recalling the formalism of differential forms and currents on Berkovich spaces in \autoref{ssec:differential_forms_on_berkovich_spaces}, we construct the currents scaled by the dynamics in \autoref{ssec:the_currents_in_berkovich_dynamics}.

Note that the definition of currents and differential forms on Berkovich spaces makes use of tropicalizations.
Thus, many of the implicit constructions in this section can be made explicit in the setting of \autoref{sec:general_properties_of_tropical_k3_automorphisms}.

\paragraph{Notation}
Let $K$ be a field, complete with respect to a non-trivial non-archimedean absolute value $|-|$.
The valuation on $K$ is defined by $v\colon K^\times \to \bR$ as $v(x):=-\log(|x|)$.
Let $R$ denote the valuation ring, i.e. $R=\left\lbrace x \colon v(x)\geq 0 \right\rbrace$ and $\frakm$ the maximal ideal, defined by $\frakm = \left\lbrace x \colon v(x)>0 \right\rbrace$.
The residue field is $k:=R/\frakm$.
Assume that $K$ and $k$ have both characteristic zero.

The main example is $K=\bC((t))$, so that $R=\bC[[t]]\supset t\cdot \bC[[t]]=\frakm$ and $k=\bC$.
It is often convenient to assume that $K$ is also algebraically closed, in which case an example is the field of Puiseux series $K:=\bigcup_n \bC((t^{1/n}))$ (the residue field is still $\bC$).

All examples relevant for this paper will start with a quasi-projective variety $X$ inside a toric variety (e.g. $\bP^n$) which simplifies significantly the discussion.
In particular, the relevant line bundles will have a natural algebraic structure.


\subsection{The Berkovich analytification}
	\label{ssec:the_berkovich_analytification}

Let $A$ be a finitely generated $K$-algebra; its usual spectrum $\Spec(A)$ consists of all the prime ideals in $A$.

\begin{definition}[Berkovich analytification]
	\label{def:berkovich_analytification}
	The \emph{Berkovich analytification} $\Spec(A)^{an}$ is the set of all multiplicative seminorms on $A$ extending the fixed norm $|-|$ on $K$.
\end{definition}

\subsubsection{Some associated objects}
	\label{sssec:some_associated_objects}
For $x\in \Spec(A)^{an}$ with seminorm $|-|_x$, its kernel $\frakp_x=\left\lbrace a\in A\colon |a|_x=0 \right\rbrace$ is a prime ideal, so there is a natural map $\Spec(A)^{an}\to \Spec(A)$.

The quotient $A/\frakp_x$ is an integral domain, equipped now with a genuine norm also denoted $|-|_x$.
The fraction field of $A/\frakp_x$ is denoted $\kappa(x)$ and the completion of $\kappa(x)$ for the norm $|-|_x$ is denoted $\crH(x)$.

For a general scheme $X$ over $K$, the above construction glues along local charts and defines a space $X^{an}$.
The discussion below takes place on a fixed open set $U=\Spec(A)$ and all notions are local.

\subsubsection{Topology and functions on the Berkovich spectrum}
	\label{sssec:topology_and_functions_on_the_berkovich_spectrum}
Every element $f\in A$ determines a function
\begin{align*}
	|f| \colon U^{an}& \to \bR\\
	x &\mapsto |f|_x
\end{align*}
and the topology on $U^{an}$ is the smallest one for which the above maps are all continuous.

There is also a sheaf of ``holomorphic functions'' on $U^{an}$.
For this, note first that the algebra $A$ determines ``polynomial functions'' $a(x)\in \crH(x)$ on $U^{an}$, and their ratios determine ``rational functions'', with poles along closed subsets.
The holomorphic functions $\cO_{X^{an}}$ are then assignments $f(x)\in \crH(x)$ such that in some neighborhood of $x$, $f$ can be uniformly approximated by rational functions.

\subsubsection{Models}
	\label{sssec:models}
Suppose that $X$ is a scheme over $K$.
A \emph{model} of $X$ is an $R$-scheme $\cX$ which is flat and of finite type\footnote{One could also require properness} over $R$, normal and separated, and equipped with an isomorphism $\cX\times_R K\to X$.
The central fiber $\cX_0$ of the model is defined to be the reduction $\cX\times_R (R/\frakm)$ which is now a $k$-scheme.
If the model is not proper over $R$, the central fiber could be empty.

A model $\cX$ of $X$ has \emph{simple normal crossings}, or \emph{snc} for short, if: $\cX$ is proper and regular over $R$, the reduced central fiber $\cX_{0,red}$ is a union of divisors with simple normal crossings and irreducible (or empty) intersections.

The model $\cX$ is \emph{semistable} if it is proper over $R$, and the central fiber $\cX_0$ is reduced and snc.
Semistable models always exist, after possibly a finite extension $K\subset K'$ of the base field.



\subsection{Berkovich skeletons}
	\label{ssec:berkovich_skeletons}

Although Berkovich spaces have a large number of points, their homotopy type is manageable.
Fix a projective $K$-variety $X$, and assume for simplicity that $K=\bC((t))$.

\subsubsection{Clemens, or dual complex}
	\label{sssec:clemens_or_dual_complex}

The following construction is detailed in \cite[Appendix A]{KontsevichSoibelman_AffineStructures} or \cite[\S2.1]{BoucksomJonsson_LimitsOfVolumes}.

For a simple normal crossings model $\cX$, let $I_\cX$ denote the set of irreducible components of $\cX_{0,red}$ and for $i\in I_{\cX}$ let $D_i\subset \cX_{0,red}$ denote the corresponding divisor.
Assume that $\cX_0 = \sum_i b_i D_i$ with $b_i\in \bZ_{>0}$ giving the multiplicities.

For a subset $J\subset I_\cX$ define $D_J:=\cap_{j\in J} D_j$.
The \emph{Clemens complex}, or \emph{dual complex} $\Delta(\cX)$ is defined to be the simplicial complex determined by the set $I_\cX$ with a simplex $\Delta^J$ for every $J\subset I_\cX$ such that $D_J\neq \emptyset$.
The simplex has the following natural geometric realization:
\[
	\Delta^J := \left\lbrace w\in \bR^J_{>0} \colon \sum_{i\in J} b_i w_i = 1 \right\rbrace
\]

\subsubsection{Embedding the complex into the analytification}
	\label{sssec:embedding_the_complex_into_the_analytification}
One way to build points in $X^{an}$ is by constructing valuations on the function field of $X$ which are compatible with the valuation on $K$, in particular they assign valuation $1$ to the uniformizer $t\in R\subset K$.
Given such a valuation $v'$, the associated norm is $|f|_{v'}:=e^{-v'(f)}$.

Recall that $t=0$ determines $\cX_0\subset \cX$, so the multiplicity $b_i$ of a component $D_i\subset \cX_{0,red}$ is determined from
\[
	b_i = \ord_{D_i}(t)
\]
where $\ord_{D_i}$ is the valuation determined by the divisor $D_i\subset \cX$.
It follows that the valuation $v_i:=\frac{1}{b_i}\ord_{D_i}$ agrees with the valuation on $K$ and determines a point in $X^{an}$, called a \emph{divisorial point}.

In general, to a point $p\in \Delta(\cX)$ there is associated a \emph{monomial valuation} and monomial point as follows (see \cite[Prop. 2.4.4]{MustataNicaise_WeightsSkeleton} for details).
If $p\in \Delta^J\subset \Delta(\cX)$, it determines the subset $D_J=\cap_{j\in J} D_j \subset \cX$ and weights $w_j$ on $D_j$.
Let $x_j$ denote a defining equation of $D_j$ in the local ring of $\cX$ at $D_J$.
A function $f$ on $\cX$, regular near $D_J$, can be expanded in a series
\[
	f = \sum_{\beta \in \bZ^{|J|}_{\geq 0}} c_\beta x^{\beta}
\]
with $c_\beta$ non-vanishing on $D_J$, if non-zero.
Then the valuation associated to $p$ is defined as:
\[
	v_p(f) := \min_{\substack{\beta \in \bZ^{|J|}_{\geq 0}\\ 
					c_\beta\neq 0}}
					\quad 
			\sum_{j\in J} w_j\cdot \beta_j
\]
Note that the valuation is compatible with that on $K$, since for the power-series expansion of $t$, we will have $\beta_j=b_j$ so $v_p(t)=1$.


\subsection{Line bundles on Berkovich spaces}
	\label{ssec:line_bundles_on_berkovich_spaces}

For a more detailed discussion of line bundles and metrics on Berkovich spaces, see \cite[\S1]{ChambertLoir_HeightsMeasuresSurvey}.

\subsubsection{Line bundles and metrics}
	\label{sssec:line_bundles_and_metrics}
One can view a line bundle $L$ over $X^{an}$ as either a locally free rank $1$ sheaf of $\cO_{X^{an}}$-modules, or as its total space $L\to X^{an}$ with linear structure on each fiber.
Over a point $x\in X^{an}$, the fiber is isomorphic to $\bA^1_{\crH(x)}$.

A \emph{metric} on $L$ is a function $e^{-\phi}\colon L\to \bR$ such that $\phi$ is homogeneous of degree $(-1)$ in the sense of \autoref{def:homogeneity_for_potentials} (when restricted to $L^\times$, the line bundle without the zero section).

\subsubsection{Weil metric}
	\label{sssec:weil_metric}
The basic example is that coming from $\cO(-1)\to \bP^1$, which is explicitly $\bA^2\setminus 0 \to \bP^1$.
Taking analytifications, the \emph{Weil metric} is determined by
\begin{align*}
	\phi\colon &(\bA^2\setminus 0)^{an} \to \bR\\
		& (X_0:X_1) \mapsto -\log \max(|X_0|,|X_1|) = \min(-\log|X_0|,-\log|X_1|)
\end{align*}
This induces on the dual line bundle $\cO(-1)$ the Weil metric described in \cite[\S1.3.4]{ChambertLoir_HeightsMeasuresSurvey}.

The expression given above for $\phi$ in terms of $\min$ is made in order to emphasize the analogy with the corresponding tropical constructions.

\begin{remark}[On signs]
	\label{rmk:on_signs}
	There are a number of conventions about signs in the literature, and this paper follows only some of them.
	For line bundles, it is customary to take metrics to be functions on the total space of the form $e^{-\phi}$ so that the first Chern class is $\frac{\sqrt{-1}}{\pi}\del \delbar\phi$.
	This is compatible with the choice of tropicalization as $x\mapsto -\log |x|$ and going from valuations to norms via $|x|=e^{-v(x)}$.
\end{remark}



\subsection{Berkovich spaces and tropicalization}
	\label{ssec:berkovich_spaces_and_tropicalization}		

For a more thorough treatment of the constructions below, one can look at \cite[\S4]{Gubler_FormsCurrents} or \cite[\S3]{Payne_AnalytificationLimitTropicalizations}.

\subsubsection{Basics on tori}
	\label{sssec:basics_on_tori}
Recall that $\bG_m$ denotes the algebraic group of invertible elements in a field, i.e. for a field $\kappa$ one has $\bG_m(\kappa)=\kappa^\times$.
A \emph{torus}\footnote{Split torus, one should say.} $\bT$ is a product of several copies of $\bG_m$, i.e. $\bT:=\bG_m^n$.
One typically denotes by $N:=\Hom(\bG_m,\bT)$ the \emph{co-character lattice} and by $M:=\Hom(\bT,\bG_m)$ the \emph{character lattice}.
The two groups are naturally dual and each is isomorphic to $\bZ^n$, once an isomorphism $\bG_m^n\to \bT$ has been chosen.
When dependence on $\bT$ is important, the lattices will be denoted $M(\bT),N(\bT)$, and their extensions of scalars to a ring $R\supset Z$ by $M_R,N_R$.

\subsubsection{Morphisms, homogeneity, torsors}
	\label{sssec:morphisms_homogeneity}
A group homomorphism between two tori $f:\bT_1\to\bT_2$ is equivalent to the data of the induced map on co-character lattices $N_f:N(\bT_1)\to N(\bT_2)$, as a morphism of $\bZ$-modules.
Equivalently, it is determined by the dual map on character lattices.

Recall now that a torsor $S$ for a group $G$ is an algebraic variety with a $G$-action such that the map $G\times S \to S\times S$ given by $(g,s)\mapsto (g\cdot s,s)$ is an isomorphism.
Informally, a torsor for $G$ is like $G$, but without a choice of origin.

Given two torsors $S_1,S_2$ for the tori $\bT_1,\bT_2$, and morphisms $f:\bT_1\to \bT_2$ and $f_S:S_1\to S_2$ such that $f_S$ is $f$-equivariant, call $f_S$ \emph{homogeneous of degree $N_f$} where $N_f$ is the induced map on co-character lattices of $\bT_i$.

\subsubsection{Torsors and line bundles}
	\label{sssec:torsors_and_line_bundles}
The next discussion works in any category (continuous/algebraic/analytic).
For a line bundle $L\to X$ let $L^\times \to X$ denote the line bundle with the zero section removed.
Then there is a natural action of $\bG_m$ on $L^\times$ by scaling.
For each $x\in X$ the action makes the fiber $L^\times_x$ into a $\bG_m$-torsor.
One can check that equivalently, a family of $\bG_m$-torsors over $X$ is the same as a line bundle.
To see this, trivialize locally the torsors and use the definition of line bundles in terms of gluing maps on overlaps of charts by invertible functions.

More generally, a family of $\bT$-torsors $E\to X$ gives rise to a direct sum of line bundles.
Suppose now that $M\colon \bT\to \bT$ is a map of tori, induced by an endomorphism of the character lattice.
Given two families of $\bT$-torsors $E_i\to X$ with $i=1,2$, a map $F\colon E_1\to E_2$ preserving the fibers will be called \emph{homogeneous of degree $M$} if $F(t\cdot e_1)=M(t)\cdot F(e_1)$ for all $t\in \bT,e_1\in E_1$.

Note that the conditions can be expressed in terms of morphisms (i.e. as a diagram) so it is not necessary to speak of individual elements.
In the examples of interest, when $f\colon X\to X$ will be an automorphism, our torsors will be $E$ and $f^*E$.

\subsubsection{Tropicalizations for subvarieties of tori}
	\label{sssec:tropicalizations_for_subvarieties_of_tori}
For a torus $\bT$, its analytification admits a natural \emph{tropicalization map}
\begin{align*}
	\Trop\colon \bT^{an} & \to N(\bT)_\bR\\
	(x_1,\ldots, x_n) & \mapsto (-\log|x_1|,\ldots,-\log|x_n|)
\end{align*}
which can also be expressed canonically, using the identification $N\cong M^\vee$.

If $X\subset \bT$ is a subvariety, then the composition of $X^{an}\into \bT^{an}\xrightarrow{\Trop} {N_\bR}$ gives its tropicalization.
The image is a balanced (see \autoref{sssec:balancing_and_smoothness}) polyhedral complex.

\subsubsection{Toric varieties and extended tropicalization}
	\label{sssec:toric_varieties_and_extended_tropicalization}
Recall (see e.g. \cite{Fulton_Toric}) that a toric variety $Y$ is a variety equipped with an action of a torus $\bT$ and a dense open orbit isomorphic to $\bT$ inside $Y$.
The toric variety $Y$ is determined by a fan $\sigma$ in the co-character lattice $N(\bT)$; to denote the dependence of the variety on the fan, write $Y_\sigma$ for the toric variety.

Associated to the fan $\sigma$ is a partial compactification $N(\bT)_\sigma\supset N(\bT)$ and there is an extended tropicalization map
\[
	\Trop\colon Y_\sigma \to N(\bT)_\sigma
\]
compatible with the tropicalization of the torus $\bT\subset Y_\sigma$.
For example, the extended tropicalization of $\bP^1$ is $\left\lbrace -\infty \right\rbrace\cup \bR \cup \left\lbrace +\infty \right\rbrace$.

\subsubsection{Analytification as a limit of tropicalizations}
	\label{sssec:analytification_as_a_limit_of_tropicalizations}
If $X$ is quasi-projective, Payne \cite{Payne_AnalytificationLimitTropicalizations} showed that one can recover the Berkovich analytification as a projective (i.e. inverse) limit of tropicalizations.
Specifically, for each embedding into a toric variety $\iota:X\into Y_\sigma$, there is an associated tropical variety $\Trop(X,\iota)\subset N(\bT_\iota)_\sigma$.
Furthermore, for any two such embeddings $\iota_1,\iota_2$, there is a third $\iota_3$ that dominates them, as well as corresponding maps of tori that make the diagram equivariant.
The natural map
\[
	X^{an} \to \varprojlim \Trop(X,\iota)
\]
is then a homeomorphism, where $X^{an}$ is equipped with the topology from \autoref{sssec:topology_and_functions_on_the_berkovich_spectrum} and the right-hand side is given the projective limit topology.

\begin{remark}
	One can speak of tropicalizations without introducing Berkovich analytic spaces.
	Specifically, any $K$-valued point of $X\subset \bT$ has a natural tropicalization in $N(\bT)_\bR$ by taking $(-\log)$ of its coordinates.
	By considering all finite extensions $K'\supset K$ and $K'$-valued points, and taking the closure of the resulting set in $N(\bT)_\bR$ gives the same object as tropicalizing the analytification.
\end{remark}



\subsection{Differential forms on Berkovich spaces}
	\label{ssec:differential_forms_on_berkovich_spaces}

Lagerberg \cite{Lagerberg_Supercurrents} introduced a notion of super-forms on $\bR^n$ as a way to mimic the calculus of $(p,q)$-forms on $\bC^n$.
It was taken by Chamber-Loir \& Ducros \cite{ChamberLoirDucros} as a basis for a formalism of differential forms on Berkovich spaces.
An expository account of some of those results is in the notes of Gubler \cite{Gubler_FormsCurrents}.

\subsubsection{Super-forms on $\bR^n$}
	\label{sssec:super_forms_on_br_n}
Let $x_i$ be coordinates on $\bR^n$.
\emph{Super-forms} are expressions
\[
	\omega = \sum_{{I'},{I''}} \omega_{{I'},{I''}}(x) d'x_{I'}\wedge d''x_{I''}
\]
where $\omega_{{I'},{I''}}(x)$ are smooth functions and $d'x_{I'}=d'x_{i_1}\wedge\cdots\wedge d'x_{i_k}$ (and similarly for $d''x_{I''}$) is a formal expression obeying the rules of the exterior algebra.
In particular ${I'},{I''}\subset \left\lbrace 1\ldots n \right\rbrace$ and consist of distinct elements.
Equipped with a graded-symmetric wedge product, super-forms are a graded algebra; the bidegree of $\omega$ as expressed above is $(|I'|,|{I''}|)$ and its total degree is denoted $\deg \omega:=|I'|+|{I''}|$.

The operator $J$ is defined by $J(dx_i')=dx_i''$ and $J(dx_i'')=dx_i'$ and extended naturally to the algebra; note that $J^2=1$.
The differential operators are defined as derivations of the algebra via:
\begin{align*}
	d'f&:=\sum_i \frac{\partial f}{\partial x_i}dx_i' \quad \text{for functions}\\
	d'(\alpha\wedge\beta)& = (d'\alpha)\wedge \beta + (-1)^{\deg \alpha} \alpha \wedge (d'\beta)\\
	d'' &= J d' J\\
\end{align*}
Super-forms can be considered on open subsets of $\bR^n$, as well as all of $\bR^n$.

Under the ``tropicalization'' map $(\bC^\times)^n\to \bR^n$ defined by $x_i:=-\log |z_i|$, one can think of the identification $d'x_i=-d\log |z_i|$ and $d''x_i = d\Arg(z_i)$.
Note that this identification is not compatible for the types of forms of bidegree $(p,q)$ in complex geometry and in the sense described above.
However, one can identify the space of super-forms on $\bR^n$ with the space of $(\bS^1)^n$-invariant forms on $(\bC^\times)^n$, where $\bS^1\subset \bC^\times$ is the unit circle. 

The formalism of super-forms developed by Lagerberg allows for an integration theory admitting a Stokes formula, change of variables (for affine maps), and definitions of currents (i.e. super-forms with distributional coefficients) by a duality pairing with ordinary smooth super-forms.

\subsubsection{Positivity for super-forms}
	\label{sssec:positivity_for_super_forms}
Just like in complex geometry, there are several notions of positivity for $(p,p)$-super-forms, though they agree for $(1,1)$-forms.
A key fact \cite[Prop. 2.5]{Lagerberg_Supercurrents} is that a function $u$ is convex if and only if $d'd''u\geq 0$ in the sense of currents.
Conversely, for any closed positive $(1,1)$-current $T$ there exists a convex $u$ with $d'd''u=T$ (by \cite[Prop. 2.6]{Lagerberg_Supercurrents}).
In fact, H\"ormander-type $L^2$-estimates are developed in \cite{Lagerberg_L2}.
Furthermore, an intersection theory is also available for positive currents with continuous potentials; in particular, Monge--Amp\`ere operators of convex functions are well-defined.

It is useful to remark that tropical cycles in $\bR^n$ can be defined using the formalism of super-forms.
Namely, if $f$ is a PL function given as a minimization of finitely many affine functions $f:=\min_\alpha \xi_{\alpha}$, then $-d'd''f$ is the cycle of integration along the break locus of $f$.

\subsubsection{Differential forms on Berkovich spaces}
	\label{sssec:differential_forms_on_berkovich_spaces}
As discussed in \autoref{sssec:analytification_as_a_limit_of_tropicalizations}, the Berkovich space $X^{an}$ can be viewed as a projective limit of tropicalizations.
Using such maps $X^{an}\supset U \to \bR^n$ called charts, Chambert-Loir--Ducros \cite{ChamberLoirDucros} define differential forms on $X^{an}$ using an injective (i.e. direct) limit construction.
Note that currents on Berkovich spaces are therefore defined as projective (i.e. inverse) limits.
In particular, a current on $X^{an}$ gives one on any (extended) tropicalization.



\subsection{The currents in Berkovich dynamics}
	\label{ssec:the_currents_in_berkovich_dynamics}

Throughout, all analytifications are in the sense of Berkovich.

\subsubsection{Setup}
	\label{sssec:setup}
Suppose that $X$ is a projective K3 surface and $f\colon X\to X$ an automorphism, all defined over $K$.
Let $E\to X$ be a $\bT:=\bG^n_m$-bundle to which $f$ lifts as
\[
	\begin{tikzcd}
	E \arrow{r}{F}\arrow{d} & E\arrow{d}\\
	X \arrow{r}{f}          & X
	\end{tikzcd}
\]
and such that the homogeneity of $F$ (see \autoref{sssec:torsors_and_line_bundles}) is given by a morphism $F_\bT\colon\bT\to \bT$, i.e.
\[
	F(t\cdot p) = F_\bT(t)\cdot F(p)\quad \forall t\in \bT, \forall p\in E
\]
Recall that the torus morphism $F_\bT:\bT\to \bT$ is the same as a map $N_F\colon N(\bT)\to N(\bT)$ on the co-character lattice, or a dual morphism $M_F\colon M(\bT)\to M(\bT)$ on the character lattice.

\begin{definition}[Homogeneity for potentials]
	\label{def:homogeneity_for_potentials}
	Suppose that $E\to X$ is a torsor over $X$ for the torus $\bT$.
	A function $G\colon E^{an}\to \bR$ is called \emph{homogeneous of degree $\alpha\in M(\bT)$}, or \emph{$\alpha$-homogeneous} if for any $x\in X^{an}$ the map on the fiber:
	\[
		G\colon E^{an}_x \to \bR
	\]
	satisfies $G(t\cdot p)= G(p) + \log |\alpha(t)|$, where $t\in \bT(\crH(x))$ and $\alpha:\bT\to \bG_m$, and the norm of $\alpha(t)$ is coming from $\crH(x)$.

	We will abuse notation and allow $\alpha\in M(\bT)_\bR$ and not just $M(\bT)_\bZ$.
	To make sense of $\log|\alpha(t)|$, write $\alpha=\sum c_i \alpha_i$ with $\alpha_i\in M(\bT)_\bZ$ and $c_i\in \bR$ (where we used additive notation for characters).
	Then $\log|\alpha(t)|:= \sum_i c_i \log |\alpha_i(t)|$ where now $\alpha_i(t)\in \crH(x)$.
\end{definition}

\subsubsection{Behavior under pullback}
	\label{sssec:behavior_under_pullback}
Suppose that $G\colon E^{an}\to \bR$ is $\alpha$-homogeneous, and $F\colon E\to E$ is homogeneous for the morphism $M_F\colon M(\bT)\to M(\bT)$ on the character lattice and $F_\bT\colon \bT\to \bT$ on the torus.
Then the pullback $F^*G$ is $M_F(\alpha)$-homogeneous.
Indeed:
\begin{align*}
	F^*G(t\cdot p ) &= G(F(t\cdot p)) = G(F_\bT(t)\cdot F(p)) \\
	& = G(F(p)) + \log |\alpha(F_\bT(t))|\\
	& = F^*G(p) + \log |M_F(\alpha)(t)|
\end{align*}
by the definition of $M_F$.

\begin{theorem}[Existence of analytic potential]
	\label{thm:existence_of_analytic_potential}
	With the setup as in \autoref{sssec:setup}, suppose that $M_F$ acting on $M(\bT)_\bR$ has an eigenvector $v_F$ with eigenvalue $\lambda_F>1$.
	Then there exists a unique
	\[
		G\colon E^{an} \to \bR
	\]
	which is $\alpha_F$-homogeneous and satisfies $F^*G=\lambda_F \cdot G$.
\end{theorem}
\begin{proof}
	The proof is analogous to \autoref{thm:existence_of_a_potential}.
	Consider $\cC(E^{an},v_F)$, the space of continuous functions on $E^{an}$ which are $v_F$-homogeneous, equipped with the distance
	\[
		\dist(G_1,G_2) := \sup_{e\in E^{an}} |G_1(e)-G_2(e)|
	\]
	which makes into a complete metric space, in fact an affine space over $\cC(X^{an})$, the space of continuous functions on $X^{an}$.
	The strict contraction $\frac{1}{\lambda_F}F^*$ has a unique fixed point, which is the desired $G$.
\end{proof}
To obtain positivity properties of the current defined by the potential $G$, we need a strengthening of the above construction.

\begin{theorem}[Convergence to the potential]
	\label{thm:convergence_to_the_potential}
	Suppose that $P\colon E^{an}\to \bR$ is an $\alpha_P$-homogeneous function.
	Assume that the action of $M_F$ on $M(\bT)_\bR$ has an eigenvector $v_F$ with eigenvalue $\lambda_F>1$ and such that $\frac{1}{\lambda_F^n}M_F^n(\alpha_P)\to v_F$ exponentially fast as $n\to \infty$.

	Then the sequence of functions $\frac{1}{\lambda_F^n}(F^*)^nP$ converges, uniformly on compact sets of $E^{an}$, to the function $G$ constructed in \autoref{thm:existence_of_analytic_potential}.
\end{theorem}
\begin{proof}
	As a preliminary, fix a finite cover $X=\cup U_i$ by open sets and $s_i:U_i \to E$ sections.
	By abuse of notation, we will push forward open sets under the action of $f$ and sections under $F$.

	For each pair $i,j$, the sections are related under the dynamics by
	\[
		F_*s_i(x) = t_{i,j}(f(x)) s_j(f(x)) \quad \forall x\in f(U_i)\cap U_j
	\]
	for maps $t_{i,j}:f(U_i)\cap U_j \to \bT$.
	Note that the above identity can be expressed algebraically, without taking pointwise values, so in particular it makes sense after analytification.

	Set $P_n:=\frac{1}{\lambda_F^n}(F^*)^nP$ and let us now express its behavior under $\frac{1}{\lambda_F}F^*$:
	\begin{align}
		\label{eqn:P_n_recursion}
	\begin{split}
		P_{n+1}(s_i(x)) & =\frac{1}{\lambda_F} F^*P_n(s_i(x)) =\frac{1}{\lambda_F} P_n(F(s_i(x)))\\
		& = \frac{1}{\lambda_F} P_n(t_{i,j}(f(x))\cdot s_j(f(x)) )\\
		& = \frac{1}{\lambda_F}\Big( P_n(s_j(f(x))) + \log|\alpha_{P_n}(t_{i,j}(f(x)))| \Big)
	\end{split}
	\end{align}
	where $\alpha_{P_n}=\frac{1}{\lambda_F^n}M_F^n(\alpha_P)$ denotes the homogeneity of $P_n$.
	By shrinking the charts if necessary to precompact $U_i'\subset U_i$ (in the analytic topology) assume that the $t_{i,j}(x)$ vary in a bounded set in $\bT$.
	Now for the homogeneities, we can take $n_0$ such that $\forall n\geq n_0$ we have:
	\[
		-m \leq \log|\alpha_{P_n}(t_{i,j}(f(x)))| \leq m
	\]
	for some uniform $m>0$.
	Moreover, by the exponentially fast convergence of $\alpha_{P_n}$ to $v_F$, and the boundedness of $t_{i,j}(x)$, we can assume that there exists $C_1>0, \delta>1$ such that
	\[
		\Big|\log|\alpha_{P_n}(t_{i,j}(f(x)))| -  \log|\alpha_{P_{n-1}}(t_{i,j}(f(x)))|\Big| \leq \frac{C_1}{\delta^n}
	\]
	for all $x,i,j$.
	Defining now $d_n:=\sup_{i,x}|P_n(s_i(x))-P_{n-1}(s_i(x))|$ and using \autoref{eqn:P_n_recursion} gives:
	\begin{align*}
		d_{n+1} & \leq \frac{1}{\lambda_F}d_n + \frac{1}{\lambda_F}\cdot \frac{C_1}{\delta^n}
	\end{align*}
	and since $\lambda_F>1$, $\delta>1$, it is clear than $\sum_{i\geq 0} d_i$ converges absolutely.
	Therefore $P_n(s_i(x))$ converges uniformly, and so $P_n$ converges uniformly on compact sets to a function $G'$ whose homogeneity is clearly $v_F$.
	Since $\frac{1}{\lambda_F}F^*P_n = P_{n+1}$, it follows that $\frac{1}{\lambda_F}F^*G'=G'$ and by the uniqueness part of \autoref{thm:existence_of_analytic_potential}, it follows that $G'=G$.
\end{proof}

\begin{theorem}[Positivity of the current]
	\label{thm:positivity_of_the_current}
	Assume the K3 surface automorphism $f \colon X\to X$ is projective and the action of $f^*$ on $\Pic(X)$ is hyperbolic, i.e, there exists a unique up to scale $v\in \Pic(X)\otimes_\bZ\bR$ which is an eigenvector with eigenvalue $\lambda>1$.

	Then for the space $E$ from \autoref{sssec:a_general_construction}, there exists a choice of homogeneity $v$, or perhaps $-v$, with the following properties
	\begin{enumerate}
		\item The $v$-homogeneous function
		\[
			G\colon E^{an}\to \bR
		\]
		constructed in \autoref{thm:existence_of_analytic_potential} satisfies $d'd''G\geq 0$.
		\item There exists a closed positive current $\eta_v$ on $X^{an}$, obtained by pulling back $d'd''G$ from $E^{an}$ along local sections, such that $f^{*}\eta_v = \lambda \eta_v$.
	\end{enumerate}
\end{theorem}
\begin{proof}
	By the projectivity assumption of $f$ and $X$, there exists at least one very ample line bundle $L_a$ on $X$.
	Now $L_a$ carries a positive metric using the construction from \autoref{sssec:weil_metric}, written in the form $e^{-P_a}$ for a function $P_a\colon (L_a^\times)^{an}\to \bR$ satisfying $d'd''P_a\geq 0$.
	Express $L_a$ in the given basis of $\Pic(X)$ to obtain a map $E\to L_a^\times$ which is homogeneous of degree depending on the coordinates of $L_a$.
	Let $P$ be the pullback of $P_a$ to $E$, which still satisfies $d'd''P\geq 0$.

	There exists a positive rescaling of $P$ which ensures that its homogeneity $\alpha_P$ satisfies $\frac{1}{\lambda^n}M_F(\alpha_P)\to v$, or possibly $-v$; note that the ample cone is open in the real Picard group, so a generic choice of $L_a$ will have this property.
	\autoref{thm:convergence_to_the_potential} implies that $G$ is a uniform limit of functions $P_n$ satisfying $d'd''P_n\geq 0$, so it itself satisfies the same property.

	For part (ii), note that fixing local algebraic sections $s_i\colon X\to E$ on open sets $U_i$ allows us to define $\eta_v:=d'd''(s_i^*G)$.
	Let us check that $\eta_v$ is independent of the choice of $s_i$.
	Another section $s_i'$ will differ from $s_i$ by a map $r\colon U_i\to \bT$, where $E$ is trivialized on $U_i$ as $U_i\times \bT$ for a torus $\bT$.
	Then $d'd''(s_i^*G)$ and $d'd''((s_i')^*G)$ will differ by linear combinations of expressions of the form $d'd'' (a_j\cdot \log |r_j|)$ where $r_j$ are nowhere vanishing algebraic functions in $U_i$ (depending on $r$) and $a_j$ are real numbers depending on the homogeneity of $G$.
	Therefore the expressions $d'd''(a_j \cdot \log |r_j|)$ vanish (see e.g. \cite[4.6.5]{ChamberLoirDucros}), showing independence of the local sections.

	Finally the equation $f^*\eta_v=\lambda \eta_v$ follows from the same property for $G$.
\end{proof}

\begin{corollary}[Invariant measures]
	\label{cor:invariant_measures}
	With the setup as in \autoref{thm:positivity_of_the_current}, let $\eta_+$ be the current associated to $f$, and $\eta_-$ the one associated to $f^{-1}$.
	Then $\mu^{an}:=\eta_+\wedge\eta_-$ is an $f$-invariant positive measure on $X^{an}$.

	Furthermore, the measure is non-zero, and the currents satisfy $\eta_+^2=0=\eta_-^2$.
\end{corollary}
\begin{proof}
	The invariance of the measure follows because $f^*\eta_\pm = \lambda^{\pm 1}\eta_\pm$.
	For the second part, on intersections, recall that the currents $\eta_\pm$ can be written as $\eta_\pm = \lim_{n\to +\infty} \frac{1}{\lambda^n}\left( f^{\pm n} \right)^*c_1(L_a)$ where $c_1(L_a)$ is a $(1,1)$-form representing the first Chern class of an ample line bundle $L_a$.
	By \cite[6.4.3]{ChamberLoirDucros}, the integrals can be computed in cohomology:
	\[
		\int_{X^{an}}c_1(L_1)\wedge c_1(L_2) = \left( [c_1(L_1)]\cdot [c_1(L_2)] \right) \cap [X]
	\]
	for any two metrized line bundles, where $[c_1(L_i)]$ denotes the first Chern class of $L_i$ in the Chow ring (or Picard group in this case).

	A direct computation in the Picard group gives that 
	\begin{align*}
		\lim_{n\to +\infty} \frac{1}{\lambda^{2n}} \left( [c_1 \big(\left(f^{n}\right)^*L_a\big)] \cdot [c_1 \big(\left(f^{-n}\right)^*L_a\big)] \right) \cap [X] > 0\\	
		\lim_{n\to +\infty} \frac{1}{\lambda^{2n}} \left( [c_1 \big(\left(f^{n}\right)^*L_a\big)] \cdot [c_1 \big(\left(f^{n}\right)^*L_a\big)] \right) \cap [X] = 0\\
	\end{align*}
	Indeed, this follows since the corresponding classes $[\eta_\pm]$ in the Picard group, which are scaled by $f$ by $\lambda^{\pm 1}$, satisfy the same relations: $ [\eta_+]\cdot [\eta_-]>0$ and $[\eta_+]^2 = 0 = [\eta_-]^2$.

	Because $\eta_+\wedge \eta_+$ is a positive measure that integrates to zero, it follows that it is zero, and similarly for $\eta_-$.	
\end{proof}



\section{Examples of PL maps on tropical K3 surfaces}
	\label{sec:examples_of_pl_maps_on_tropical_k3_surfaces}

This section describes two classes of tropical K3 automorphisms.
The first class, discussed in \autoref{ssec:kummer_examples}, is concerned with the uniformly hyperbolic case coming from the Kummer construction.
At the opposite extreme, Rubik's cube examples in \autoref{ssec:rubik_s_cube_example} are the simplest perturbations of the case when automorphisms act by a finite group.

\subsection{Kummer examples}
	\label{ssec:kummer_examples}

A natural class of uniformly hyperbolic automorphisms of K3 surfaces comes from Kummer examples.
The starting point is a $2$-dimensional torus with an automorphism coming from a linear action on the universal cover.
The quotient by the involution $p\mapsto -p$ (plus a blowup of the resulting singular points in the algebraic case) gives a K3 surface with a uniformly hyperbolic automorphism.
The natural tropicalization of this example is described below.
These examples also appear in the paper of Spalding and Veselov \cite{SpaldingVeselov_TropicalCayley}, though they are natural and I was led to them independently of \cite{SpaldingVeselov_TropicalCayley}.

\subsubsection{The construction}
	\label{sssec:the_construction}

Let $E$ be an elliptic curve, tropical or in the usual complex-geometric sense.
It has a double cover to $\bP^1$, denoted $c:E\to \bP^1$.
The product $E\times E$ has three maps to $\bP^1$, written as
\begin{align*}
	C \colon & E\times E \to \bP^1 \times \bP^1 \times \bP^1\\
	& (a,b) \mapsto (c(a),c(b),c(a+b))
\end{align*}
and it is clear that $C(p)=C(-p)$, so that $E\times E$ is a double cover of its image $K:=C(E\times E)$, and there is a natural identification $K = E\times E / \pm 1$.

The product $E\times E$ is equipped with an action of $\GL_2\bZ$ which takes $(a,b)\in E\times E$ to the linear combination determined by the matrix, using the group structure on $E$.
Let $\tilde{\Gamma}(2)\subset \GL_2(\bZ)$ be the subgroup of matrices which are congruent to $\id \mod 2$.

The image $C(E)=K\subset \bP^1\times \bP^1\times \bP^1$ has three involutions (see \autoref{sssec:the_basic_example}) coming from the exchange of the double-sheeted covers $K\to \bP^1\times \bP^1$, denoted by $\sigma_x,\sigma_y,\sigma_z$.

\begin{proposition}[Semiconjugacy of the linear action]
	\label{prop:semiconjugacy_of_the_linear_action}
	The quotient $\tilde{\Gamma}(2)/\pm 1$ is freely generated by the three involutions
	\[
		\iota_x:= \begin{bmatrix}
			1 & 2 \\
			0 & -1 \\
		\end{bmatrix}
		\quad 
		\iota_y := \begin{bmatrix}
			-1 & 0 \\
			2 & 1 \\
		\end{bmatrix}
		\quad 
		\iota_z := \begin{bmatrix}
			-1 & 0 \\
			0 & 1 \\
		\end{bmatrix}.
	\]
	The representation $\tilde{\Gamma}(2) \to \tilde{\Gamma}(2)/\pm 1 \to \Aut(K)$ taking $\iota_x\to \sigma_x, \iota_y\to \sigma_y, \iota_z\to \sigma_z$ semiconjugates the action of $\tilde{\Gamma}(2)$ on $E\times E$ and that of the group generated by $\sigma_x,\sigma_y,\sigma_z$ on $K$.
\end{proposition}
\begin{proof}
	The claim about the freeness of the group generated by the involutions follows from the classical fact that the subgroup $\Gamma(2)\subset \SL_2\bZ$ of matrices congruent to $\id\mod 2$ is generated by $- \id$ and $\begin{bmatrix}
		1 & 2 \\
		0 & 1 \\
	\end{bmatrix}
	, \begin{bmatrix}
		1 & 0 \\
		2 & 1 \\
	\end{bmatrix}
	$.
	To check that the representation is indeed a semiconjugation, recall that $C(a,b)=(c(a),c(b),c(a+b))$.
	For $\iota_x$, the action gives
	\begin{align*}
		C(\iota_x(a,b)) & = C(a+2b, -b)= (c(a+2b),c(-b), c(a+b)) \\
		& = (c(a+2b),c(b),c(a+b))
	\end{align*}	
	so that indeed $\iota_x$ exchanges the two points in the image of $C$ for which the second and third positions are the same.
	The calculation for $\iota_y$ is similar.
	For $\iota_z$, the action gives
	\[
		C(\iota_z(a,b)) = C(-a,b) = (c(a),c(b),c(-a+b))
	\]
	so that the desired property holds.
\end{proof}

\subsubsection{The tropical Kummer K3}
	\label{sssec:the_tropical_kummer_k3}
The tropical incarnation of the above discussion starts with $E:=\bR/\bZ$, $\Trop(\bP^1) = \bR\cup \left\lbrace \pm \infty \right\rbrace$ and the map
\begin{align*}
	c\colon \bR/\bZ & \to \bR \subset \Trop(\bP^1)\\
	a & \mapsto 4\cdot \dist(a,\bZ) - 1
\end{align*}
where the normalization is chosen so that the image is $[-1,1]$.
Let
\[
	C : \bR^2/\bZ^2 \to \bR^3
\]
be the map constructed in \autoref{sssec:the_construction}.

\begin{figure}[ht]
	\centering
	\includegraphics[width=1.0\linewidth]{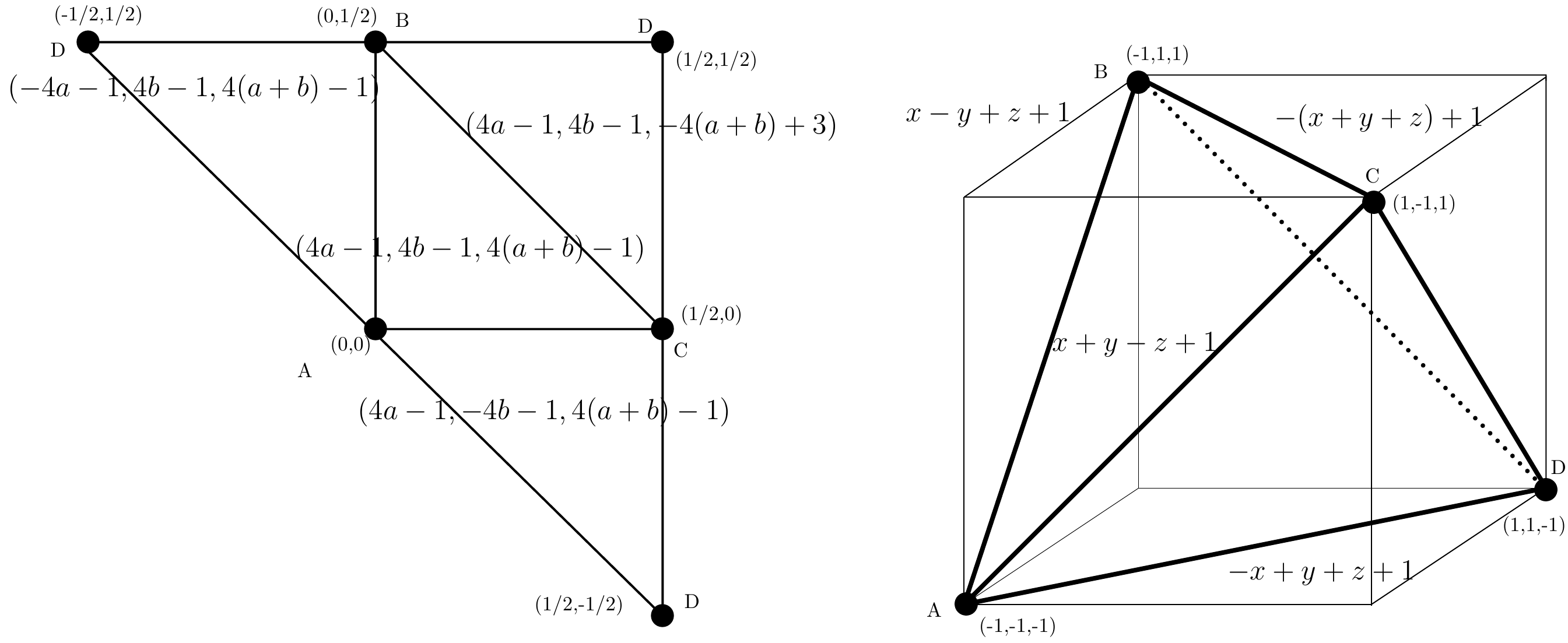}
	\captionsetup{width=0.95\linewidth}
	\caption{A fundamental domain in the $(a,b)$ plane $\bR^2$ for the $\bZ^2$ and $\pm 1$ action, and its image under the map to $\bR^3$.
	The domain is divided into $4$ triangles where the embedding is affine, with corresponding affine maps to $\bR^3$ indicated on each triangle.
	The face and equations of the image tetrahedron are:\newline
	$ABC: x+y-z+1=0 \quad \hfill  BCD:  -(x+y+z) + 1 = 0$
	$ABD: x-y+z+1= 0 \quad \hfill  ACD: -x+y+z + 1 = 0$
	}
	\label{fig:pictures_tropical_kummer}
\end{figure}

The image $K$ is a tetrahedron, determined by the equation
\[
	h(x,y,z) = \min(-x+y+z, x-y+z,x+y-z,-x-y-z) = -1
\]
To determine say the explicit form of the involution $\sigma_x$, rewrite
\[
	h(x,y,z) = \min(-x + \min(y+z,-y-z), x + \min(-y+z,y-z))
\]
so that the involution becomes (using \autoref{sssec:the_basic_example})
\begin{align*}
	x & \mapsto -x + \min(y+z,-y-z) - \min(y-z,-y+z)\\
	 & = -x + |y+z| - |y-z|
\end{align*}
Note that the entire picture scales, i.e. the level sets $h=-\alpha$ give a family of isomorphic tropical Kummer K3 surfaces, and the automorphisms commute with this scaling.

\begin{remark}
It is amusing to observe that the interior of the tropical Kummer K3 surface is isomorphic to the moduli space of representations of the free group on two letters into $\SU(2)$.
A conjugacy class in $\SU(2)$ is determined by an angle $\theta\in[0,\pi]$ and the isomorphism is given by taking a representation to the three conjugacy classes at the cusps, viewing the free group on two letters as the fundamental group of the thrice-punctured sphere.	
\end{remark}


\subsection{Rubik's cube example}
	\label{ssec:rubik_s_cube_example}

\begin{figure}[h]
	\centering
	\begin{tikzpicture}
	[x={(-0.448761cm, -0.552560cm)},
	y={(0.678244cm, 0.301157cm)},
	z={(-0.581893cm, 0.777163cm)},
	scale=0.750000,
	back/.style={loosely dotted, thin},
	edge/.style={color=black, thick},
	facet/.style={fill=gray,fill opacity=0.400000},
	vertex/.style={inner sep=1pt,circle,draw=black!25!black,fill=black!75!black,thick,anchor=base}]
%
%
\coordinate (1.90000, 2.50000, -2.50000) at (1.90000, 2.50000, -2.50000);
\coordinate (2.50000, 1.90000, -2.50000) at (2.50000, 1.90000, -2.50000);
\coordinate (-2.50000, 1.10000, 2.50000) at (-2.50000, 1.10000, 2.50000);
\coordinate (-2.50000, 2.50000, 1.10000) at (-2.50000, 2.50000, 1.10000);
\coordinate (-2.50000, -2.50000, -2.50000) at (-2.50000, -2.50000, -2.50000);
\coordinate (-2.50000, -2.50000, 2.50000) at (-2.50000, -2.50000, 2.50000);
\coordinate (-2.50000, 2.50000, -2.50000) at (-2.50000, 2.50000, -2.50000);
\coordinate (2.50000, -2.50000, -2.50000) at (2.50000, -2.50000, -2.50000);
\coordinate (2.50000, -2.50000, 1.50000) at (2.50000, -2.50000, 1.50000);
\coordinate (1.50000, -2.50000, 2.50000) at (1.50000, -2.50000, 2.50000);
\coordinate (1.50000, 1.10000, 2.50000) at (1.50000, 1.10000, 2.50000);
\coordinate (1.90000, 1.50000, 2.10000) at (1.90000, 1.50000, 2.10000);
\coordinate (2.50000, 1.50000, 1.50000) at (2.50000, 1.50000, 1.50000);
\coordinate (2.50000, 1.90000, 1.10000) at (2.50000, 1.90000, 1.10000);
\coordinate (1.90000, 2.50000, 1.10000) at (1.90000, 2.50000, 1.10000);
\draw[edge,back] (-2.50000, -2.50000, -2.50000) -- (-2.50000, -2.50000, 2.50000);
\draw[edge,back] (-2.50000, -2.50000, -2.50000) -- (-2.50000, 2.50000, -2.50000);
\draw[edge,back] (-2.50000, -2.50000, -2.50000) -- (2.50000, -2.50000, -2.50000);
\node[vertex] at (-2.50000, -2.50000, -2.50000)     {};
\fill[facet] (-2.50000, 2.50000, 1.10000) -- (1.90000, 2.50000, 1.10000) -- (1.90000, 2.50000, -2.50000) -- (-2.50000, 2.50000, -2.50000) -- cycle {};
\fill[facet] (1.90000, 2.50000, 1.10000) -- (-2.50000, 2.50000, 1.10000) -- (-2.50000, 1.10000, 2.50000) -- (1.50000, 1.10000, 2.50000) -- (1.90000, 1.50000, 2.10000) -- cycle {};
\fill[facet] (1.50000, 1.10000, 2.50000) -- (-2.50000, 1.10000, 2.50000) -- (-2.50000, -2.50000, 2.50000) -- (1.50000, -2.50000, 2.50000) -- cycle {};
\fill[facet] (2.50000, 1.50000, 1.50000) -- (2.50000, -2.50000, 1.50000) -- (1.50000, -2.50000, 2.50000) -- (1.50000, 1.10000, 2.50000) -- (1.90000, 1.50000, 2.10000) -- cycle {};
\fill[facet] (1.90000, 2.50000, 1.10000) -- (1.90000, 1.50000, 2.10000) -- (2.50000, 1.50000, 1.50000) -- (2.50000, 1.90000, 1.10000) -- cycle {};
\fill[facet] (1.90000, 2.50000, 1.10000) -- (1.90000, 2.50000, -2.50000) -- (2.50000, 1.90000, -2.50000) -- (2.50000, 1.90000, 1.10000) -- cycle {};
\fill[facet] (2.50000, 1.90000, 1.10000) -- (2.50000, 1.90000, -2.50000) -- (2.50000, -2.50000, -2.50000) -- (2.50000, -2.50000, 1.50000) -- (2.50000, 1.50000, 1.50000) -- cycle {};
\draw[edge] (1.90000, 2.50000, -2.50000) -- (2.50000, 1.90000, -2.50000);
\draw[edge] (1.90000, 2.50000, -2.50000) -- (-2.50000, 2.50000, -2.50000);
\draw[edge] (1.90000, 2.50000, -2.50000) -- (1.90000, 2.50000, 1.10000);
\draw[edge] (2.50000, 1.90000, -2.50000) -- (2.50000, -2.50000, -2.50000);
\draw[edge] (2.50000, 1.90000, -2.50000) -- (2.50000, 1.90000, 1.10000);
\draw[edge] (-2.50000, 1.10000, 2.50000) -- (-2.50000, 2.50000, 1.10000);
\draw[edge] (-2.50000, 1.10000, 2.50000) -- (-2.50000, -2.50000, 2.50000);
\draw[edge] (-2.50000, 1.10000, 2.50000) -- (1.50000, 1.10000, 2.50000);
\draw[edge] (-2.50000, 2.50000, 1.10000) -- (-2.50000, 2.50000, -2.50000);
\draw[edge] (-2.50000, 2.50000, 1.10000) -- (1.90000, 2.50000, 1.10000);
\draw[edge] (-2.50000, -2.50000, 2.50000) -- (1.50000, -2.50000, 2.50000);
\draw[edge] (2.50000, -2.50000, -2.50000) -- (2.50000, -2.50000, 1.50000);
\draw[edge] (2.50000, -2.50000, 1.50000) -- (1.50000, -2.50000, 2.50000);
\draw[edge] (2.50000, -2.50000, 1.50000) -- (2.50000, 1.50000, 1.50000);
\draw[edge] (1.50000, -2.50000, 2.50000) -- (1.50000, 1.10000, 2.50000);
\draw[edge] (1.50000, 1.10000, 2.50000) -- (1.90000, 1.50000, 2.10000);
\draw[edge] (1.90000, 1.50000, 2.10000) -- (2.50000, 1.50000, 1.50000);
\draw[edge] (1.90000, 1.50000, 2.10000) -- (1.90000, 2.50000, 1.10000);
\draw[edge] (2.50000, 1.50000, 1.50000) -- (2.50000, 1.90000, 1.10000);
\draw[edge] (2.50000, 1.90000, 1.10000) -- (1.90000, 2.50000, 1.10000);
\node[vertex] at (1.90000, 2.50000, -2.50000)     {};
\node[vertex] at (2.50000, 1.90000, -2.50000)     {};
\node[vertex] at (-2.50000, 1.10000, 2.50000)     {};
\node[vertex] at (-2.50000, 2.50000, 1.10000)     {};
\node[vertex] at (-2.50000, -2.50000, 2.50000)     {};
\node[vertex] at (-2.50000, 2.50000, -2.50000)     {};
\node[vertex] at (2.50000, -2.50000, -2.50000)     {};
\node[vertex] at (2.50000, -2.50000, 1.50000)     {};
\node[vertex] at (1.50000, -2.50000, 2.50000)     {};
\node[vertex] at (1.50000, 1.10000, 2.50000)     {};
\node[vertex] at (1.90000, 1.50000, 2.10000)     {};
\node[vertex] at (2.50000, 1.50000, 1.50000)     {};
\node[vertex] at (2.50000, 1.90000, 1.10000)     {};
\node[vertex] at (1.90000, 2.50000, 1.10000)     {};
\end{tikzpicture}
\begin{tikzpicture}%
	[x={(-0.448761cm, -0.552560cm)},
	y={(0.678244cm, 0.301157cm)},
	z={(-0.581893cm, 0.777163cm)},
	scale=0.750000,
	back/.style={loosely dotted, thin},
	edge/.style={color=black, thick},
	facet/.style={fill=gray,fill opacity=0.400000},
	vertex/.style={inner sep=1pt,circle,draw=black!25!black,fill=black!75!black,thick,anchor=base}]
%
%
\coordinate (1.90000, 2.50000, -2.50000) at (1.90000, 2.50000, -2.50000);
\coordinate (2.50000, 1.90000, -2.50000) at (2.50000, 1.90000, -2.50000);
\coordinate (-2.50000, 1.10000, 2.50000) at (-2.50000, 1.10000, 2.50000);
\coordinate (-2.50000, 2.50000, 1.10000) at (-2.50000, 2.50000, 1.10000);
\coordinate (-2.50000, -2.50000, -2.50000) at (-2.50000, -2.50000, -2.50000);
\coordinate (-2.50000, -2.50000, 2.50000) at (-2.50000, -2.50000, 2.50000);
\coordinate (-2.50000, 2.50000, -2.50000) at (-2.50000, 2.50000, -2.50000);
\coordinate (2.50000, -2.50000, -2.50000) at (2.50000, -2.50000, -2.50000);
\coordinate (2.50000, -2.50000, 1.50000) at (2.50000, -2.50000, 1.50000);
\coordinate (1.50000, -2.50000, 2.50000) at (1.50000, -2.50000, 2.50000);
\coordinate (0.90000, 1.10000, 2.50000) at (0.90000, 1.10000, 2.50000);
\coordinate (1.50000, 0.50000, 2.50000) at (1.50000, 0.50000, 2.50000);
\coordinate (2.50000, 0.50000, 1.50000) at (2.50000, 0.50000, 1.50000);
\coordinate (0.90000, 2.50000, 1.10000) at (0.90000, 2.50000, 1.10000);
\coordinate (1.90000, 2.50000, 0.10000) at (1.90000, 2.50000, 0.10000);
\coordinate (2.50000, 1.90000, 0.10000) at (2.50000, 1.90000, 0.10000);
\draw[edge,back] (-2.50000, -2.50000, -2.50000) -- (-2.50000, -2.50000, 2.50000);
\draw[edge,back] (-2.50000, -2.50000, -2.50000) -- (-2.50000, 2.50000, -2.50000);
\draw[edge,back] (-2.50000, -2.50000, -2.50000) -- (2.50000, -2.50000, -2.50000);
\node[vertex] at (-2.50000, -2.50000, -2.50000)     {};
\fill[facet] (1.90000, 2.50000, 0.10000) -- (1.90000, 2.50000, -2.50000) -- (-2.50000, 2.50000, -2.50000) -- (-2.50000, 2.50000, 1.10000) -- (0.90000, 2.50000, 1.10000) -- cycle {};
\fill[facet] (1.50000, 0.50000, 2.50000) -- (1.50000, -2.50000, 2.50000) -- (-2.50000, -2.50000, 2.50000) -- (-2.50000, 1.10000, 2.50000) -- (0.90000, 1.10000, 2.50000) -- cycle {};
\fill[facet] (0.90000, 2.50000, 1.10000) -- (-2.50000, 2.50000, 1.10000) -- (-2.50000, 1.10000, 2.50000) -- (0.90000, 1.10000, 2.50000) -- cycle {};
\fill[facet] (2.50000, 0.50000, 1.50000) -- (2.50000, -2.50000, 1.50000) -- (1.50000, -2.50000, 2.50000) -- (1.50000, 0.50000, 2.50000) -- cycle {};
\fill[facet] (2.50000, 1.90000, 0.10000) -- (2.50000, 0.50000, 1.50000) -- (1.50000, 0.50000, 2.50000) -- (0.90000, 1.10000, 2.50000) -- (0.90000, 2.50000, 1.10000) -- (1.90000, 2.50000, 0.10000) -- cycle {};
\fill[facet] (2.50000, 1.90000, 0.10000) -- (2.50000, 1.90000, -2.50000) -- (1.90000, 2.50000, -2.50000) -- (1.90000, 2.50000, 0.10000) -- cycle {};
\fill[facet] (2.50000, 1.90000, 0.10000) -- (2.50000, 1.90000, -2.50000) -- (2.50000, -2.50000, -2.50000) -- (2.50000, -2.50000, 1.50000) -- (2.50000, 0.50000, 1.50000) -- cycle {};
\draw[edge] (1.90000, 2.50000, -2.50000) -- (2.50000, 1.90000, -2.50000);
\draw[edge] (1.90000, 2.50000, -2.50000) -- (-2.50000, 2.50000, -2.50000);
\draw[edge] (1.90000, 2.50000, -2.50000) -- (1.90000, 2.50000, 0.10000);
\draw[edge] (2.50000, 1.90000, -2.50000) -- (2.50000, -2.50000, -2.50000);
\draw[edge] (2.50000, 1.90000, -2.50000) -- (2.50000, 1.90000, 0.10000);
\draw[edge] (-2.50000, 1.10000, 2.50000) -- (-2.50000, 2.50000, 1.10000);
\draw[edge] (-2.50000, 1.10000, 2.50000) -- (-2.50000, -2.50000, 2.50000);
\draw[edge] (-2.50000, 1.10000, 2.50000) -- (0.90000, 1.10000, 2.50000);
\draw[edge] (-2.50000, 2.50000, 1.10000) -- (-2.50000, 2.50000, -2.50000);
\draw[edge] (-2.50000, 2.50000, 1.10000) -- (0.90000, 2.50000, 1.10000);
\draw[edge] (-2.50000, -2.50000, 2.50000) -- (1.50000, -2.50000, 2.50000);
\draw[edge] (2.50000, -2.50000, -2.50000) -- (2.50000, -2.50000, 1.50000);
\draw[edge] (2.50000, -2.50000, 1.50000) -- (1.50000, -2.50000, 2.50000);
\draw[edge] (2.50000, -2.50000, 1.50000) -- (2.50000, 0.50000, 1.50000);
\draw[edge] (1.50000, -2.50000, 2.50000) -- (1.50000, 0.50000, 2.50000);
\draw[edge] (0.90000, 1.10000, 2.50000) -- (1.50000, 0.50000, 2.50000);
\draw[edge] (0.90000, 1.10000, 2.50000) -- (0.90000, 2.50000, 1.10000);
\draw[edge] (1.50000, 0.50000, 2.50000) -- (2.50000, 0.50000, 1.50000);
\draw[edge] (2.50000, 0.50000, 1.50000) -- (2.50000, 1.90000, 0.10000);
\draw[edge] (0.90000, 2.50000, 1.10000) -- (1.90000, 2.50000, 0.10000);
\draw[edge] (1.90000, 2.50000, 0.10000) -- (2.50000, 1.90000, 0.10000);
\node[vertex] at (1.90000, 2.50000, -2.50000)     {};
\node[vertex] at (2.50000, 1.90000, -2.50000)     {};
\node[vertex] at (-2.50000, 1.10000, 2.50000)     {};
\node[vertex] at (-2.50000, 2.50000, 1.10000)     {};
\node[vertex] at (-2.50000, -2.50000, 2.50000)     {};
\node[vertex] at (-2.50000, 2.50000, -2.50000)     {};
\node[vertex] at (2.50000, -2.50000, -2.50000)     {};
\node[vertex] at (2.50000, -2.50000, 1.50000)     {};
\node[vertex] at (1.50000, -2.50000, 2.50000)     {};
\node[vertex] at (0.90000, 1.10000, 2.50000)     {};
\node[vertex] at (1.50000, 0.50000, 2.50000)     {};
\node[vertex] at (2.50000, 0.50000, 1.50000)     {};
\node[vertex] at (0.90000, 2.50000, 1.10000)     {};
\node[vertex] at (1.90000, 2.50000, 0.10000)     {};
\node[vertex] at (2.50000, 1.90000, 0.10000)     {};
\end{tikzpicture}
	\caption{Typical pictures at the corners of a tropical K3 surface.}
	\label{fig:corners_of_a_tropical_K3}
\end{figure}
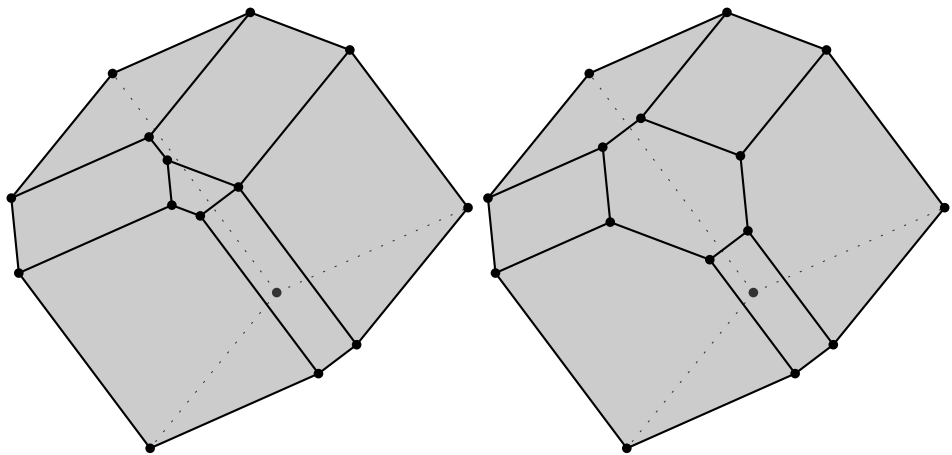

In general, the combinatorics of the skeleton of a tropical K3 surface can be quite involved, as \autoref{fig:corners_of_a_tropical_K3} illustrates.
We will describe one of the perhaps simplest examples of dynamics that's not of finite order.

The composition of two involutions, say $\sigma_x\cdot \sigma_y$ will be twisting the curves $z=const$ by amounts varying with $z$.
This is reminiscent of transformations applied to a Rubik's cube.

\subsubsection{The coefficients of the family}
	\label{sssec:the_coefficients_of_the_family}

It suffices to set the coefficients $c_{i,j,k}$ for the equation $h^{\circ}$ as described in \autoref{sssec:the_basic_example}.
All but the following\footnote{This is equivalent to omitting the corresponding linear form in the minimization} entries $c_{i,j,k}$ are set to $+\infty$:
\begin{eqnarray*}
	c_{-1,0,0} = 0 & \quad c_{0,-1,0} = 0 & \quad c_{0,0,-1} = 0 \\
	c_{1,0,0} = 0  & \quad c_{0,1,0} = 0  & \quad c_{0,0,1} = 0 \\
	               & c_{-1,-1,-1} = 1     &
\end{eqnarray*}
The corresponding function is $h^\circ = \min(-|x|,-|y|,-|z|, -(x+y+z)+1)$.

In the interval $t\in[0,\frac{1}{2}]$, the level set $h^\circ=t$ is a cube and the dynamics of the automorphism group is finite order.
Indeed, the reflections simply exchange two opposite faces, and act as reflections in the remaining four squares.

For $t\in[\frac{1}{2},1]$ the level set is a cube with a corner chopped off, the dynamics becomes non-trivial but it still has a domain where it is finite-order, but the size of the domain shrinks to zero as $t\to 1$.
Indeed the corner that has been cut off will be moved around non-trivially around the surface.
The domain where the group action is of finite order can be described as follows.
Consider the reflections $x\mapsto -x$ and similarly for $y,z$, and take the images of the cut-off corner under these reflections.
This will give a total of $8$ corners and the map will be of finite order on the union of rectangles on each face which do not intersect the corners.

For $t>1$ the chopped off corner affects the dynamics of the automorphism group on the entire sphere.

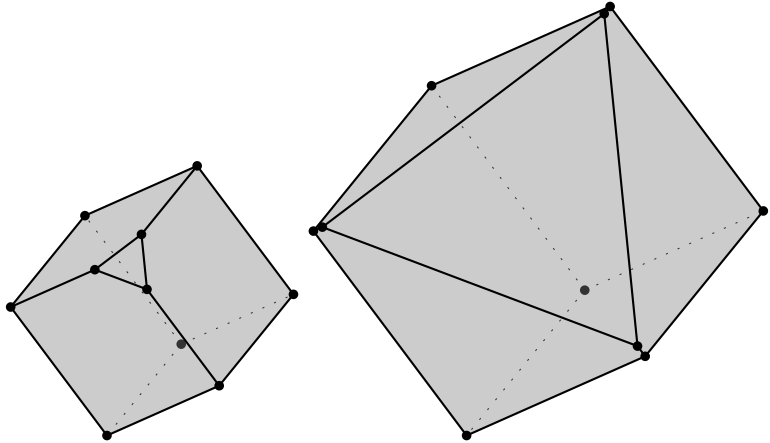
\begin{figure}[ht]
	\centering
\begin{tikzpicture}%
	[x={(-0.448761cm, -0.552560cm)},
	y={(0.678244cm, 0.301157cm)},
	z={(-0.581893cm, 0.777163cm)},
	scale=0.550000,
	back/.style={loosely dotted, thin},
	edge/.style={color=black, thick},
	facet/.style={fill=gray,fill opacity=0.400000},
	vertex/.style={inner sep=1pt,circle,draw=black!25!black,fill=black!75!black,thick,anchor=base}]
%
%
\coordinate (-2.00000, -2.00000, -2.00000) at (-2.00000, -2.00000, -2.00000);
\coordinate (-2.00000, -2.00000, 2.00000) at (-2.00000, -2.00000, 2.00000);
\coordinate (-2.00000, 2.00000, -2.00000) at (-2.00000, 2.00000, -2.00000);
\coordinate (-2.00000, 2.00000, 2.00000) at (-2.00000, 2.00000, 2.00000);
\coordinate (1.00000, 2.00000, 2.00000) at (1.00000, 2.00000, 2.00000);
\coordinate (2.00000, -2.00000, -2.00000) at (2.00000, -2.00000, -2.00000);
\coordinate (2.00000, -2.00000, 2.00000) at (2.00000, -2.00000, 2.00000);
\coordinate (2.00000, 1.00000, 2.00000) at (2.00000, 1.00000, 2.00000);
\coordinate (2.00000, 2.00000, -2.00000) at (2.00000, 2.00000, -2.00000);
\coordinate (2.00000, 2.00000, 1.00000) at (2.00000, 2.00000, 1.00000);
\draw[edge,back] (-2.00000, -2.00000, -2.00000) -- (-2.00000, -2.00000, 2.00000);
\draw[edge,back] (-2.00000, -2.00000, -2.00000) -- (-2.00000, 2.00000, -2.00000);
\draw[edge,back] (-2.00000, -2.00000, -2.00000) -- (2.00000, -2.00000, -2.00000);
\node[vertex] at (-2.00000, -2.00000, -2.00000)     {};
\fill[facet] (2.00000, 2.00000, 1.00000) -- (2.00000, 1.00000, 2.00000) -- (2.00000, -2.00000, 2.00000) -- (2.00000, -2.00000, -2.00000) -- (2.00000, 2.00000, -2.00000) -- cycle {};
\fill[facet] (2.00000, 2.00000, 1.00000) -- (1.00000, 2.00000, 2.00000) -- (-2.00000, 2.00000, 2.00000) -- (-2.00000, 2.00000, -2.00000) -- (2.00000, 2.00000, -2.00000) -- cycle {};
\fill[facet] (2.00000, 1.00000, 2.00000) -- (1.00000, 2.00000, 2.00000) -- (-2.00000, 2.00000, 2.00000) -- (-2.00000, -2.00000, 2.00000) -- (2.00000, -2.00000, 2.00000) -- cycle {};
\fill[facet] (2.00000, 2.00000, 1.00000) -- (1.00000, 2.00000, 2.00000) -- (2.00000, 1.00000, 2.00000) -- cycle {};
\draw[edge] (-2.00000, -2.00000, 2.00000) -- (-2.00000, 2.00000, 2.00000);
\draw[edge] (-2.00000, -2.00000, 2.00000) -- (2.00000, -2.00000, 2.00000);
\draw[edge] (-2.00000, 2.00000, -2.00000) -- (-2.00000, 2.00000, 2.00000);
\draw[edge] (-2.00000, 2.00000, -2.00000) -- (2.00000, 2.00000, -2.00000);
\draw[edge] (-2.00000, 2.00000, 2.00000) -- (1.00000, 2.00000, 2.00000);
\draw[edge] (1.00000, 2.00000, 2.00000) -- (2.00000, 1.00000, 2.00000);
\draw[edge] (1.00000, 2.00000, 2.00000) -- (2.00000, 2.00000, 1.00000);
\draw[edge] (2.00000, -2.00000, -2.00000) -- (2.00000, -2.00000, 2.00000);
\draw[edge] (2.00000, -2.00000, -2.00000) -- (2.00000, 2.00000, -2.00000);
\draw[edge] (2.00000, -2.00000, 2.00000) -- (2.00000, 1.00000, 2.00000);
\draw[edge] (2.00000, 1.00000, 2.00000) -- (2.00000, 2.00000, 1.00000);
\draw[edge] (2.00000, 2.00000, -2.00000) -- (2.00000, 2.00000, 1.00000);
\node[vertex] at (-2.00000, -2.00000, 2.00000)     {};
\node at (-2.00000, -2.00000, 2.00000) [above = 2pt]    {};
\node[vertex] at (-2.00000, 2.00000, -2.00000)     {};
\node[vertex] at (-2.00000, 2.00000, 2.00000)     {};
\node[vertex] at (1.00000, 2.00000, 2.00000)     {};
\node[vertex] at (2.00000, -2.00000, -2.00000)     {};
\node[vertex] at (2.00000, -2.00000, 2.00000)     {};
\node[vertex] at (2.00000, 1.00000, 2.00000)     {};
\node[vertex] at (2.00000, 2.00000, -2.00000)     {};
\node[vertex] at (2.00000, 2.00000, 1.00000)     {};
\end{tikzpicture}
\begin{tikzpicture}%
	[x={(-0.448761cm, -0.552560cm)},
	y={(0.678244cm, 0.301157cm)},
	z={(-0.581893cm, 0.777163cm)},
	scale=0.1750000,
	back/.style={loosely dotted, thin},
	edge/.style={color=black, thick},
	facet/.style={fill=gray,fill opacity=0.400000},
	vertex/.style={inner sep=1pt,circle,draw=black!25!black,fill=black!75!black,thick,anchor=base}]
%
%
\coordinate (-10.00000, -10.00000, 10.00000) at (-10.00000, -10.00000, 10.00000);
\coordinate (-10.00000, -10.00000, -10.00000) at (-10.00000, -10.00000, -10.00000);
\coordinate (-10.00000, 10.00000, 10.00000) at (-10.00000, 10.00000, 10.00000);
\coordinate (-10.00000, 10.00000, -10.00000) at (-10.00000, 10.00000, -10.00000);
\coordinate (10.00000, 10.00000, -10.00000) at (10.00000, 10.00000, -10.00000);
\coordinate (10.00000, -10.00000, -10.00000) at (10.00000, -10.00000, -10.00000);
\coordinate (10.00000, -10.00000, 10.00000) at (10.00000, -10.00000, 10.00000);
\coordinate (10.00000, 10.00000, -9.00000) at (10.00000, 10.00000, -9.00000);
\coordinate (10.00000, -9.00000, 10.00000) at (10.00000, -9.00000, 10.00000);
\coordinate (-9.00000, 10.00000, 10.00000) at (-9.00000, 10.00000, 10.00000);
\draw[edge,back] (-10.00000, -10.00000, 10.00000) -- (-10.00000, -10.00000, -10.00000);
\draw[edge,back] (-10.00000, -10.00000, -10.00000) -- (-10.00000, 10.00000, -10.00000);
\draw[edge,back] (-10.00000, -10.00000, -10.00000) -- (10.00000, -10.00000, -10.00000);
\node[vertex] at (-10.00000, -10.00000, -10.00000)     {};
\fill[facet] (-9.00000, 10.00000, 10.00000) -- (-10.00000, 10.00000, 10.00000) -- (-10.00000, -10.00000, 10.00000) -- (10.00000, -10.00000, 10.00000) -- (10.00000, -9.00000, 10.00000) -- cycle {};
\fill[facet] (-9.00000, 10.00000, 10.00000) -- (10.00000, 10.00000, -9.00000) -- (10.00000, -9.00000, 10.00000) -- cycle {};
\fill[facet] (-9.00000, 10.00000, 10.00000) -- (-10.00000, 10.00000, 10.00000) -- (-10.00000, 10.00000, -10.00000) -- (10.00000, 10.00000, -10.00000) -- (10.00000, 10.00000, -9.00000) -- cycle {};
\fill[facet] (10.00000, -9.00000, 10.00000) -- (10.00000, -10.00000, 10.00000) -- (10.00000, -10.00000, -10.00000) -- (10.00000, 10.00000, -10.00000) -- (10.00000, 10.00000, -9.00000) -- cycle {};
\draw[edge] (-10.00000, -10.00000, 10.00000) -- (-10.00000, 10.00000, 10.00000);
\draw[edge] (-10.00000, -10.00000, 10.00000) -- (10.00000, -10.00000, 10.00000);
\draw[edge] (-10.00000, 10.00000, 10.00000) -- (-10.00000, 10.00000, -10.00000);
\draw[edge] (-10.00000, 10.00000, 10.00000) -- (-9.00000, 10.00000, 10.00000);
\draw[edge] (-10.00000, 10.00000, -10.00000) -- (10.00000, 10.00000, -10.00000);
\draw[edge] (10.00000, 10.00000, -10.00000) -- (10.00000, -10.00000, -10.00000);
\draw[edge] (10.00000, 10.00000, -10.00000) -- (10.00000, 10.00000, -9.00000);
\draw[edge] (10.00000, -10.00000, -10.00000) -- (10.00000, -10.00000, 10.00000);
\draw[edge] (10.00000, -10.00000, 10.00000) -- (10.00000, -9.00000, 10.00000);
\draw[edge] (10.00000, 10.00000, -9.00000) -- (10.00000, -9.00000, 10.00000);
\draw[edge] (10.00000, 10.00000, -9.00000) -- (-9.00000, 10.00000, 10.00000);
\draw[edge] (10.00000, -9.00000, 10.00000) -- (-9.00000, 10.00000, 10.00000);
\node[vertex] at (-10.00000, -10.00000, 10.00000)     {};
\node[vertex] at (-10.00000, 10.00000, 10.00000)     {};
\node[vertex] at (-10.00000, 10.00000, -10.00000)     {};
\node[vertex] at (10.00000, 10.00000, -10.00000)     {};
\node[vertex] at (10.00000, -10.00000, -10.00000)     {};
\node[vertex] at (10.00000, -10.00000, 10.00000)     {};
\node[vertex] at (10.00000, 10.00000, -9.00000)     {};
\node[vertex] at (10.00000, -9.00000, 10.00000)     {};
\node[vertex] at (-9.00000, 10.00000, 10.00000)     {};
\end{tikzpicture}
	\caption{Tropical K3 surfaces in the Rubik's cube family.
	Left: level set in $[\frac{1}{2},1]$.
	Right: level set $>1$.
	The surfaces are not drawn to scale, i.e. in the $\bR^3$ that contains both, the one on the left is much smaller.}
	\label{fig:rubiks_cube}
\end{figure}

The plots of some unstable manifolds (in fact, iterates of some curves) in this family is displayed in \autoref{fig:unstable_manifolds}.

\begin{figure}[h]
	\centering
	\includegraphics[width=0.49\linewidth]{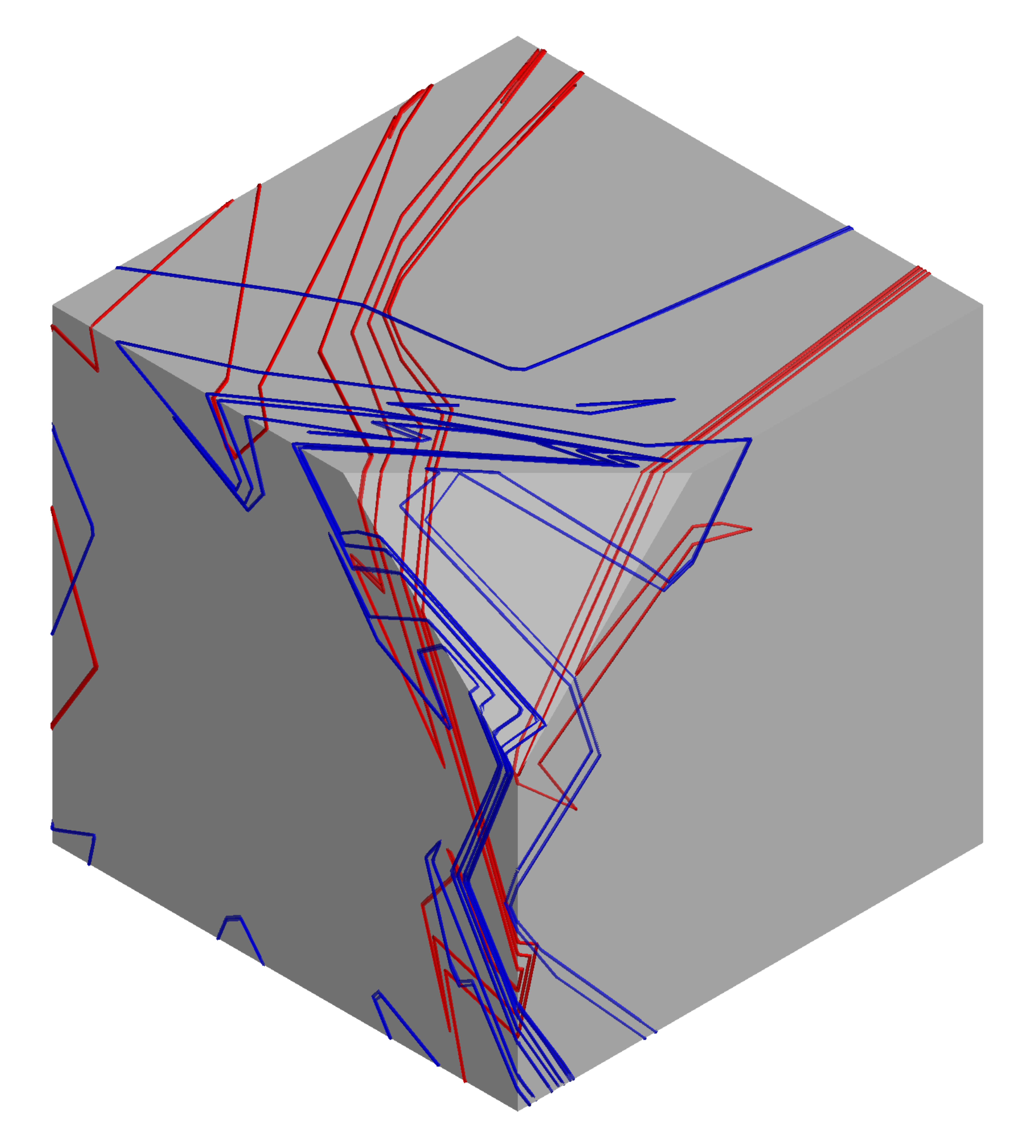}
	\includegraphics[width=0.48\linewidth]{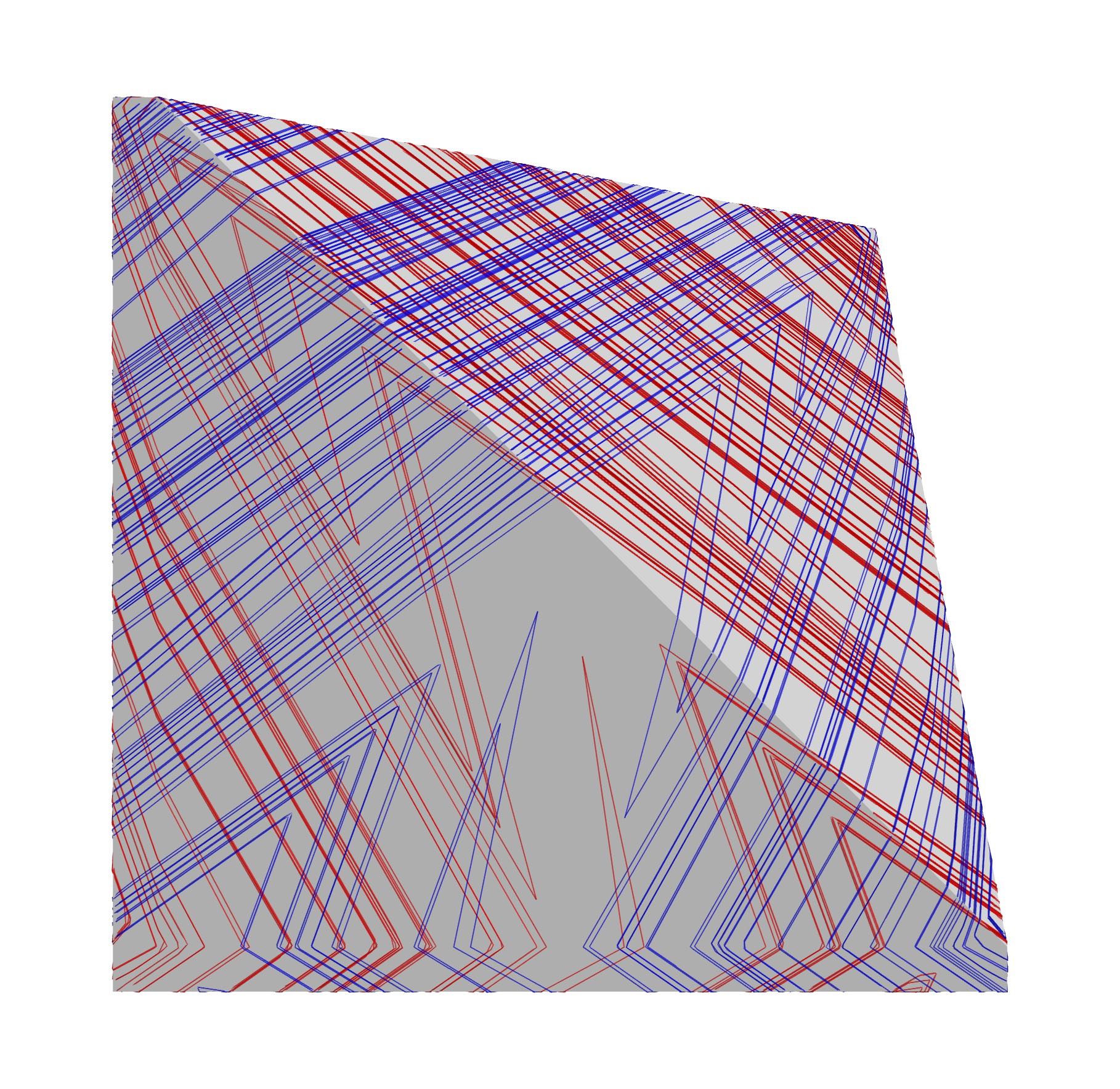}
	\caption{Forward (red) and backward (blue) iterates of the triangle face on the tropical K3.
	Left: for a small value of $t$.
	Right: for a large value of $t$.
	\autoref{fig:pictures_of_stable_unstable} contains further examples of iterates of the triangle face for a Rubik's cube example for large $t$.
	}
	\label{fig:unstable_manifolds}
\end{figure}



\section{Expanding PL maps on the line}
	\label{sec:expanding_pl_maps_on_the_line}

The basic object of study in this section will be PL maps on $\bR$ of a special form.
An analogous discussion is present in the work of Favre \& Rivera-Letelier \cite{Favre_RiveraLetelier} but in the context of Berkovich spaces.
The goal is to illustrate how elementary considerations, inspired by constructions in the Berkovich setting, can be used to study PL maps of the line.
It seems plausible that a similar elementary discussion, using the constructions in \autoref{sec:general_properties_of_tropical_k3_automorphisms}, could exist for tropical K3 automorphisms.

After discussing the formalism in \autoref{ssec:general_properties_of_pl_maps_of_the_line}, in \autoref{ssec:measures_from_potentials_1d} we construct the associated potential.
We can verify by elementary considerations that it is convex and hence leads to a measure.
The key to convexity is that there is a ``cone invariant by the dynamics'' (built from an appropriate class of convex functions).
In \autoref{ssec:properties_of_the_measure} we study some of the properties of the constructed measure; it is not always invariant by the dynamics since we are not in the setting of Berkovich spaces.
Finally, in \autoref{ssec:the_tent_map} we compute the relevant objects in the case of a uniformly hyperbolic tent map.

\paragraph{Notation}
We view $\bR$ as the set of equivalence classes of $(X_0:X_1)\in \bR^2$ modulo the equivalence $(X_0:X_1)\sim (X_0+t:X_1+t) \forall t\in \bR$, with the quotient map sending $(X_0:X_1)\mapsto X_1-X_0$.
The analogues of rational functions will then be PL maps $f:\bR\to \bR$ which can be lifted to PL maps $F:\bR^2\to \bR^2$ which are homogeneous of degree $n$, i.e. 
\[
	F( X_0+t:X_1+t ) = F(X_0:X_1) + n\cdot (t:t)
\]
and are expressible as
\begin{align}
\label{eqn:general_F_definition_1d}
\begin{split}
	F(X_0:X_1) & = (F_0:F_1) \text{ with} \\
	F_i(X_0:X_1) & = \min_j (a_{i}^j X_0 + b_{i}^jX_1 + c_{i}^j) \text{ and such that:}\\
	 & a_{i}^j + b_{i}^j = n  \quad a_{i}^j, b_i^j\in \bN
\end{split}
\end{align}

\subsection{General properties of PL maps of the line}
	\label{ssec:general_properties_of_pl_maps_of_the_line}

\subsubsection{Dehomogenization}
	\label{sssec:dehomogenization_f_1d}
The transformation in \autoref{eqn:general_F_definition_1d} descends to a map $f:\bR\to \bR$ by dehomogenizing the variables:
\begin{align*}
	f(x) & = f(X_1-X_0) \sim (0:F_1-F_0) \\
	 & = \min_j (a_{1}^j X_0 + b_{1}^jX_1 + c_{1}^j) - \min_j (a_{0}^j X_0 + b_{0}^jX_1 + c_{0}^j) \\
	 & = \min_j (b^j_{1}x + c_1^j) - \min_j (b_{0}^j x + c_0^j)
\end{align*}
where the last line is obtained by subtracting $n\cdot X_0$ from both minimizations, and using that $a_{i}+b_{i}=n$.

\subsubsection{Assumption on degree and coefficients}
	\label{sssec:assumption_on_degree_and_coefficients}

Different expressions for a PL function can have different degrees but give the same actual map $f:\bR \to \bR$.
For this reason, and to have the correct notion of degree, assume that in \autoref{eqn:general_F_definition_1d} there is at least one term of the form $n X_0 + c_\bullet $ occurring in one of $F_0, F_1$ (and possibly in both), since if such a term is missing one should subtract $(X_1:X_1)$ to get a lower-degree polynomial.
Similarly, assume there is at least one term of the form, $n X_1 + c_\bullet$, otherwise one could subtract $(X_0:X_0)$ to lower the degree.

In terms of the dehomogenized expressions in \autoref{sssec:dehomogenization_f_1d}, this means that at least one of $b_\bullet^\bullet$ must equal $n$, and at least one must equal $0$.

\subsubsection{Section and Cocycle}
	\label{sssec:section_and_cocycle}
There is also a natural section of the projection $\bR^2\to \bR$ sending $(X_0:X_1)\to X_1-X_0$ defined by $\sigma(x):= (-x/2:x/2)$.
Note however that this section does not commute with applying the transformations $F$ and $f$ on the respective spaces, namely we have
\[
	F(\sigma(x)) - \sigma(f(x)) = c(x)\cdot (1:1)
\]
since the two elements on the left-hand side agree when projected to $\bR$.
A direct calculation gives
\begin{align*}
	c(x) & = \frac{1}{2}\big(F_0(-\frac{x}{2}:\frac{x}{2}) + F_1(-\frac{x}{2}:\frac{x}{2})\big)\\
	& = \frac{1}{2}\big( \min_j ((b^j_{0} - a^j_{0})\frac{x}{2} + c^j_0) + \min_j(b^j_{1} - a^j_{1})\frac{x}{2} + c^j_1  \big)\\
	& =\frac{1}{2}\big( -n\cdot x + \min_j (b_0^jx + c_0^j) + \min_j (b_1^j x + c_1^j) \big)
\end{align*}
using that $a_0^j = n - b_0^j$.

\begin{definition}[Adapted Cone]
	\label{def:adapted_cone}
	Let $\cC$ denote the cone of concave functions on $\bR^2$ which are expressible as
	\[
		\phi = \min_{\alpha \in A} a_{\alpha} X_0 + b_{\alpha}X_1 + c_{\alpha} 
	\]
	with $a_{\alpha},b_\alpha\geq 0$.
\end{definition}

The following property will play a key role in obtaining convexity of potential functions.

\begin{proposition}
	\label{prop:preserving_cones_1d}
	Transformations $F:\bR^2\to \bR^2$ of the form in \autoref{eqn:general_F_definition_1d} preserve the cone $\cC$.
\end{proposition}
\begin{proof}
	It suffices to check the condition for affine linear functions of the form
	\[
		\xi(X_0:X_1) = a X_0 + b X_1 + c \text{ with }a,b\geq 0.
	\]
	since these are the extreme points of the cone $\cC$.
	Computing directly
	\begin{align*}
		\xi(F(X_0:X_1))& = a F_0 + b F_1 + c\\
		& = a \min_j (a_{0}^j X_0 + b_{0}^jX_1 + c_{0}^j) + b \min_j (a_{1}^j X_0 + b_{1}^jX_1 + c_{1}^j) + c
	\end{align*}
	However, note that all the $a_\bullet^j, b_\bullet^j$ are non-negative.
	Using the manipulations (for positive $a,b$)
	\begin{align*}
		a \cdot \min_i \xi_i(X_0,X_1) & = \min_i a \cdot \xi_i(X_0:X_1)\\
		\min_{i\in A_1} a\xi_i + \min_{j\in A_2} b\xi_j & = \min_{(i,j)\in A_1,A_2} (a\xi_i + b\xi_j)
	\end{align*}
	so the terms defining $\xi(F)$ can be grouped to give
	\[
		\xi(F(X_0:X_1)) = \min_{\alpha \in A} a_{\alpha} X_0 + b_{\alpha} X_1 + c_\alpha
	\]
	for some explicit $a_{\alpha},b_\alpha\geq 0$.
\end{proof}

\begin{proposition}[Existence of a Potential]
	\label{prop:existence_of_a_potential}
	Assuming $n\geq 2$, there exists a unique function $G:\bR^2\to \bR$ which simultaneously:
	\begin{itemize}
		\item Is homogeneous of degree $1$, i.e. 
		\[
			G(X_0 + t:X_1 +t ) =G(X_0:X_1) + t 	
		\]
		\item Belongs to the cone $\cC$ and the following limits exist:
		\begin{align*}
			\lim_{X_1-X_0\to \infty} G(X_0:X_1) - \min(X_0:X_1)\\
			\lim_{X_0-X_1\to \infty} G(X_0:X_1) - \min(X_0:X_1)
		\end{align*}
		and are finite real numbers.
		\item Under iteration of the dynamics, satisfies the functional equation
		\[
			G(F(X_0:X_1)) = n\cdot G(X_0:X_1)
		\]
	\end{itemize}
\end{proposition}

\begin{proof}
	Let $\cP$ denote the space of functions satisfying the homogeneity condition, belonging to the cone $\cC$, and such that for any $G\in \cP$ there exist constants $K\geq 0, c_0,c_1\in \bR$ such that
	\[
		G(X_0:X_1) = \begin{cases}
			X_1 + c'' &\text{ if }X_1-X_0\leq K\\
			X_0 + c' &\text{ if }X_0-X_1\leq K\\
		\end{cases}
	\]
	On $\cP$, define the distance
	\[
		\dist(G_1,G_2) := \sup_{t} |G_1(-t/2:t/2) - G_2(-t/2:t/2)|
	\]
	The completion $\ov{\cP}$ of $\cP$ for the above distance is equal to the space of functions satisfying the first two conditions in the proposition, since the functions are concave and the requirement of belonging to $\cC$ bounds the slopes.

	Define now the rescaling $R:\ov{\cP}\to \ov{\cP}$ by
	\[
		R(G) := \frac{1}{n} F^*G
	\]
	Provided we check that $R$ is well-defined on $\cP$ and is a strict contraction, the unique fixed point of $R$ in $\ov{\cP}$ will be the desired function $G$.

	The strict contraction is immediately checked using the homogeneity of $G$ plus the cocycle relation from \autoref{sssec:section_and_cocycle}:
	\begin{align*}
		\dist(F^*G_1,F^*G_2) & = \sup_t |G_1(F(\sigma(t))) - G_2(F(\sigma(t)))|\\
		 & = \sup_t |G_1(\sigma(f(t)) + c(t)(1:1)) - G_2(\sigma(f(t)) + c(t)(1:1)) |\\
		 & = \sup_t |G_1(\sigma(f(t))) - G_2(\sigma(f(t)))|\\
		 &\leq \dist(G_1,G_2)
	\end{align*}
	With the added factor of $\frac{1}{n}$, the strict contraction follows.	

	The homogeneity condition on $\cM$ is preserved by $R$, using the homogeneity of $F$:
	\[
		\frac{1}{n}G(F(X_0+t:X_1+t)) = \frac{1}{n}G(F(X_0:X_1)+n\cdot(t:t)) = \frac{1}{n}G(F(X_0:X_1)) + t.
	\]
	Since $F$ preserves the cone $\cC$ (\autoref{prop:preserving_cones_1d}) the only remaining property to check is the behavior for $|X_1-X_0|\gg 0$.
	We check it for $X_1-X_0\ll 0$, the argument for the other situation is analogous.

	Suppose therefore that $X_1-X_0\leq K'$, for a constant $K'$ depending only on $F$, such that in that regime we have:
	\[
		F(X_0:X_1) = (a_0 X_0 + (n-a_0)X_1 + c_0: a_1 X_0 + (n-a_1)X_1 + c_1)
	\]
	By the non-degeneracy assumption in \autoref{sssec:assumption_on_degree_and_coefficients}, at least one, and perhaps both $a_i$ vanish.
	We now have three possibilities:
	\begin{itemize}
		\item $a_0=a_1=0$ so $F(X_0:X_1)=(nX_1+c_0:nX_1 + c_1)$ and then
		\[
			\frac{1}{n}G(F(X_0:X_1)) = X_1 + \frac{1}{n}G(c_0:c_1)
		\]
		and the condition is satisfied.
		\item $a_0=0,a_1> 0$, so $F(X_0:X_1)=(nX_1 + c_0: a_1X_0 + (n-a_1)X_1 + c_1)$.
		This can be rewritten as
		\[
			F(X_0:X_1) = n(X_1:X_1) + (c_0:a_1(X_0-X_1) + c_1)
		\]
		and since we were assuming that $X_0\gg X_1$ we have
		\[
			\frac{1}{n}G(F(X_0:X_1)) = X_1 + \frac{1}{n}G(c_0:a_1(X_0-X_1)+c_1) = X_1 + \frac{1}{n}c_0 + c'
		\]
		where we used that for $X_0\ll X_1$ we have $G(X_0:X_1) = X_0 + c'$ for some $c'$.
		\item $a_0>0, a_1 = 0$, so reasoning as in the previous case:
		\[
			F(X_0:X_1) = n(X_1:X_1) + (a_0(X_0-X_1) + c_0:c_1)
		\]
		and now
		\[
			\frac{1}{n} G(F(X_0:X_1)) = X_1 + \frac{1}{n}G(a_0(X_0-X_1)+c_0:c_1) = X_1 + \frac{1}{n}c_1 + c''
		\]
		where we used that for $X_0\gg X_1$ we have $G(X_0:X_1) = X_1 + c''$ for some $c''$.
	\end{itemize}
\end{proof}

\begin{remark}
	The functions $G$ constructed in \autoref{prop:existence_of_a_potential} above is typically in the smaller space $\cP$ defined during the proof, i.e. it is affine linear at infinity.
	The only case when this cannot be guaranteed is when in the cases considered during the proof $a_0=0, a_1=1$ or vice-versa.
	When $a_0=0,a_1=0$ or $a_0=0,a_1>1$ there exists a constant $K>0$ such that the rescaling $R$ will preserve the functions which agree, up to a constant, with $\min(X_0,X_1)$ for $|X_0-X_1|\geq K$.
\end{remark}

\subsubsection{Pulled back potential}
	\label{sssec:pulled_back_potential}
Consider now the pull-back of the potential from $\bR^2$ to $\bR$ via the section $\sigma$ defined in \autoref{sssec:section_and_cocycle}:
\[
	g(x): = G(\sigma(x))
\]
Since $G$ is concave and $\sigma$ is linear, it follows that $g$ is also concave.
Under the dynamics, the behavior of $g(x)$ can be computed using the cocycle relation from \autoref{sssec:section_and_cocycle}:
\begin{align*}
	g(f(x)) & = G(\sigma(f(x))) = G(F(\sigma(x)) - c(x)\cdot (1:1))\\
	&= nG(\sigma(x)) - c(x) = n\cdot g(x) - c(x)
\end{align*}
So this gives the basic property of $g$:
\begin{align}
	\label{eqn:g_functional_equation}
	g(f(x)) = n\cdot g(x) - c(x)
\end{align}
which by homogeneity of $G$ is equivalent to $G(F(X)) = n\cdot G(X)$.

\subsection{Measures from potentials}
	\label{ssec:measures_from_potentials_1d}

Recall the basic identity \autoref{eqn:g_functional_equation}
\[
	g(f(x)) = n\cdot g(x) - c(x)
\]
Recall that $g$ and $c$ are concave functions, so their second derivatives give a measure.
Moreover the measure associated to $c(x)$ is atomic.

\subsubsection{Derivatives of a concave function}
	\label{sssec:derivatives_of_a_concave_function}
Recall that if $g$ is a concave function on an interval in $\bR$, then at any point $x$ it has left and right derivatives:
\[
	g(x) = g(x_0) + 
	\begin{cases}
		g'_l(x_0)(x-x_0) + o(x-x_0) &\text{ if }x\leq x_0\\
		g'_r(x_0)(x-x_0) + o(x-x_0) &\text{ if }x\geq x_0		
	\end{cases}
\]
and moreover the two derivatives agree, except at countably many points (see \cite[Cor. 1.1.6, Thm. 1.1.7]{Hormander_Convexity}.

\begin{definition}[Measure associated to a potential]
	\label{def:measure_associated_to_a_potential}
	Suppose that $g:I\to \bR$ is a concave function on an open interval $I\subset \bR$.
	Define the positive measure $\mu_g$ by the distributional identity:
	\[
		\mu_g(\phi) := -\int_I g(x) \phi''(x) dx
	\]
	for any smooth test function $\phi$, compactly supported in $I$.

	The definition will be used even if $g$ is not concave, although the result can be a general distribution.
\end{definition}

\begin{proposition}[Affine change of variables for measures]
	\label{prop:affine_change_of_variables_for_measures}
	Suppose that $g$ is a concave function on an interval $(y_1,y_2)$ and let $f:(x_1,x_2) \toisom (y_1,y_2)$ be an affine function of the form $f(x)=ax+b$.

	Then the measure associated to the pulled-back function $f^*g$ satisfies:
	\[
		\mu_{(f^* g)} = |a|\cdot f^{-1}_*(\mu_g)
	\]
	i.e. it is proportional to the measure $\mu_g$ pushed forward under $f^{-1}$.
\end{proposition}
\begin{proof}
	For the proof, assume that $x_1<x_2$, $y_1<y_2$, and so $a>0$, the orientation-reversing case being similar.
	For a compactly supported test function $\phi$ in $(x_1,x_2)$, its pull-back by $f^{-1}$ to $(y_1,y_2)$ satisfies
	\[
		(f^{-1})^*\phi(y) = \phi\left(\frac{y-b}{a}\right) \qquad \frac{d^2}{dy^2}\left( (f^{-1})^*\phi(y) \right) = \frac{1}{a^2}\phi''\left(\frac{y-b}{a}\right)
	\]
	so we can now compute:
	\begin{align*}
		\mu_{f^*g}(\phi)& = \int_{x_1}^{x_2} g(ax+b) \phi''(x) dx && y=ax+b\\
		& = \int_{y_1}^{y_2} g(y) \phi''\left(\frac{y-b}{a}\right)\frac{dy}{a}\\
		& = \int_{y_1}^{y_2} g(y) \left[ \frac{d^2}{dy^2}\left( (f^{-1})^*\phi(y)\right) \right] a\cdot dy\\
		& = a\cdot \mu_g((f^{-1})^*\phi)\\
		& = a\cdot f^{-1}_*(\mu_g)(\phi)
	\end{align*}
	which is the desired conclusion.
\end{proof}

The next proposition deals with the case when the function $f$ changes slope.
Since the statements are invariant under translations on the source or target, and under adding a constant to the functions under discussion, the result is stated assuming a simpler local form.

\begin{proposition}[Piecewise affine change of variables]
	\label{prop:piecewise_affine_change_of_variables}
	Suppose that $f:[-\ve,\ve]\to [-1,1] $ is given by
	\[
		f(x) = \begin{cases}
			f'_l(0) \cdot x, &\text{ if }x\leq 0\\
			f'_r(0) \cdot x, & \text{ if } x\geq 0
		\end{cases}
	\]
	and $g$ is a concave function on $[-1,1]$, with left and right derivatives at $x=0$ denoted $g'_l(0),g'_r(0)$ (see \autoref{sssec:derivatives_of_a_concave_function}).
	Then the distribution $\mu_{f^*g}$ is still a measure, satisfying
	\begin{align}
		\label{eqn:pullback_piecewise_affine}
		\mu_{f^*g} = f'_l(0) \cdot  f^{-1}_*\mu_{g|_{(-1,0)}} + f'_r(0) \cdot f^{-1}_*\mu_{g|_{(0,1)}} + c_{f,g}\cdot  \delta_{0}
	\end{align}
	where $\delta_0$ is the Dirac-delta mass at the origin, and $c_{f,g}$ is a constant given by
	\[
		c_{f,g} := 
		\begin{cases}
			f'_l(0)\cdot g'_l(0) - f'_r(0)\cdot g'_r(0) &\text{ if } f'_l(0), f'_r(0)\geq 0 \\
			\big( f'_l(0)-f'_r(0) \big)\cdot g'_l(0) & \text{ if } f'_l(0)\geq 0, f'_r(0)\leq 0\\
			\big( f'_l(0)-f'_r(0) \big)\cdot g'_r(0) & \text{ if } f'_l(0)\leq 0, f'_r(0)\geq 0 \\
			f'_l(0)\cdot g'_r(0) - f'_r(0)\cdot g'_l(0) &\text{ if } f'_l(0), f'_r(0)\leq 0 \\
		\end{cases}
	\]
	The cases depend on whether the mapping $f$ is bijective near $0$, or if it folds a neighborhood of zero onto one of the sides.
\end{proposition}

Note that while $\mu_g$ is invariant under the addition of a linear function to $g$, the measure associated to the pullback $\mu_{f^*g}$ \emph{is} changed when adding a linear function to $g$.

\begin{proof}
	The result is a straightforward check when $g$ is piecewise linear of the form:
	\[
		g(x) = \begin{cases}
			g'_l(0)\cdot x &\text{ if }x\leq 0\\
			g'_r(0)\cdot x &\text{ if }x\geq 0
		\end{cases}
	\]
	The result also follows when $g$ agrees with such a function near the origin.

	A general concave function $g$ can be approximated by a sequence $g_i$ for which the claim is already established, with the properties:
	\[
		\norm{g-g_i}_{C^0} \leq \ve_i, \quad \norm{\mu_{g} - \mu_{g_i}} \leq \ve_i, \quad \ve_i\to 0
	\]
	where the norm on measures is that of total variation.
	Indeed, $\mu_g$ decomposes into $\mu_{g,1} + \mu_g(\left\lbrace 0 \right\rbrace)\delta_0$ and $\mu_{g,1}$ can be approximated in total variation by measures which are not supported near $0$.

	Because $g_i\xrightarrow{C^0}g$ the same is true for the pullbacks $f^*g_i \xrightarrow{C^0} f^*g$ and then the same is true in the sense of distributions $\mu_{f^*g_i}\to \mu_{f^*g}$.
	Because $\norm{\mu_g-\mu_{g_i}}\to 0$, the right-hand sides in \autoref{eqn:pullback_piecewise_affine} for $g_i$ will also converge to those for $g$, giving the result.
\end{proof}

\subsection{Properties of the measure}
	\label{ssec:properties_of_the_measure}

Recall that:
\begin{align*}
	f(x) & = \underbrace{\min_j (b^j_{1}x + c_1^j)}_{f_1} - \underbrace{\min_j (b_{0}^j x + c_0^j)}_{f_0} && \text{see }\autoref{sssec:dehomogenization_f_1d} \\
	c(x) & = \frac{1}{2}\big( \min_j (b_1^j x + c_1^j) + \min_j (b_0^jx + c_0^j) \big) - \frac{n\cdot x}{2} && \text{see }\autoref{sssec:section_and_cocycle}\\
	 & = \frac{1}{2}(f_1 + f_0) - \frac{n\cdot x}{2}
\end{align*}
Assume that all the linear functionals above do occur in the minimization, and that there are points $x_1^{1} \leq \cdots \leq x_1^{k}$ such that
\begin{align}
	\label{eqn:break_points_of_f}
	f_1(x):=\min_{j=0\ldots k}(b_1^j\cdot x + c_1^j) = b_1^i\cdot x + c_1^i \quad \text{ if }x\in [x_1^{i},x_1^{i+1}]
\end{align}
where by convention $x_{1}^0:=-\infty, x^{k+1}_1:=+\infty$.
The slopes on the intervals come in decreasing order:
\[
	b_1^0 > b_1^1 > \cdots > b_1^k \quad b_1^\bullet \in \left\lbrace n,n-1,\ldots,0 \right\rbrace
\]

Similarly, let $x_0^{1}\leq \cdots \leq x_0^l$ be such that
\[
	f_0(x):=\min_{j=0\ldots l}(b_0^j\cdot x + c_0^j) = b_0^i\cdot x + c_0^i \quad \text{ if }x\in [x_0^{i},x_0^{i+1}]
\]
with the same conventions for the boundary points.
The slopes on the intervals again come in decreasing order:
\[
	b_0^0 > b_0^1 > \cdots > b_0^k \quad b_0^\bullet \in \left\lbrace n,n-1,\ldots,0 \right\rbrace
\]
The \emph{break points} $x_1^i, x_0^j$ are interspersed on the real axis, cut it up into intervals, and the slope of $f$ depends on the interval.

\begin{proposition}
	\label{prop:mu_c_has_atomic_mass_only_breakpoints}
	\leavevmode
	\begin{enumerate}
		\item The measure associated to $c$ via \autoref{def:measure_associated_to_a_potential} has mass only at the break points, where it is atomic:
	\[
		\mu_c\left(\left\lbrace x_\alpha^i \right\rbrace\right) = \frac{1}{2}(b_\alpha^{i-1}-b_\alpha^i) >0 \quad \text{ where }\alpha=0,1
	\]
	\item The total mass of $\mu_g$ is $1$.
	\end{enumerate}	
\end{proposition}
\begin{proof}
	Part (i) follows from the formula $c = \frac{1}{2}(f_1+f_0) -\frac{nx}{2}$.

	Part (ii) follows since $g$ is concave, and its slope at $\pm\infty$ is $\mp \frac 12$, using the definition of $\sigma$ and the behavior at infinity of $G$ from \autoref{prop:existence_of_a_potential}.
\end{proof}

\begin{proposition}
	\label{prop:characterization_of_monotonicity_1d}
	If there is a non-trivial interval on which $f$ has slope $\pm n$, i.e. maximal possible, then $f$ is monotonic (not necessarily strictly monotonic).
\end{proposition}
\begin{proof}
	Without loss of generality suppose that on some interval $f$ achieves slope $n$, the case of slope $-n$ being analogous, but with $f_1$ and $f_0$ interchanged.
	Since $f = f_1 - f_0$, the slopes of $f_{1/0}$ come in decreasing order, the maximal slope of $f_1$ is $n$ and the minimal slope of $f_0$ is $0$, it follows that all the break points of $f_0$ (if any at all) are to the left of the first break point of $f_1$, i.e. $x_0^l \leq x_1^1$.
	Moreover, the slope of $f_1$ in $(-\infty,x_1^1]$ is $n$ and the slope of $f_0$ in $[x_0^l,+\infty)$ is $0$.
	It follows that $f$ is non-strictly increasing, since the slopes of $f$ will always be non-negative.
\end{proof}

\begin{proposition}[No atoms except at the break points]
	\label{prop:no_atoms_except_at_the_break_points}
	Suppose that $f$ does not achieve slope $\pm n$ on any interval, for example $f$ is not monotonic.

	Then the only possible atoms of the measure $\mu_g$ associated to the potential $g$ from \autoref{sssec:pulled_back_potential} are at the points $x_\bullet^i$ defined in \autoref{eqn:break_points_of_f}.
\end{proposition}
\begin{proof}
	Pick a point $x_0$, which is not a break point, and such that its mass $\mu_g(\left\lbrace x_0 \right\rbrace)>0$ is largest among all atoms of $\mu_g$ and in the neighborhood of $x_0$, $f$ has slope $b$ with $|b|<n$.

	Recall the functional equation (\autoref{eqn:g_functional_equation})
	\[
		g(f(x)) = n\cdot g(x) - c(x)
	\]
	which translates into the equality of measures (by applying $-\partial_x^2$ to both sides):
	\[
		\mu_{f^*g} = n\cdot \mu_g - \mu_c.
	\]
	By \autoref{prop:mu_c_has_atomic_mass_only_breakpoints} $\mu_c$ is a sum of delta-masses at the break points, while by the change of variables formula from \autoref{prop:affine_change_of_variables_for_measures} we have
	\[
		\mu_{f^*g}(\left\lbrace x_0 \right\rbrace) = |b|\cdot \mu_g(\left\lbrace f(x_0) \right\rbrace)
	\]
	From the assumption that $|b|<n$ and that $\mu_g(\left\lbrace x_0 \right\rbrace)$ is maximal among atoms of $\mu_g$, it follows that $x_0$ cannot be an atom of $\mu_g$.
\end{proof}

\subsection{The tent map}
	\label{ssec:the_tent_map}

The tent map on the interval $[-1/2,1/2]$ can be expressed as
\[
	f(x) = \min \left( 2x+\frac{1}{2}, -2x + \frac{1}2 \right) = -2|x|+\frac{1}{2}
\]
It fixes $-1/2$, expands by a factor of $2$ the segment $[-1/2,0]$ to $[-1/2,1/2]$, then turns around and expands the segment $[0,1/2]$ to $[1/2,-1/2]$ the other way.

The homogenization of the tent map can be done in the following steps, taking $x:=X_1-X_0$ as the inhomogeneous variable:
\begin{align*}
	F(X_0:X_1) & \approx (0 : \min( 2(X_1-X_0) + 1/2, -2(X_1-X_0) + 1/2) ) \quad \text{ add }2(X_0+X_1) \\
	& \approx \left( 2X_0 + 2X_1 : \min \left( 4X_0 + \frac{1}{2}, 4 X_1 + \frac{1}{2} \right) \right)
\end{align*}
The value of $F$ depends on two possibilities:
\begin{itemize}
	\item If $X_1 \leq X_0$ (which corresponds to $x\leq 0$ in the inhomogeneous setting) then
	\[
		F(X_0:X_1) = (2X_0 + 2X_1 : 4X_1+1/2)
	\]
	\item If $X_1 \geq X_0$ (which corresponds to $x\geq 0$ in the inhomogeneous setting) then
	\[
		F(X_0:X_1) = (2X_0 + 2X_1: 4X_0 + 1/2)
	\]
\end{itemize}

\subsubsection{Finding the potential function}
	\label{sssec:finding_the_potential_function}

We would like to find a function $G$ satisfying
\begin{align}
	\label{eqn:G_functional_equation}
	G(F(X_\bullet)) = 4 G(X_\bullet)
\end{align}



There are two simple cases:
\begin{enumerate}
	\item Suppose that $X_0 \geq X_1 + \frac{1}{2}$, which is equivalent to $2X_0 + 2X_1 \geq (4X_1 + \frac{1}{2}) + \frac{1}{2} $, so that the condition is stable under applying $F$.
	Then one can check that 
	\[
		F^n(X_0:X_1) = (E_n: 4^nX_1 + \frac{1}{2}(4^0 + \cdots + 4^{n-1}))
	\]
	and therefore 
	\[
		G(X_0:X_1):= X_1 + \frac{1}{6}	
	\]
	will satisfy the recurrence in \autoref{eqn:G_functional_equation}.
	\item Suppose that $X_0 \leq X_1 - \frac{1}{2}$, which implies that $2X_0 + 2X_1 \geq (4X_0 + \frac{1}{2}) + \frac{1}{2} $, so that after one iterate we land in the previous situation.
	Then
	\begin{align*}
		G(X_0:X_1)&:= \frac{1}{4}G(2X_0+2X_1:4X_0+1/2) \\
		& = \frac{1}{4}(4X_0+1/2 + 1/6)\\
		& = X_0 + \frac{1}{6}
	\end{align*}
\end{enumerate}

\begin{figure}[ht]
	\centering
	\includegraphics[width=0.8\linewidth]{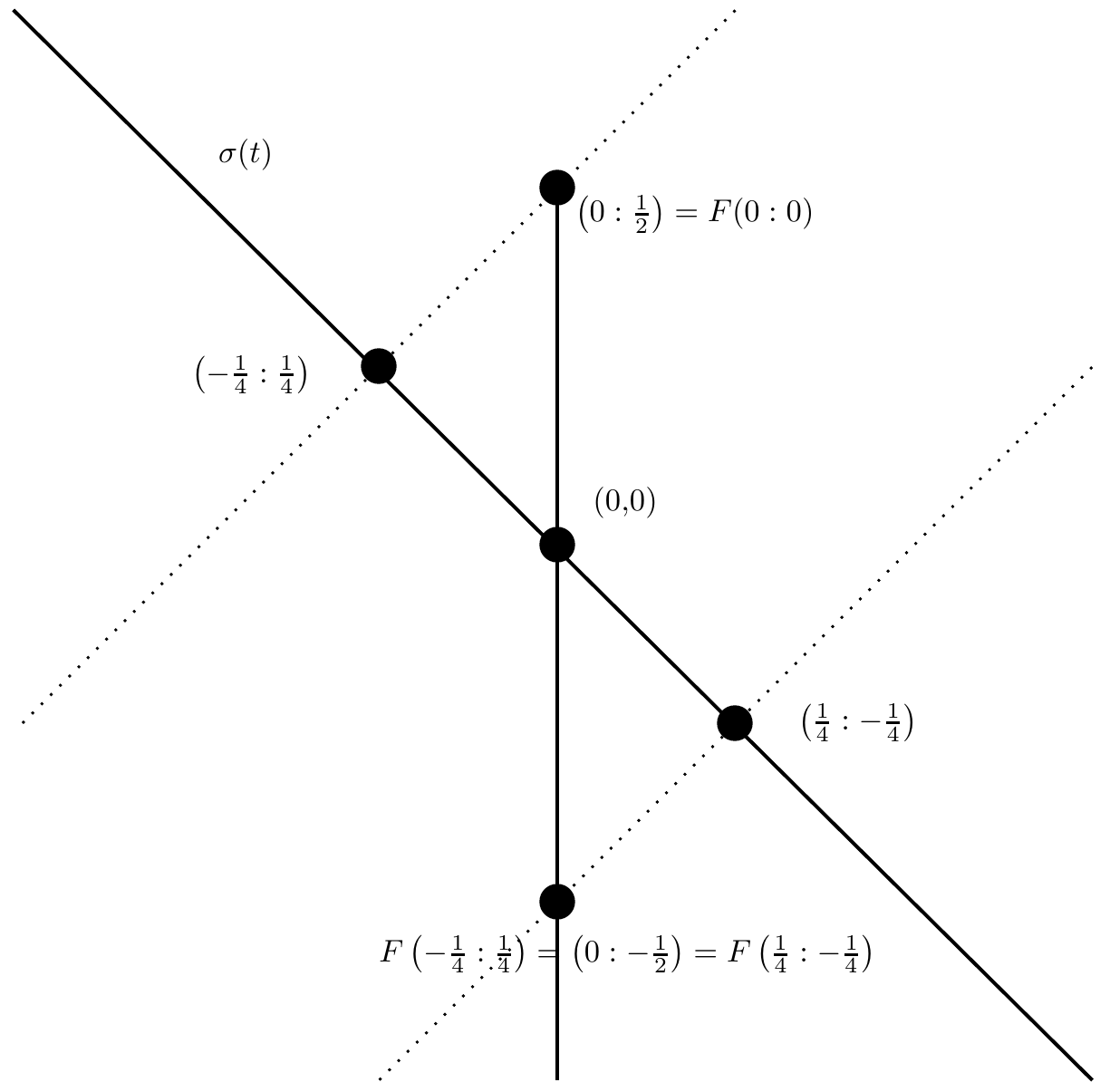}
	\caption{The lifted tent map, and its action on the section $\sigma$.}
	\label{fig:pictures_lifted_tent_map}
\end{figure}

\subsubsection{Computing $G$ in the strip}
	\label{sssec:computing_}
Consider now the map $\sigma(x) = (-x/2:x/2)$ viewed as a section from the base $\bR$ to $\bR^2$, under the projection map $(X_0:X_1)\to X_1-X_0$.
Computing $G$ on the image of $\sigma$ determines it completely, due to the homogeneity conditions $G$ satisfies.
For the dynamics, we have:
\begin{align*}
	F(-x/2:x/2) =
	\begin{cases}
		(0:+2x + 1/2), &\text{ if } x\leq 0\\
		(0:-2x + 1/2), &\text{ if } x\geq 0		
	\end{cases}
\end{align*}
The cohomological condition (see \autoref{sssec:section_and_cocycle}) is then:
\begin{align}
	\label{eqn:tent_map_cohomology}
	F(\sigma(x)) - \sigma(f(x)) = (-|x|+1/4)(1:1)
\end{align}
i.e. $c(x) = -|x| + \frac{1}{4}$.

\begin{proposition}
	\label{prop:computing_g}
	The function $G$ restricted to the line $(-x/2:x/2)$ is equal to:
	\[
	G(-x/2:x/2)=
	\begin{cases}
		-\frac{1}{2}|x| + \frac{1}{6} &\text{ if } |x|\geq \frac{1}{2}\\
		-\frac{1}{2}x^2 + \frac{1}{24} & \text{ if } |x|\leq \frac{1}{2}
	\end{cases}
	\]
\end{proposition}
\begin{proof}
	By uniqueness of the function $G$, it suffices to check that it satisfies the functional equation in \autoref{eqn:G_functional_equation}.
	By homogeneity and the condition at $|x|\gg 0$ already being satisfied, it suffices in fact to check that $g(x):=G(\sigma(x))$ satisfies the functional equation from \autoref{eqn:g_functional_equation}.
	Recall that $c(x)= -|x| + \frac{1}{4}$ so we have
	\begin{align*}
		g(f(x)) & = -\frac{1}{2}(-2|x|+\frac{1}{2})^2 + \frac{1}{24}\\
		& = -2 x^2 + |x| - \frac{1}{12}
		\intertext{ while }
		4g(x) - c(x) & = 4(-\frac{1}{2}x^2 + \frac{1}{24}) - (-|x| + \frac{1}{4})\\
		& = -2 x^2 + |x| - \frac{1}{12}
	\end{align*}
	which is the desired identity.
\end{proof}

\begin{corollary}
	\label{cor:maximal_entropy_for_tent_map}
	The measure $\mu_g$ for the tent map is Lebesgue measure on $[-\frac{1}{2},\frac{1}{2}]$ and hence the measure of maximal entropy.
	Indeed, the tent map restricted to this interval is semi-conjugated to the Bernoulli shift on $\left\lbrace 0,1 \right\rbrace^{\bN}$, and Lebesgue measure corresponds to the uniform measure $\left( \frac{1}{2}\delta_0 + \frac{1}{2}\delta_1 \right)^{\bN}$.
\end{corollary}

\begin{remark}
	The quadratic function $-\frac{1}{2}x^2 + \frac{1}{24}$ satisfies the condition in \autoref{eqn:g_functional_equation} on all of $\bR$, but it grows too fast at infinity.
\end{remark}

\subsubsection{Some further questions}
	\label{sssec:some_further_questions}

Here are some natural questions that arise from the above discussion.
\begin{enumerate}
	\item Following Favre--Rivera-Letelier \cite{Favre_RiveraLetelier}, one can define the local degree at a point as the slope of $f(x)$.
	Then the non-atomic part of $\mu_g$ should be invariant on the set of points for which at the preimages, the degree adds up to $n$ (see \cite[\S5.2]{Favre_RiveraLetelier}).
	\item More generally, investigate further the atomic and AC parts of the measure constructed measure $\mu$.
	\item There is a natural class of post-critically finite maps in the PL setting.
	These are the maps for which the break points have finite forward orbits.
	To arrange for such behavior, one can start with all the defining data of $f$ in $\bQ$.
	Then the denominators of the critical points will stay bounded, so either the orbits escape to infinity at a definite rate, or are finite.
\end{enumerate}







\bibliographystyle{sfilip}
\bibliography{tropical_dynamics}
\end{document}